\documentclass{amsart}
\usepackage{amssymb,latexsym,amsmath,amscd,graphicx,graphics,epic,eepic,bm,color,array,mathrsfs,fullpage}
\usepackage{enumerate}
\usepackage[all,knot,poly]{xy}

\newcommand{\zed}{\mathbb{Z}}

\newcommand{\C}{\mathbb{C}}
\newcommand{\fil}{\mathcal{F}}

\newcommand{\im}{\mathrm{Im}}

\newcommand{\ve}{\varepsilon}

\newcommand{\id}{\mathrm{id}}

\newcommand{\tw}{\mathrm{tw}}
\newcommand{\rank}{\mathrm{rank}}
\newcommand{\gdim}{\mathrm{gdim}}

\newcommand{\slmf}{\mathfrak{sl}}

\theoremstyle{plain}
\newtheorem{theorem}{Theorem}[section]
\newtheorem{lemma}[theorem]{Lemma}

\newtheorem{corollary}[theorem]{Corollary}

\newtheorem{question}[theorem]{Question}

\theoremstyle{definition}
\newtheorem{definition}[theorem]{Definition}

\newtheorem{acknowledgments}{Acknowledgments\ignorespaces}

\theoremstyle{remark}
\newtheorem{remark}[theorem]{Remark}

\numberwithin{equation}{section}

\begin{document}

\title{Equivariant Khovanov-Rozansky Homology and Lee-Gornik Spectral Sequence}

\author{Hao Wu}

\thanks{The author was partially supported by NSF grant DMS-1205879 and a Collaboration Grant for Mathematicians from the Simons Foundation.}

\address{Department of Mathematics, The George Washington University, Monroe Hall, Room 240, 2115 G Street, NW, Washington DC 20052, USA. Telephone: 1-202-994-0653, Fax: 1-202-994-6760}

\email{haowu@gwu.edu}

\subjclass[2010]{Primary 57M27}

\keywords{Equivariant Khovanov-Rozansky homology, homological thickness, Lee-Gornik spectral sequence, exact couple} 

\begin{abstract}
Lobb observed in \cite{Lobb-2-twists} that each equivariant $\slmf(N)$ Khovanov-Rozansky homology over $\C[a]$ admits a standard decomposition of a simple form. 

In the present paper, we derive a formula for the corresponding Lee-Gornik spectral sequence in terms of this decomposition. Based on this formula, we give a simple alternative definition of the Lee-Gornik spectral sequence using exact couples. We also demonstrate that an equivariant $\slmf(N)$ Khovanov-Rozansky homology over $\C[a]$ can be recovered from the corresponding Lee-Gornik spectral sequence via this formula. Therefore, these two algebraic invariants are equivalent and contain the same information about the link.  

As a byproduct of the exact couple construction, we generalize Lee's endomorphism on the rational Khovanov homology to a natural $\bigwedge^\ast \C^{N-1}$-action on the $\slmf(N)$ Khovanov-Rozansky homology.

A numerical link invariant called torsion width comes up naturally in our work. It determines when the corresponding Lee-Gornik spectral sequence collapses and is bounded from above by the homological thickness of the $\slmf(N)$ Khovanov-Rozansky homology. We use the torsion width to explain why the Lee spectral sequences of certain H-thick links collapse so fast.
\end{abstract}

\maketitle

\section{Introduction}\label{sec-intro}

Our goal is to understand the equivariant $\slmf(N)$ Khovanov-Rozansky homology defined by Krasner in \cite{Krasner} and its relations to other versions of the Khovanov-Rozansky homology. Since the algebra is much easier over a principal ideal domain, we focus on equivariant $\slmf(N)$ Khovanov-Rozansky homologies over $\C[a]$. 

Lobb observed in \cite{Lobb-2-twists} that each such homology admits a standard decomposition of a simple form. The first result of the present paper is a decomposition formula for the corresponding Lee-Gornik spectral sequence in terms of Lobb's decomposition. Based on this formula, we define a simple exact couple whose spectral sequence is isomorphic to the corresponding Lee-Gornik spectral sequence minus some repeated pages. 

We also explain how to recover the $\zed^{\oplus2}$-graded $\C[a]$-module structure of the equivariant $\slmf(N)$ Khovanov-Rozansky homology from the corresponding Lee-Gornik spectral sequence. Therefore, an equivariant $\slmf(N)$ Khovanov-Rozansky homology over $\C[a]$ and the corresponding Lee-Gornik spectral sequence determine each other and encode the same information of the link. When recovering the equivariant $\slmf(N)$ Khovanov-Rozansky homology, a numerical link invariant, the torsion width, shows up naturally. It determines exactly when the Lee-Gornik spectral sequence collapses and is bounded from above by the homological thickness of the $\slmf(N)$ Khovanov-Rozansky homology. It also allows us to explain the fast collapsing of the Lee spectral sequences of certain H-thick links.

The aforementioned exact couples equip the $\slmf(N)$ Khovanov-Rozansky homology with extra differentials. Using these differentials, we define a natural $\bigwedge^\ast \C^{N-1}$-action on the $\slmf(N)$ Khovanov-Rozansky homology, which generalizes Lee's endomorphism $\Phi$ on the rational Khovanov homology defined in \cite[Section 4]{Lee2}. In the process of this construction, we prove the non-existence of ``small" torsion components in certain equivariant $\slmf(N)$ Khovanov-Rozansky homologies over $\C[a]$. 

In the remainder of this section, we briefly review the background of this work and state our results. All links and link cobordisms in this paper are oriented.

\subsection{Equivariant $\slmf(N)$ Khovanov-Rozansky homology over $\C[a]$} Following the grading convention in \cite{KR1}, let $x$ be a homogeneous variable of degree $2$ and $a$ a homogeneous variable of degree $2k$, where $k$ is a positive integer. We consider the following homogeneous polynomial of degree $2(N+1)$ in $\C[x,a]$.
\begin{equation}\label{def-P}
P(x,a)=x^{N+1} + \sum_{j=1}^{\left\lfloor \frac{N}{k} \right\rfloor} \lambda_j a^j x^{N+1-jk}, 
\end{equation}
where $\lambda_1,\dots, \lambda_{\left\lfloor \frac{N}{k} \right\rfloor} \in \C$.

For any oriented link diagram $D$, one can use $P(x,a)$ to specialize Krasner's construction in \cite{Krasner} to give a bounded chain complex $C_P(D)$ of graded matrix factorizations over $\C[a]$. We will review the construction of $C_P(D)$ in more details in Section \ref{sec-homology-def}. For now, recall that $C_P(D)$ comes with: 
\begin{itemize}
	\item Two $\zed$-gradings: the homological grading and the polynomial grading;
	\item A filtration: the $x$-filtration $\fil_x$;
	\item Two differential maps: $d_{mf}$ from the underlying matrix factorizations and $d_\chi$ from crossing information.
\end{itemize}
The homology $H(C_P(D), d_{mf})$ is a finitely generated free $\C[a]$-module that inherits both $\zed$-gradings and the $x$-filtration. The equivariant $\slmf(N)$ Khovanov-Rozansky homology of $D$ over $\C[a]$ with potential $P(x,a)$ is defined to be the homology
\begin{equation}\label{def-H_P}
H_P(D) = H(H(C_P(D), d_{mf}), d_\chi),
\end{equation}
which, again, inherits both $\zed$-gradings and the $x$-filtration. 

As a special case of Krasner's work in \cite{Krasner}, we have the following theorem.

\begin{theorem}\cite{Krasner}\label{thm-H_P-invariance}
Every Reidemeister move of $D$ induces a homotopy equivalence of $C_P(D)$ that preserves both $\zed$-gradings and the $x$-filtration. Consequently, $H_P(D)$, with its two $\zed$-gradings and $x$-filtration, is invariant under Reidemeister moves.
\end{theorem} 

\begin{remark}
Strictly speaking, Krasner only proved the invariance under braid-like Reidemeister moves. But the proof of the invariance under Reidemeister move II$_b$ is very similar and given in \cite[Theorem 8.2]{Wu-color-equi}. 

Also, the $x$-filtration is not mentioned in \cite{Krasner}. But, from its definition in Section \ref{sec-homology-def} below, one can see that the homotopy equivalence associated to Reidemeister moves given in \cite{Krasner,Wu-color-equi} preserve the $x$-filtration.
\end{remark}

Define $C_N(D) = C_P(D)/aC_P(D)$. Then $C_N(D)$ is isomorphic to the $\slmf(N)$ Khovanov-Rozansky chain complex in \cite{KR1}. It inherits from $C_P(D)$: 
\begin{itemize}
	\item the homological grading and the polynomial grading\footnote{The increasing filtration induced by this polynomial grading is the same as the $x$-filtration that $C_N(D)$ inherits from $C_P(D)$.},
	\item both differential maps: $d_{mf}$ and $d_\chi$.
\end{itemize}
The homology
\begin{equation}\label{def-H_N}
H_N(D) = H(H(C_N(D), d_{mf}), d_\chi),
\end{equation}
is the $\slmf(N)$ Khovanov-Rozansky homology defined in \cite{KR1}. The invariance of $H_N(D)$ was established by Khovanov and Rozansky in \cite{KR1} but can now be viewed as a corollary of Theorem \ref{thm-H_P-invariance}. 

\begin{corollary}\cite{KR1}\label{cor-H_N-invariance}
Every Reidemeister move of $D$ induces a homotopy equivalence of $C_N(D)$ that preserves both the homological grading and the polynomial grading. Consequently, $H_N(D)$, with its homological grading and polynomial grading, is invariant under Reidemeister moves.
\end{corollary}

\begin{proof}
The standard quotient map $C_P(D) \rightarrow C_N(D)$ preserves homotopy equivalence. So Corollary \ref{cor-H_N-invariance} follows from Theorem \ref{thm-H_P-invariance}.
\end{proof}

Define $\hat{C}_P(D) = C_P(D)/(a-1)C_P(D)$. Then $\hat{C}_P(D)$ is a bounded chain complex of filtered matrix factorizations over $\C$. It inherits from $C_P(D)$: 
\begin{itemize}
	\item the homological grading,
	\item the $x$-filtration $\fil_x$,
	\item both differential maps: $d_{mf}$ and $d_\chi$.
\end{itemize}
We call the homology
\begin{equation}\label{def-hat-H_P}
\hat{H}_P(D) = H(H(\hat{C}_P(D), d_{mf}), d_\chi),
\end{equation}
the deformed $\slmf(N)$ Khovanov-Rozansky homology with potential $P(x,1)$. This version of the Khovanov-Rozansky homology was originally introduced by Lee \cite{Lee2} in the $\slmf(2)$ case and then by Gornik \cite{Gornik} in the general $\slmf(N)$ case. Its invariance was first established by the author in \cite{Wu7} but can now be viewed as a corollary of Theorem \ref{thm-H_P-invariance}. 

\begin{corollary}\cite{Wu7}\label{cor-hat-H_P-invariance}
Every Reidemeister move of $D$ induces a homotopy equivalence of $\hat{C}_P(D)$ that preserves both the homological grading and the $x$-filtration. Consequently, $\hat{H}_P(D)$, with its homological grading and $x$-filtration, is invariant under Reidemeister moves.
\end{corollary}

\begin{proof}
The standard quotient map $C_P(D) \rightarrow \hat{C}_P(D)$ preserves homotopy equivalence. So Corollary \ref{cor-hat-H_P-invariance} follows from Theorem \ref{thm-H_P-invariance}.
\end{proof}

\subsection{The Lee-Gornik spectral sequences} 
\begin{theorem}\cite{Gornik,Lee2}\label{thm-spectral-sequence}
Let $D$ be a diagram of an oriented link $L$. Then
\begin{itemize}
	\item The $x$-filtration $\fil_x$ on the chain complex $(H(\hat{C}_P(D), d_{mf}), d_\chi)$ induces a spectral sequence $\{\hat{E}_r(L)\}$ converging to $\hat{H}_P(L)$ with $\hat{E}_1(L) \cong H_N(L)$,
	\item The $x$-filtration $\fil_x$ on the chain complex $(H(C_P(D), d_{mf}), d_\chi)$ induces a spectral sequence $\{E_r(L)\}$ converging to $H_P(L)$ with $E_1(L) \cong H_N(L)\otimes_\C \C[a]$.
\end{itemize}
\end{theorem}

\begin{remark}
Only the $E_0$-pages of $\{\hat{E}_r(L)\}$ and $\{E_r(L)\}$ depend on the choice of the diagram $D$. By Theorem \ref{thm-H_P-invariance} and Corollary \ref{cor-hat-H_P-invariance}, for $r\geq 1$, $\hat{E}_r(L)$ and $E_r(L)$ are link invariants.
\end{remark}

The spectral sequence $\{\hat{E}_r(L)\}$ was first observed by Lee \cite{Lee2} in the $\slmf(2)$ case and then generalized to the $\slmf(N)$ case by Gornik \cite{Gornik}. A complete construction of $\{\hat{E}_r(L)\}$ can be found in \cite{Wu7}. The construction of $\{E_r(L)\}$ is very similar and given in Section \ref{sec-spectral-sequence} below\footnote{In fact, we construct a somewhat more general spectral sequence. See Theorem \ref{thm-spectral-sequence-general} below.}. We call $\{\hat{E}_r(L)\}$ the Lee-Gornik spectral sequence over $\C$ and $\{E_r(L)\}$ the Lee-Gornik spectral sequence over $\C[a]$. 

\subsection{Lobb's decomposition theorem}  As shown in Section \ref{sec-spectral-sequence} below, $(H(C_P(D), d_{mf}), d_\chi)$ is a bounded chain complex of finitely generated graded free $\C[a]$-module. Lobb \cite{Lobb-2-twists} observed that this implies \linebreak $(H(C_P(D), d_{mf}), d_\chi)$ decomposes into a direct sum of simple graded chain complexes of the forms 
\begin{eqnarray}
\label{def-F-i-s} F_{i,s} & = & 0\rightarrow \C[a]\|i\|\{s\} \rightarrow 0, \\
\label{def-T-i-m-s} T_{i,m,s} & = & 0\rightarrow \C[a]\|i-1\|\{s+2km\} \xrightarrow{a^m} \C[a]\|i\|\{s\} \rightarrow 0,
\end{eqnarray}
where $\|i\|$ indicates that the component is at homological degree $i$ and, following \cite{KR1}, $\{s\}$ means shifting the polynomial grading up by $s$. Therefore, $H_P(D)$ is the direct sum of a free graded $\C[a]$-module and torsion components of the form $\C[a]/(a^m)$. The torsion part of $H_P(D)$ is not yet well understood. But the free part of $H_P(D)$ is relatively simple and can be explicitly described using the deformed $\slmf(N)$ Khovanov-Rozansky homology $\hat{H}_P(D)$. Theorem \ref{thm-H_P-decomp} below is a more precise formulation of the decomposition of $H_P(L)$ observed by Lobb in \cite{Lobb-2-twists}.

For any oriented link $L$, denote by $\hat{H}^i_P(L)$ the component of $\hat{H}_P(L)$ of homological grading $i$ and by $\hat{\mathcal{H}}^i_P(L)$ the graded $\C$-linear space associated to the filtered space $(\hat{H}^i_P(L),\fil_x)$. That is, $\hat{\mathcal{H}}^i_P(L) = \bigoplus_{j\in \zed} \hat{\mathcal{H}}^{i,j}_P(L)$, where $\hat{\mathcal{H}}^{i,j}_P(L) = \fil_x^j \hat{H}^i_P(L) / \fil_x^{j-1} \hat{H}^i_P(L)$.

\begin{theorem}\cite{Lobb-2-twists}\label{thm-H_P-decomp}
Given an oriented link $L$ and a homological degree $i$, there is a (possibly empty) finite sequence 
\[
\{(m_{i,1},s_{i,1}),\dots, (m_{i,n_i},s_{i,n_i})\} \subset \zed_{>0} \times \zed
\] 
such that, as graded $\C[a]$-modules, 
\begin{equation}\label{eq-H_P-decomp}
H^i_P(L) \cong (\hat{\mathcal{H}}^i_P(L) \otimes_{\C} \C[a]) \oplus \bigoplus_{l=1}^{n_i} ( \C[a]/(a^{m_{i,l}}))\{s_{i,l}\},
\end{equation}
where $H^i_P(L)$ is component of $H_P(L)$ of homological grading $i$. Moreover, the sequence $$\{(m_{i,1},s_{i,1}),\dots, (m_{i,n_i},s_{i,n_i})\}$$ is unique up to permutation.
\end{theorem}

A complete proof of Theorem \ref{thm-H_P-decomp} is given in Subsection \ref{subsec-homology-decomp} below. A byproduct of this theorem is a decomposition of $H_N(L)$, which we formulate in the following corollary. See Subsection \ref{subsec-homology-decomp} below for its proof.

\begin{corollary}\label{cor-H_N-decomp}
Using notations in Theorem \ref{eq-H_P-decomp}, we have
\begin{equation}\label{eq-H_N-decomp}
H^i_N(L) \cong \hat{\mathcal{H}}^i_P(L) \oplus \left(\bigoplus_{l=1}^{n_i}\C\{s_{i,l}\}\right) \oplus \left(\bigoplus_{l=1}^{n_{i+1}} \C\{2km_{i+1,l}+s_{i+1,l}\}\right),
\end{equation}
where $H^i_N(L)$ is component of $H_N(L)$ of homological grading $i$.
\end{corollary}

\subsection{Decompositions of the Lee-Gornik spectral sequences} The first results of the present paper are formulas for $\{E_r(L)\}$ and $\{\hat{E}_r(L)\}$ in terms of decomposition \eqref{eq-H_P-decomp}. To state our results, we need to introduce a non-standard tensor product ``$\boxtimes$" of bigraded vector spaces.\footnote{The definition of ``$\boxtimes$" in Definition \ref{def-box-times} comes from the normalization we use in the definition of the spectral sequence of a filtered chain complex. If one uses a different normalization, then the definition of ``$\boxtimes$" needs to change accordingly.}

\begin{definition}\label{def-box-times}
Let $\mathcal{H} = \bigoplus_{i,j} \mathcal{H}^{i,j}$ and $E = \bigoplus_{p,q} E^{p,q}$ be two $\zed^{\oplus 2}$-graded $\C$-spaces. Then 
\[
\mathcal{H} \boxtimes E = \bigoplus_{\alpha,\beta} (\mathcal{H} \boxtimes E)^{\alpha,\beta}
\] 
is the $\zed^{\oplus 2}$-graded $\C$-space satisfying 
\[
(\mathcal{H} \boxtimes E)^{\alpha,\beta} = \bigoplus_{j+p=\alpha,~q+i-j=\beta}  \mathcal{H}^{i,j} \otimes_\C E^{p,q}.
\]
\end{definition}

Next, we define the $x$-filtration $\fil_x$ of $F_{i,s}$, $T_{i,m,s}$, $\hat{F}_{i,s} = F_{i,s}/(a-1)F_{i,s}$ and $\hat{T}_{i,m,s} = T_{i,m,s}/(a-1)T_{i,m,s}$. 
\begin{eqnarray}
\label{def-F-i-s-fil}
\fil_x^p F_{i,s} & = & \begin{cases}
0\rightarrow \C[a]\|i\|\{s\} \rightarrow 0 & \text{if } p\geq s, \\
0 & \text{if } p< s.
\end{cases} \\
\label{def-T-i-m-s-fil}
\fil_x^p T_{i,m,s} & = & \begin{cases}
0\rightarrow \C[a]\|i-1\|\{s+2km\} \xrightarrow{a^m} \C[a]\|i\|\{s\} \rightarrow 0 & \text{if } p\geq s+2km \\
0\rightarrow \C[a]\|i\|\{s\} \rightarrow 0 & \text{if } s \leq p< s+2km \\
0 & \text{if }  p<s.
\end{cases} \\
\label{def-hat-F-i-s-fil}
\fil_x^p \hat{F}_{i,s} & = & \begin{cases}
0\rightarrow \C\|i\|\rightarrow 0 & \text{if } p\geq s, \\
0 & \text{if } p< s,
\end{cases} \\
\label{def-hat-T-i-m-s-fil}
\fil_x^p \hat{T}_{i,m,s} & = & \begin{cases}
0\rightarrow \C\|i-1\| \xrightarrow{1} \C\|i\| \rightarrow 0 & \text{if } p\geq s+2km \\
0\rightarrow \C\|i\|\rightarrow 0 & \text{if } s \leq p< s+2km \\
0 & \text{if }  p<s,
\end{cases}
\end{eqnarray}

The filtered chain complexes $F_{i,s}$, $T_{i,m,s}$, $\hat{F}_{i,s}$ and $\hat{T}_{i,m,s}$ are very simple. Their spectral sequences are given in the following lemma, which is proved in Subsection \ref{subsec-spectral-sequence-decomp} below.

\begin{lemma}\label{lemma-ss-components}
For any $r\geq 0$,
\begin{eqnarray}
\label{eq-F-i-s-ss} E_r^{p,q}(F_{i,s}) & \cong & \begin{cases}
\C[a]\{s\} & \text{if } p=s \text{ and } q=i-s, \\
0 & \text{otherwise,}
\end{cases} \\
\label{eq-T-i-m-s-ss}E_r^{p,q}(T_{i,m,s})  & \cong & \begin{cases}
(\C[a]/(a^m))\{s\} & \text{if } p=s, ~q=i-s \text{ and } r \geq 2km+1, \\
\C[a]\{s\} & \text{if } p=s, ~q=i-s \text{ and } r \leq 2km, \\
\C[a]\{s+2km\} & \text{if } p= s+2km,~q=i-1-s-2km \text{ and } r \leq 2km, \\
0 & \text{otherwise,}
\end{cases} \\
\label{eq-hat-F-i-s-ss}E_r^{p,q}(\hat{F}_{i,s}) & \cong & \begin{cases}
\C & \text{if } p=s \text{ and } q=i-s, \\
0 & \text{otherwise.}
\end{cases} \\
\label{eq-hat-T-i-m-s-ss}E_r^{p,q}(\hat{T}_{i,m,s})  & \cong & \begin{cases}
\C & \text{if } p=s, ~q=i-s \text{ and } r \leq 2km, \\
\C & \text{if } p= s+2km,~q=i-1-s-2km \text{ and } r \leq 2km, \\
0 & \text{otherwise.}
\end{cases} 
\end{eqnarray}
Note that:
\begin{itemize}
  \item Isomorphisms \eqref{eq-F-i-s-ss} and \eqref{eq-T-i-m-s-ss} preserve the polynomial grading.
	\item Both $\{E_r(T_{i,m,s})\}$ and $\{E_r(\hat{T}_{i,m,s})\}$ collapse exactly at their $E_{2km+1}$-pages.\footnote{We say that a spectral sequence $\{E_r\}$ collapses exactly at its $E_t$-page if $E_{t-1} \ncong E_t$ but $E_{t+r} \cong E_t$ $\forall r\geq0$.}
	\item Both $\{E_r(F_{i,s})\}$ and $\{E_r(\hat{F}_{i,s})\}$ collapse at their $E_0$-pages.
	\item $E_r(F_{i,s}) \cong \C\|i\|\{s\} \boxtimes E_r(F_{0,0})$, $E_r(\hat{F}_{i,s}) \cong \C\|i\|\{s\} \boxtimes E_r(\hat{F}_{0,0})$, where ``$\boxtimes$" is the product defined in Definition \ref{def-box-times} and $\C\|i\|\{s\}$ is the $\zed^{\oplus2}$-graded $\C$-space given by
\[
(\C\|i\|\{s\})^{p,q} = \begin{cases}
\C & \text{if } p=i \text{ and } q=s, \\
0 & \text{otherwise.}
\end{cases} 
\]
\end{itemize}
\end{lemma}

Combining Lemma \ref{lemma-ss-components} and the following theorem, we get explicit formulas for $\{E_r(L)\}$ and $\{\hat{E}_r(L)\}$ in terms of Lobb's decomposition (Theorem \ref{thm-H_P-decomp}.) 

\begin{theorem}\label{thm-spectral-sequence-decomp}
For an oriented link $L$, let $\hat{\mathcal{H}}_P(L)= \bigoplus_{i\in \zed} \hat{\mathcal{H}}^i_P(L) = \bigoplus_{(i,j)\in \zed^{\oplus 2}} \hat{\mathcal{H}}^{i,j}_P(L)$ and, for each $i$, 
\[
\{(m_{i,1},s_{i,1}),\dots, (m_{i,n_i},s_{i,n_i})\} \subset \zed_{>0} \times \zed
\] 
the sequence given in Theorem \ref{thm-H_P-decomp}. Then, for any $r\geq 1$,
\begin{eqnarray}
\label{eq-spectral-sequence-decomp-hat-E}  \hat{E}_r(L) & \cong & (\hat{\mathcal{H}}_P(L) \boxtimes E_r(\hat{F}_{0,0})) \oplus \bigoplus_{i \in \zed} \bigoplus_{l=1}^{n_i} E_r (\hat{T}_{i,m_{i,l},s_{i,l}}), \\
\label{eq-spectral-sequence-decomp-E} E_r(L) & \cong & (\hat{\mathcal{H}}_P(L) \boxtimes E_r(F_{0,0})) \oplus \bigoplus_{i \in \zed} \bigoplus_{l=1}^{n_i} E_r (T_{i,m_{i,l},s_{i,l}}), \end{eqnarray}
where isomorphism \eqref{eq-spectral-sequence-decomp-hat-E} preserves the usually $(p,q)$-grading of spectral sequences, while isomorphism \eqref{eq-spectral-sequence-decomp-E} preserves the usually $(p,q)$-grading of spectral sequences as well as the polynomial grading of each $E_r^{p,q}$-component.
\end{theorem}

Theorem \ref{thm-spectral-sequence-decomp} is proved in Subsection \ref{subsec-spectral-sequence-decomp} below. The key to its proof is that, when the chain complex $(H(C_P(D), d_{mf}), d_\chi)$ is decomposed into complexes of the forms $F_{i,s}$ and $T_{i,m,s}$, the $x$-filtration decomposes accordingly. This is established in Subsection \ref{subsec-fil-decomp}.

\subsection{Lee-Gornik spectral sequence via exact couples}\label{subsec-intro-couples} Let us recall the definition of exact couples of $\zed^{\oplus 2}$-graded $\C$-linear spaces.

\begin{definition}\label{def-exact-couple}
An exact couple of $\zed^{\oplus 2}$-graded $\C$-linear spaces is a tuple $(A,E,f,g,h)$ such that
\begin{itemize}
	\item $A$ and $E$ are $\zed^{\oplus 2}$-graded $\C$-linear spaces,
	\item $A \xrightarrow{f} A$, $A \xrightarrow{g} E$ and $E \xrightarrow{h} A$ are homogeneous homomorphisms of $\zed^{\oplus 2}$-graded $\C$-linear spaces,
	\item the triangle 
	
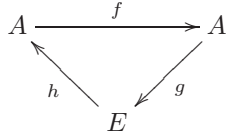
\[
	\xymatrix{
	A \ar[rr]^{f} && A  \ar[ld]^{g} \\
	& E \ar[lu]^{h} &
	}
	\]
	is exact.
\end{itemize}
\end{definition}

Any exact couple $(A,E,f,g,h)$ has a derived couple $(A',E',f',g',h')$, which is itself an exact couple. We will review the definition of the derived couple in Subsection \ref{subsec-couples}. For now, we just point out that $d:=g\circ h$ is a differential on $E$, and $E'$ is defined to be the homology of $(E,d)$. 

Starting with an exact couple $(A^{(1)},E^{(1)},f^{(1)},g^{(1)},h^{(1)})$, one can inductive define a sequence \linebreak $\{(A^{(r)},E^{(r)},f^{(r)},g^{(r)},h^{(r)})\}$ of exact couples, where $(A^{(r)},E^{(r)},f^{(r)},g^{(r)},h^{(r)})$ is the derived couple of \linebreak $(A^{(r-1)},E^{(r-1)},f^{(r-1)},g^{(r-1)},h^{(r-1)})$. Let $d^{(r)} = g^{(r)} \circ h^{(r)}$. Then $\{(E^{(r)},d^{(r)})\}$ is the spectral sequence induced by $(A^{(1)},E^{(1)},f^{(1)},g^{(1)},h^{(1)})$.

Now let $D$ be a link diagram. Recall that $C_N(D) = C_P(D)/aC_P(D)$. Denote by $\pi_a$ the standard quotient map $C_P(D) \rightarrow C_N(D)$, which induces a homomorphism $H(C_P(D),d_{mf}) \xrightarrow{\pi_a} H(C_N(D),d_{mf})$. But $H(C_P(D),d_{mf})$ is a free $\C[a]$-module (see for example Corollary \ref{cor-graph-homology-free} below). So there is a short exact sequence 
\[
0 \rightarrow H(C_P(D),d_{mf}) \xrightarrow{a} H(C_P(D),d_{mf}) \xrightarrow{\pi_a} H(C_N(D),d_{mf}) \rightarrow 0,
\]
which induces an exact couple
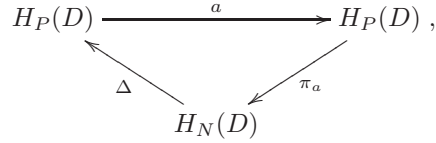
\[
	\xymatrix{
	H_P(D) \ar[rr]^{a} && H_P(D)  \ar[ld]^{\pi_a} \\
	& H_N(D) \ar[lu]^{\Delta} &
	},
\]
where $\Delta$ is the connecting homomorphism from the long exact sequence construction, which is homogeneous with bidegree $(1,-2k)$.

\begin{theorem}\label{thm-exact-couple-lee-gornik}
Denote by $\{(\tilde{E}^{(r)}(D),d^{(r)})\}$ the spectral sequence induced by the exact couple 
\[
(A^{(1)}(D),\tilde{E}^{(1)}(D),f^{(1)},g^{(1)},h^{(1)})=(H_P(D),H_N(D),a,\pi_a,\Delta).
\] 
Then 
\[
\tilde{E}^{(r)}_{p,q}(D) \cong \hat{E}_{2k(r-1)+1}^{q,p-q}(D),
\]
where $\{\hat{E}_{r}(D)\}$ is the Lee-Gornik spectral sequence of $D$ over $\C$ given in Theorem \ref{thm-spectral-sequence}.
\end{theorem}

The proof of Theorem \ref{thm-exact-couple-lee-gornik} in Subsection \ref{subsec-couples} below is straightforward. We simply compute the sequence of derived exact couples for each component in decomposition \eqref{eq-H_P-decomp} and compare it to Lemma \ref{lemma-ss-components}.

\begin{remark}
From Theorem \ref{thm-exact-couple-lee-gornik}, it may seem like $\{\tilde{E}^{(r)}\}$ is missing a lot of pages of $\{\hat{E}_{r}(D)\}$. But, by Lemma \ref{lemma-ss-components} and Theorem \ref{thm-spectral-sequence-decomp}, one can see that $\hat{E}_{2km+1}(L)\cong \hat{E}_{2km+2}(L)\cong \cdots \cong \hat{E}_{2k(m+1)}(L)$ for any link $L$ and any non-negative integer $m$. So the missing pages are just identical copies of pages of $\{\tilde{E}^{(r)}\}$.
\end{remark}

\subsection{A natural $\bigwedge^\ast \C^{N-1}$-action on $H_N(L)$}\label{subsec-intro-action} For a link $L$, we take a closer look at the exact couple $(H_P(L),H_N(L),a,\pi_a,\Delta)$ defined in the previous subsection. It equips $H_N(L)$ with a differential $d^{(1)}=\pi_a \circ \Delta$. Note that all the above construction depends on a particular homogeneous polynomial $P(x,a)$ of form \eqref{def-P}. In this subsection, we temporarily bring $P$ back in the notation of this differential on $H_N(L)$ and write $d_P^{(1)}$ instead of $d^{(1)}$. 

We consider the polynomial 
\begin{equation}\label{def-P-i}
P_i(x,b_i)=x^{N+1} + b_ix^{i}, 
\end{equation}
where $1\leq i \leq N$ and $b_i$ is a homogeneous variable of degree $2N+2-2i$. Applying the exact couple constructed in Theorem \ref{thm-exact-couple-lee-gornik} to $P_i$, we define on $H_N(L)$ a homogeneous differential map $\delta_i :=d_{P_i}^{(1)}$ of homological degree $1$ and polynomial degree $2i-2N-2$. 

We prove that $\delta_1,\dots,\delta_{N-1}$ give a natural $\bigwedge^\ast \C^{N-1}$-action on $H_N(L)$ and this action can not be extended by adding other $d_P^{(1)}$'s. The following is a lemma needed in the construction, which provides some control on how small a torsion component in Lobb's decomposition can be.\footnote{Corollary \ref{cor-ht-bound} provides some control on how large a torsion component in Lobb's decomposition can be.}

\begin{lemma}\label{lemma-torsion-lower-bound}
Let $a$ be a homogeneous variable of degree $2k$, $2\leq m \leq \left\lfloor\frac{N}{k}\right\rfloor$ and
\[
P(x,a) = x^{N+1} + \sum_{i=m}^{\left\lfloor\frac{N}{k}\right\rfloor} \lambda_i a^i x^{N+1-ki},
\]
where $\lambda_m,\dots,\lambda_{\left\lfloor\frac{N}{k}\right\rfloor}$ are scalars. Then, for any link $L$, we have $m_{i,l}\geq m$  $\forall ~i,l$ in decomposition \eqref{eq-H_P-decomp} of $H_P(L)$. That is, $H_P(L)$ does not contain torsion components isomorphic to any of $\C[a]/(a),\dots,\C[a]/(a^{m-1})$.
\end{lemma}

\begin{theorem}\label{thm-delta-action}
Let $L$ be any link. As endomorphisms of $H_N(L)$, 
\begin{enumerate}
	\item $\delta_N=0$;
	\item $\delta_i\delta_j+\delta_j\delta_i = 0$ for any $1\leq i, j \leq N-1$;
	\item each $\delta_i$ is natural in the sense that it commutes with homomorphisms of $H_N(L)$ induced by link cobordisms;
	\item for a polynomial $P(x,a)=x^{N+1} + \sum_{j=1}^{\left\lfloor \frac{N}{k} \right\rfloor} \lambda_j a^j x^{N+1-jk}$ with $\deg a = 2k$ and $\lambda_i \in \C$, 
\[
d_P^{(1)} = \begin{cases}
0 & \text{if } \lambda_1 =0 \text{ or } k=1, \\
\lambda_1 \delta_{N+1-k} & \text{otherwise.}
\end{cases}
\] 
\end{enumerate}

Let $V$ be a $\zed^{\oplus 2}$-graded $(N-1)$-dimensional $\C$-linear space with a homogeneous basis $\{v_1,\dots,v_{N-1}\}$ such that $v_i$ has bidegree $(1,2i-2N-2)$. Then the mapping $v_i\mapsto \delta_i$ induces a natural $\zed^{\oplus 2}$-grading preserving action of $\bigwedge^\ast V$ on $H_N(L)$.
\end{theorem}

\begin{remark}
In the case $N=2$, we get just one differential $\delta_1$ on the rational Khovanov homology $H_2(L)$. This $\delta_1$ is essentially the differential $\Phi$ in \cite[Section 4]{Lee2}.
\end{remark}

\begin{question}
Are there more relations between $\delta_1,\dots,\delta_{N-1}$? That is, does the above $\bigwedge^\ast V$-action factor through a quotient ring of $\bigwedge^\ast V$? 
\end{question}

In Subsection \ref{subsec-example-link} below, we compute $H_{P_i}(L)$ for the closed $2$-braid $L$ in Figure \ref{fig-link-example} and observe that, on $H_N(L)$, the differentials $\delta_1,\dots,\delta_{N-1}$ are non-zero, but $\delta_i\delta_j =0$ for any $1\leq i,j\leq N-1$.

\begin{figure}[ht]
\[
\xymatrix{
 \setlength{\unitlength}{1pt}
\begin{picture}(120,70)(-20,-10)

\qbezier(0,0)(5,0)(10,10)
\qbezier(10,10)(15,20)(20,20)

\qbezier(9,12)(0,30)(0,35)
\qbezier(11,8)(15,0)(20,0)

\qbezier(20,0)(25,0)(30,10)
\qbezier(30,10)(35,20)(40,20)

\qbezier(20,20)(25,20)(29,12)
\qbezier(31,8)(35,0)(40,0)

\qbezier(40,0)(45,0)(50,10)
\qbezier(50,10)(55,20)(60,20)

\qbezier(40,20)(45,20)(49,12)
\qbezier(51,8)(55,0)(60,0)

\qbezier(60,0)(65,0)(70,10)
\qbezier(70,10)(80,30)(80,35)

\qbezier(60,20)(65,20)(69,12)
\qbezier(71,8)(75,0)(80,0)

\qbezier(0,0)(-20,0)(-20,30)
\qbezier(-20,30)(-20,60)(40,60)
\qbezier(40,60)(100,60)(100,30)
\qbezier(100,30)(100,0)(80,0)

\qbezier(0,35)(0,50)(40,50)
\qbezier(80,35)(80,50)(40,50)

\put(40,60){\vector(-1,0){0}}

\put(40,50){\vector(-1,0){0}}

\put(35,-10){$L$}

\end{picture}
}
\]

\caption{An example}\label{fig-link-example}

\end{figure}
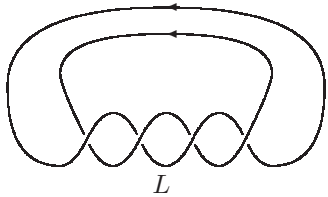

\subsection{The torsion width} In turns out that one can recover the $\zed^{\oplus 2}$-graded module structure of $H_P(L)$ from $\{\hat{E}_r(L)\}$ using Lemma \ref{lemma-ss-components} and Theorem \ref{thm-spectral-sequence-decomp}. We describe an algorithm that does this in Subsection \ref{subsec-recover-H_P} below. Roughly speaking, we look at the pages of $\{\hat{E}_r(L)\}$ backward starting from $\hat{E}_\infty(L)$ to recover first the free part of $H_P(L)$ and then the torsion components from large to small. To do this, we need to know where to start, that is, $\{\hat{E}_r(L)\}$ collapses at what page. For this purpose, we introduce a numerical link invariant called torsion width.

\begin{definition}\label{def-tw}
Let $L$ be an oriented link. Using the notations in Theorem \ref{thm-H_P-decomp}, we define the torsion width of $H_P(L)$ to be $\tw_P(L) = \max\{m_{i,l}~|~ i\in \zed, 1\leq l \leq n_i\}$\footnote{We use the convention that $\tw_P(L) =0$ if $H_P(L)$ is a free $\C[a]$-module.}, which, by Theorem \ref{thm-H_P-decomp}, is a link invariant. Equivalently, one has $\tw_P(L) = \min\{ m ~|~ m\in\zed_{\geq0},~a^m H_P(L) \text{ is free}\}$.
\end{definition}

\begin{corollary}\label{cor-E-collapse}
Let $L$ be an oriented link with torsion width $\tw_P(L) = w$, and $D$ a diagram of $L$. Assume $D$ is not a union of disjoint circles embedded in the plane. Then both spectral sequences $\{E_r(L)\}$ and $\{\hat{E}_r(L)\}$ from Theorem \ref{thm-spectral-sequence} collapse exactly at their $E_{2kw+1}$-pages, where $2k=\deg a$. Consequently, $\{\tilde{E}^{(r)}(L)\}$ collapses exactly at its $E^{(w+1)}$-page. 
\end{corollary}

\begin{proof}
From Lemma \ref{lemma-ss-components}, we know that $\{E_r(F_{i,s})\}$ and $\{E_r(\hat{F}_{i,s})\}$ both collapse exactly at their $E_0$-pages, while $\{E_r(T_{i,m,s})\}$ and $\{E_r(\hat{T}_{i,m,s})\}$ both collapse exactly at their $E_{2km+1}$-pages. So this corollary follows from Theorem \ref{thm-spectral-sequence-decomp}.
\end{proof}

Next we define the thickness of the $\slmf(N)$ Khovanov-Rozansky homology.

\begin{definition}\label{def-ht}
For an oriented link $L$, denote by $H_N^{i,j}(L)$ the component of $H_N(L)$ of homological degree $i$ and polynomial degree $j$. Define the $\slmf(N)$ homological thickness $\mathrm{ht}_N(L)$ and the local $\slmf(N)$ homological thickness $\mathrm{lht}_N(L)$ of $L$ to be
\begin{eqnarray}
\label{eq-def-ht} \mathrm{ht}_N(L) & = & \max\{1+\frac{1}{2}[(2i_1 +j_1)-(2i_2+j_2)]~|~ H_N^{i_1,j_1}(L) \neq 0, ~H_N^{i_2,j_2}(L) \neq 0,\} \\
\label{eq-def-lht} \mathrm{lht}_N(L) & = & \max\{\frac{j_1-j_2}{2}~|~ \exists i \in \zed, \text{ such that } H_N^{i,j_1}(L) \neq 0, ~H_N^{i+1,j_2}(L) \neq 0.\}
\end{eqnarray}
\end{definition}

Of course, $\mathrm{ht}_N(L)$ is a naive generalization of the homological thickness of the rational Khovanov homology. Note that $\mathrm{lht}_N(L)$ is not always defined. For example, $\mathrm{lht}_N(\text{unknot})$ is not defined. Even when $\mathrm{lht}_N(L)$ is defined, it is not clear whether it is always non-negative. But, from their definitions, one can see that $\mathrm{lht}_N(L) \leq \mathrm{ht}_N(L)$ if $\mathrm{lht}_N(L)$ is defined. See \cite[Figure 24 and Table 2]{Elliott} for a knot $K_1$ satisfying $\mathrm{lht}_2(K_1) =3< \mathrm{ht}_2(K_1)=4$.

\begin{corollary}\label{cor-ht-bound} 
$k\cdot \tw_P(L) \leq \mathrm{ht}_N(L)$ and, if $H_N(L) \ncong \hat{\mathcal{H}}_P(L) := \bigoplus_{i\in \zed} \hat{\mathcal{H}}^i_P(L)$, then $\mathrm{lht}_N(L)$ is defined and $k\cdot\tw_P(L) \leq \mathrm{lht}_N(L)$, where $2k=\deg a$.
\end{corollary}

\begin{proof}
By Corollary \ref{cor-H_N-decomp}, each torsion component $(\C[a]/(a^m))\|i\|\{s\}$ of $H_P(L)$ generates a pair of $1$-dimensional components of $H_N(L)$: $\C\|i\|\{s\}$ and $\C\|i-1\|\{2km+s\}$. Corollary \ref{cor-ht-bound} follows from this observation.
\end{proof}

\subsection{Recovering $H_P(L)$ from $\{\tilde{E}^{(r)}(L)\}$}\label{subsec-recover-H_P} In this subsection, we give an algorithm to recover the $\zed^{\oplus2}$-graded $\C[a]$-module structure of $H_P(L)$ from the $\zed^{\oplus2}$-graded $\C$-linear space structure on pages of the Lee-Gornik spectral sequence over $\C$. We write down the algorithm in terms of $\{\tilde{E}^{(r)}(L)\}$ to have slightly simpler notations.

From page $\tilde{E}^{(1)}(L) = H_N(L)$, one can find the $\slmf(N)$ homological thickness $\mathrm{ht}_N(L)$ of $L$. By Corollary \ref{cor-ht-bound}, we know that $\tw_P(L) \leq \tau:= \left\lfloor \frac{\mathrm{ht}_N(L)}{k}\right\rfloor$. So $a^\tau H_P(L)$ is a free $\C[a]$-module and, by Corollary \ref{cor-E-collapse}, $\{\tilde{E}^{(r)}(L)\}$ collapses at or before the page $\tilde{E}^{(\tau+1)}(L)$.

Now consider the pages $\{\tilde{E}^{(r)}(L) ~|~ 1\leq r \leq \tau+1\}$. By Lemma \ref{lemma-ss-components} and Theorems \ref{thm-spectral-sequence-decomp}, \ref{thm-exact-couple-lee-gornik}, we observe the following.
\begin{enumerate}[1.]
	\item Start with $\tilde{E}^{(\tau+1)}(L) \cong \tilde{E}^{(\infty)}(L)$. Note that:
	\begin{itemize}
		\item Each free component $\C[a]\|i\|\{s\}$ of $H_P(L)$ contributes a $1$-dimensional component $\C\|i\|\{s\}$ to $\tilde{E}^{(\infty)}(L)$.
	  \item Torsion components of $H_P(L)$ contribute nothing to $\tilde{E}^{(\infty)}(L)$.
  \end{itemize}
        So we can recover all the generators of the free part of $H_P(L)$ from $\tilde{E}^{(\tau+1)}(L)$.
	\item Next look at $\tilde{E}^{(\tau)}(L)$. 
	\begin{itemize}
	  \item Each free component $\C[a]\|i\|\{s\}$ of $H_P(L)$ contributes a $1$-dimensional component $\C\|i\|\{s\}$ to $\tilde{E}^{(\tau)}(L)$.
	  \item Each component $\C[a]/(a^\tau)\|i\|\{s\}$ of of $H_P(L)$ contributes a component $\C\|i\|\{s\} \oplus \C\|i-1\|\{2k\tau+s\}$ to $\tilde{E}^{(\tau)}(L)$.
	  \item For $m<\tau$, a component $\C[a]/(a^m)\|i\|\{s\}$ of of $H_P(L)$ contributes nothing to $\tilde{E}^{(\tau)}(L)$.
  \end{itemize}
        Since we know all the generators of the free part of $H_P(L)$ from the previous step, we can recover all generators of torsion components of $H_P(L)$ of the form $\C[a]/(a^\tau)\|i\|\{s\}$.
  \item For any $1\leq r<\tau$, assume we have recovered all generators of free components and torsion components of the form $\C[a]/(a^m)\|i\|\{s\}$ of $H_P(L)$, where   $r+1\leq m \leq \tau$. Look at the page $\tilde{E}^{(r)}(L)$.
	\begin{itemize}
	  \item Each free component $\C[a]\|i\|\{s\}$ of $H_P(L)$ contributes a $1$-dimensional component $\C\|i\|\{s\}$ to $\tilde{E}^{(r)}(L)$.
	  \item If $m\geq r$, each component $\C[a]/(a^m)\|i\|\{s\}$ of of $H_P(L)$ contributes $\C\|i\|\{s\} \oplus \C\|i-1\|\{2km+s\}$ to $\tilde{E}^{(r)}(L)$.
	  \item For $m<r$, a component $\C[a]/(a^m)\|i\|\{s\}$ of of $H_P(L)$ contributes nothing to $\tilde{E}^{(r)}(L)$.
  \end{itemize}
        So we can recover all generators of torsion components of $H_P(L)$ of the form $\C[a]/(a^r)\|i\|\{s\}$.
\end{enumerate}

The above algorithm allows us to inductively recover the $\zed^{\oplus2}$-graded $\C[a]$-module structure of $H_P(L)$ from $\{\tilde{E}^{(r)}(L)\}$. In particular, we have proved the following theorem.

\begin{theorem}\label{thm-H_P-E-equivalent}
$H_P(L)$ and $\{\tilde{E}^{(r)}(L)\}$ (or, equivalently $\{\hat{E}_r(L)\}$) determine each other and encode the same information of the link $L$.
\end{theorem}

\begin{remark}
With minor language changes, all the above theorems their proofs generalize to the colored $\slmf(N)$ link homology defined in \cite{Wu-color,Wu-color-equi,Wu-color-ras}.
\end{remark}

\subsection{Fast collapsing of the Lee spectral sequence and other observations} As we have seen, each torsion component $(\C[a]/(a^{m}))\|i\|\{s\}$ of $H_P(L)$ in Lobb's decomposition contributes a $2$-dimensional direct sum component $\C\|i\|\{s\} \oplus \C\|i-1\|\{2km+s\}$ to $H_N(L)$. Based on this, we make several observations. 

First, the pairing of $\C\|i\|\{s\}$ and $\C\|i-1\|\{2km+s\}$ is a generalization of \cite[Theorem 1.4]{Lee2}, which states that, except those with homological degree $0$, all homogeneous generators of the rational Khovanov homology of an alternating knot appear in pairs of bi-degree difference $(-1,4)$. Here, we use the torsion width to slightly generalize this theorem. 

\begin{corollary}\label{cor-sl-2}
Suppose $N=2$, $\deg a =4$ and $P(x,a)=x^3-ax$. Assume that $\mathrm{lht}_2(L) \leq 3$ for a link $L$. Then $\tw_P(L) \leq 1$. Consequently, $\{\tilde{E}^{(r)}(L)\}$ collapses at its $E^{(1)}$- or $E^{(2)}$-page. Moreover, there exists a (possibly empty) sequence of pairs of integers $\{(i_1,s_1),\dots,(i_n,s_n)\}$ such that 
\[
H_2(L) \cong \hat{\mathcal{H}}_P(L) \oplus \bigoplus_{l=1}^n (\C\|i_l\|\{s_l\} \oplus \C\|i_l-1\|\{s_l+4\}).
\] 
\end{corollary}

\begin{proof}
If $H_P(L)$ is a free $\C[a]$-module, then $\tw_P(L) =0$. So $\{\tilde{E}^{(r)}(L)\}$ collapses at its $E^{(1)}$-page and $H_2(L) \cong \hat{\mathcal{H}}_P(L)$.

Now assume $H_P(L)$ has torsions. Then $\tw_P(L) \geq 1$ and, by Corollary \ref{cor-H_N-decomp}, $H_2(L) \ncong \hat{\mathcal{H}}_P(L)$. Thus, by Corollary \ref{cor-ht-bound}, we have $2\tw_P(L) \leq \mathrm{lht}_2(L)\leq 3$. So $\tw_P(L) \leq 1$. This shows that, in this case, $\tw_P(L) = 1$ and, therefore, $\{\tilde{E}^{(r)}(L)\}$ collapses at its $E^{(2)}$-page by Corollary \ref{cor-E-collapse}. The decomposition of $H_2(L)$ follows from Corollary \ref{cor-H_N-decomp}.
\end{proof}

\begin{remark}
In \cite{Shu-Colloquium}, Shumakovitch observed that, in all the examples he knew, the Lee spectral sequence collapses at its $E^{(2)}$-page, even for H-thick links. Corollary \ref{cor-sl-2} explains why the Lee spectral sequences of some H-thick links collapse so fast. 

For example, consider the H-thick knot $K_1$ in \cite[Figure 24]{Elliott}. From \cite[Table 2]{Elliott}, one can see that $\mathrm{lht}_2(K_1)= 3$ and $\mathrm{ht}_2(K_1)= 4$. By Corollary \ref{cor-sl-2}, $\{\tilde{E}^{(r)}(K_1)\}$ collapses at its $E^{(2)}$-page.
\end{remark}

Next, we look at the two ends of the $\slmf(N)$ Khovanov-Rozansky homology of a knot.

\begin{corollary}\label{cor-H_N-ends}
Fix a positive integer $N$. For a knot $K$, define $h_{\mathrm{min}} = \min \{i~|~ H_N^i(K) \neq 0\}$ and $h_{\mathrm{max}} = \max \{i~|~ H_N^i(K) \neq 0\}$. Moreover, for a fixed $i$, define $g^i_{\mathrm{min}}= \min \{j~|~ H_N^{i,j}(K) \neq 0\}$ and $g^i_{\mathrm{max}}= \max \{j~|~ H_N^{i,j}(K) \neq 0\}$, where $H_N^{i,j}(K)$ is the component of $H_N^{i}(K)$ of polynomial grading $j$. 
\begin{enumerate}
	\item If $h_{\mathrm{min}}<0$, then $\dim_\C  H_N^{h_{\mathrm{min}}}(K) \leq \dim_\C  H_N^{h_{\mathrm{min}}+1}(K)$ and $g^{h_{\mathrm{min}}}_{\mathrm{min}}>g^{h_{\mathrm{min}}+1}_{\mathrm{min}}$.
	\item If $h_{\mathrm{max}}>0$, then $\dim_\C  H_N^{h_{\mathrm{max}}}(K) \leq \dim_\C  H_N^{h_{\mathrm{max}}-1}(K)$ and $g^{h_{\mathrm{max}}}_{\mathrm{max}}<g^{h_{\mathrm{max}}-1}_{\mathrm{max}}$.
	\item If $h_{\mathrm{max}}=h_{\mathrm{min}}=0$, then $H_N(K) \cong \hat{\mathcal{H}}_P(K)$ for any $P=P(x,a)$ of form \eqref{def-P}.
\end{enumerate}
\end{corollary}

\begin{proof}
For Part (1), consider the polynomial $P_1(x,b_1)=x^{N+1}+b_1 x$. By \cite[Theorem 2]{Gornik}, $\hat{H}_{P_1}(K)$ is supported on homological degree $0$. Therefore, the free part of $H_{P_1}(K)$ is supported on homological degree $0$. Since $h_{\mathrm{min}}<0$, $H_N^{h_{\mathrm{min}}}(K)$ comes entirely from torsion components of $H_{P_1}(K)$ at homological degree $h_{\mathrm{min}}+1$. Part (1) follows from this observation.

The proof of Part (2) is very similar and left to the reader. 

For part (3), note that, if $H_P(K)$ has torsion components, then $H_N(L)$ should occupy at least two homological degrees. But $h_{\mathrm{max}}=h_{\mathrm{min}}=0$. So $H_P(K)$ is free. Then Part (3) follows from Corollary \ref{cor-H_N-decomp}.
\end{proof}

Finally, we consider the equivariant $\slmf(N)$ Khovanov-Rozansky homology of closed negative braids.

\begin{corollary}\label{cor-negative-braids}
Let $P(x,a)$ be any polynomial of form \eqref{def-P}. Suppose the link $L$ is the closure of a negative braid, then $H_{P}^1(L) \cong 0$ and $H_{P}^0(L)$, $H_{P}^2(L)$ are both free $\C[a]$-modules. 

In particular, if a knot $K$ is the closure of a negative braid, then $H_{P_1}^1(K) \cong H_{P_1}^2(K) \cong 0$, where $P_1(x,b_1)=x^{N+1}+b_1 x$ and $b_1$ is a homogeneous variable of degree $2N$.
\end{corollary}

\begin{proof}
By the definition of $H_{N}(L)$ in \cite{KR1}, we have $C_{N}^i(L) = 0$ if $i<0$. So $H_{N}^i(L) \cong 0$ if $i<0$. This implies that $H_{P}^0(L)$ is free. In \cite[Theorem 5]{Stosic-negative-braids}, Stosic proved that $H_{N}^1(L) \cong 0$, which implies that $H_{P}^1(L) \cong 0$ and $H_{P}^2(L)$ is a free $\C[b_1]$-module. 

For the knot $K$, recall that the free part of $H_{P_1}(K)$ is supported on homological degree $0$. So $H_{P_1}^2(K)$ being free means it vanishes.
\end{proof}

\subsection{Organization of this paper} We review the constructions of $H_P(L)$, $E_r(L)$ and $\hat{E}_r(L)$ in Sections \ref{sec-homology-def} and \ref{sec-spectral-sequence}. Then we prove Theorems \ref{thm-H_P-decomp}, \ref{thm-spectral-sequence-decomp} and \ref{thm-exact-couple-lee-gornik} in Section \ref{sec-spectral-sequence-decomp}. After that, we define the $\bigwedge^\ast \C^{N-1}$-action in Section \ref{sec-action}.

We assume the reader is somewhat familiar with the construction of the $\slmf(N)$ Khovanov-Rozansky homology in \cite{KR1}.

\begin{acknowledgments}
The author would like to thank Alexander Shumakovitch for very interesting discussions.
\end{acknowledgments}

\section{Definition of $H_P$}\label{sec-homology-def}

In the remainder of this paper, $N$ is a fixed positive integer with $N \geq 2$. We review the construction of equivariant $\slmf(N)$ Khovanov-Rozansky in a more general setting, which is needed in Section \ref{sec-action}. In the current section and Section \ref{sec-spectral-sequence} below, 
\begin{equation}\label{def-P-general}
P=P(x,a_1,\dots,a_n)=x^{N+1} + xF(x,a_1,\dots,a_n), 
\end{equation}
where $x$ is a homogeneous variable of degree $2$, $a_j$ is a homogeneous variable of degree $2k_j$, and $F(x,a_1,\dots,a_n)$ is a homogeneous element of $\C[x,a_1,\dots,a_n]$ of degree $2N+2$ satisfying $F(x,0\dots,0)=0$.

\subsection{Graded and filtered matrix factorizations} Let $R=\C[x_1,\dots,x_m,a_1,\dots,a_n]$, where $x_1,\dots,x_m$ are homogeneous variables of degree $2$ and $a_j$ is a homogeneous variable of degree $2k_j$ for $1\leq j\leq n$. We endow two structures on $R$: 
\begin{itemize}
	\item The polynomial grading with degree function $\deg$ given by 
	\[
	\deg (\prod_{j=1}^n a_j^{p_j} \cdot \prod_{i=1}^m x_i^{l_i}) = \sum_{j=1}^n 2k_j p_j + \sum_{i=1}^m 2l_i.
	\]
	\item The $x$-filtration $0= \fil_x^{-1} R \subset \fil_x^0 R \subset \cdots \subset \fil_x^n R \subset \cdots$ such that $(\prod_{i=1}^n a_j^{p_j} \cdot \prod_{i=1}^m x_i^{l_i}) \in \fil_x^n R$ if and only if $\sum_{i=1}^m 2l_i\leq n$.  The degree function $\deg_x$ of $\fil_x$ is given by 
	\[
	\deg_x (\prod_{j=1}^n a_j^{p_j} \cdot \prod_{i=1}^m x_i^{l_i}) = \sum_{i=1}^m 2l_i.
	\]
\end{itemize}
Unless otherwise specified, when we say an element is homogeneous, we mean it is homogeneous with respect to the polynomial grading.

\begin{definition}\label{def-grading-filtration}
Let $M$ be an $R$-module. We say that $M$ is a graded $R$-module if it is endowed with a grading $M=\bigoplus_{i} M_i$ such that, for any homogeneous element $r$ of $R$, $rM_i \subset M_{i+\deg r}$. We say that $M$ is an $x$-filtered $R$-module if it is endowed with an increasing filtration $\fil_x$ such that, for any element $r$ of $R$, $r\fil_x^i M \subset \fil_x^{i+\deg_x r} M$.
\end{definition}

\begin{definition}\label{def-mf}
Let $w$ be a homogeneous element of $R$ with $\deg{w}=2N+2$. A matrix factorization $M$ of $w$ over $R$ is a collection of two free $R$-modules $M^0$, $M^1$ and two $R$-module homomorphisms $d^0:M^0\rightarrow M^1$, $d^1:M^1\rightarrow M^0$, called differential maps, such that $d^1d^0=w\cdot\id_{M^0}$ and $d^0d^1=w\cdot\id_{M^1}$. We usually write $M$ as $M^0 \xrightarrow{d^0} M^1 \xrightarrow{d^1} M^0$.

We call $M$ graded if $M^0$, $M^1$ are graded $R$-modules and $d^0$, $d^1$ are homogeneous homomorphisms with $\deg d^0 = \deg d^1 =N+1$. 

We call $M$ $x$-filtered if $M^0$ and $M^1$ are $x$-filtered $R$-modules and $\deg_x d^0, \deg_x d^1  \leq N+1$.
\end{definition}

In the definition of $H_P$, we use only Koszul matrix factorizations defined below.

\begin{definition}\label{def-mf-Koszul}
Let $b$ and $c$ be homogeneous elements of $R$ with $\deg{(bc)}=2N+2$. Denote by $(b,c)_R$ the Koszul matrix factorization 
\[
R \xrightarrow{b} R\{N+1-\deg{b}\} \xrightarrow{c} R,
\]
where $b,~c$ act on $R$ by multiplication and ``$\{s\}$" means shifting by $s$ both the polynomial grading and the $x$-filtration\footnote{That is, in $R\{s\}$, $\deg (\prod_{j=1}^n a_j^{p_j} \cdot \prod_{i=1}^m x_i^{l_i}) = s+\sum_{j=1}^n 2k_j p_j + \sum_{i=1}^m 2l_i$ and $\deg_x (\prod_{j=1}^n a_j^{p_j} \cdot \prod_{i=1}^m x_i^{l_i}) = s + \sum_{i=1}^m 2l_i$.} of $R$. This matrix factorization of $bc$ is both graded and $x$-filtered.

For homogeneous elements $b_1\cdots,b_l,c_1,\cdots,c_l$ of $R$ with $\deg(b_ic_i)=2N+2$, $i=1,\cdots,l$, denote by
\[
\left(%
\begin{array}{cc}
  b_1 & c_1 \\
  b_2 & c_2 \\
  \vdots & \vdots \\
  b_l & c_l \\
\end{array}%
\right)_R
\]
the Koszul matrix factorization $(b_1,c_1)_R \otimes_R(b_2,c_2)_R\otimes_R\cdots\otimes_R(b_l,c_l)_R$. This matrix factorization of $w=\sum_{i=1}^l b_ic_i$ is again both graded and $x$-filtered. 

When $R$ is clear from context, we drop it from the notation.
\end{definition}

\begin{definition}\label{definition-Koszul-basis}
As a free $R$-module, $(b,c)_R$ has a basis $\{1_0,1_1\}$, where $1_\ve$ is the ``$1$" in the copy of $R$ with $\zed_2$-grading $\ve$. More generally, the tensor product
\[
\left(%
\begin{array}{cc}
  b_1 & c_1 \\
  b_2 & c_2 \\
  \vdots & \vdots \\
  b_l & c_l \\
\end{array}%
\right)_R = (b_1,c_1)_R \otimes_R(b_2,c_2)_R\otimes_R\cdots\otimes_R(b_l,c_l)_R
\]
has a basis $\{1_{\vec{\ve}}~|~\vec{\ve}=(\ve_1,\dots,\ve_l) \in \zed_2^{l}\}$, where $1_{\vec{\ve}} = 1_{\ve_1} \otimes\cdots\otimes 1_{\ve_l}$. We call $\{1_{\vec{\ve}}\}$ the standard basis for this Koszul matrix factorization.

Note that 
\begin{itemize}
	\item $\{1_{\vec{\ve}}\}$ is a homogeneous basis with respect to the polynomial grading.
	\item $\deg 1_{\vec{\ve}} = \deg_x 1_{\vec{\ve}}$ for every $\vec{\ve}$.
	\item $\deg_x(\sum_{\vec{\ve}} f_{\vec{\ve}} 1_{\vec{\ve}}) \leq l$ if and only if $\deg_x f_{\vec{\ve}} \leq l- \deg_x 1_{\vec{\ve}}$ for every $\vec{\ve}$.
\end{itemize}
\end{definition}

\subsection{The matrix factorization associated to a MOY graph}

\begin{definition}\label{def-MOY-graph}

A MOY graph $\Gamma$ is a finite oriented graph embedded in $\mathbb{R}^2$ with the following properties:
\begin{enumerate}
    \item Edges of $\Gamma$ are divided into two types: regular edges and wide edges.
    \item Vertices of $\Gamma$ are of two types: 
     \begin{itemize}
	   \item endpoints: $1$-valent vertices that are endpoints of regular edges,
	   \item internal vertices: $3$-valent vertices with 
	   \begin{itemize}
    	\item either two regular edges pointing inward and one wide edge pointing outward,
    	\item or two regular edges pointing outward and one wide edge pointing inward.
     \end{itemize}
     \end{itemize}
\end{enumerate}
We say that $\Gamma$ is closed if it has no endpoints.

A marking of $\Gamma$ consists of 
\begin{enumerate}
	\item A finite set of marked points on $\Gamma$ such that: 
	 \begin{itemize}
	 \item every regular edge contains at least one marked point, no wide edges contain any marked points,
	 \item every endpoint of $\Gamma$ is marked, none of the internal vertices are marked.
   \end{itemize}
	\item An assignment that assigns to each marked point a different homogeneous variable of degree $2$.
\end{enumerate}
\end{definition}

In the rest of this subsection, we fix a MOY graph $\Gamma$ and a marking of $\Gamma$. Assume $x_1,\dots,x_m$ are the variables assigned to the marked points of $\Gamma$. Define $R=\C[x_1,\dots,x_m,a_1,\dots,a_n]$, where $a_j$ is a homogeneous variable of degree $2k_j$ for $1\leq j\leq n$. Let $P(x,a_1,\dots,a_n)$ be the homogeneous polynomial given in \eqref{def-P-general}.

Now cut $\Gamma$ at all the marked points. This cuts $\Gamma$ into a collect of simple marked MOY graphs $\Gamma_1,\dots,\Gamma_l$ of the two types in Figure \ref{fig-pieces-of-Gamma}.

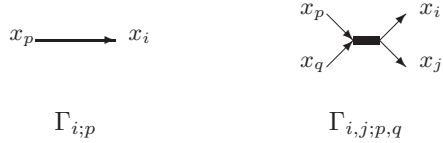
\begin{figure}[ht]
\[
\xymatrix@R=3pt{
 \setlength{\unitlength}{1pt}
\begin{picture}(50,20)(-10,0)

\put(-10,10){\small{$x_p$}}

\put(0,10){\vector(1,0){30}}

\put(35,10){\small{$x_i$}}

\end{picture} & & \setlength{\unitlength}{1pt}
\begin{picture}(50,20)(-10,0)

\put(-10,20){\small{$x_p$}}

\put(-10,0){\small{$x_q$}}

\put(0,0){\vector(1,1){10}}

\put(0,20){\vector(1,-1){10}}

\put(20,10){\vector(1,1){10}}

\put(20,10){\vector(1,-1){10}}

\put(35,20){\small{$x_i$}}

\put(35,0){\small{$x_j$}}

\linethickness{3pt}

\put(10,10){\line(1,0){10}}

\end{picture} \\
 \Gamma_{i;p} & & \Gamma_{i,j;p,q}
}
\]

\caption{Pieces of $\Gamma$}\label{fig-pieces-of-Gamma}

\end{figure}

Define 
\begin{equation}\label{eq-def-v}
v_{i;p}= \frac{P(x_i,a_1,\dots,a_n)-P(x_p,a_1,\dots,a_n)}{x_i-x_p}. 
\end{equation}
Since $P(x_i,a_1,\dots,a_n) + P(x_j,a_1,\dots,a_n)$ is symmetric in $x_i$ and $x_j$, there is a unique polynomial \linebreak $G(X,Y,a_1,\dots,a_n)$ satisfying 
\[
G(x_i+x_j,x_i x_j,a_1,\dots,a_n)=P(x_i,a_1,\dots,a_n) + P(x_j,a_1,\dots,a_n).
\] 
Define
\begin{eqnarray}
\label{eq-def-u1} u_{i,j;p,q}^{\left\langle 1\right\rangle} & = & \frac{G(x_i+x_j,x_i x_j,a_1,\dots,a_n)-G(x_p+x_q,x_i x_j,a_1,\dots,a_n)}{x_i+x_j-x_p-x_q}, \\
\label{eq-def-u2} u_{i,j;p,q}^{\left\langle 2\right\rangle} & = & \frac{G(x_p+x_q,x_i x_j,a_1,\dots,a_n)- G(x_p+x_q,x_p x_q,a_1,\dots,a_n)}{x_ix_j-x_px_q}.
\end{eqnarray}
Note that $v_{i;j},~u_{i,j;p,q}^{\left\langle 1\right\rangle},~u_{i,j;p,q}^{\left\langle 2\right\rangle}$ are all homogeneous elements of $R$.

\begin{definition}\label{def-mf-MOY}
\begin{eqnarray}
\label{eq-def-mf-Gamma-ij} C_P(\Gamma_{i;p}) & := & (v_{i;p}, x_i-x_p)_R, \\
\label{eq-def-mf-Gamma-ijpq} C_P(\Gamma_{i,j;p,q})& := & 
\left(%
\begin{array}{cc}
  u_{i,j;p,q}^{\left\langle 1\right\rangle} & x_i+x_j-x_p-x_q \\
  u_{i,j;p,q}^{\left\langle 2\right\rangle} & x_ix_j-x_px_q \\
\end{array}%
\right)_{R}\{-1\}, \\
\label{eq-def-mf-Gamma} C_P(\Gamma) & := & C_P(\Gamma_1) \otimes_R C_P(\Gamma_2) \otimes_R \cdots \otimes_R C_P(\Gamma_l).
\end{eqnarray}
\end{definition}

Note that 
\begin{enumerate}
	\item $C_P(\Gamma_{i;p})$ is a Koszul matrix factorization of $w_{i;p}= P(x_i,a_1,\dots,a_n)-P(x_p,a_1,\dots,a_n)$.
	\item $C_P(\Gamma_{i,j;p,q})$ is a Koszul matrix factorization of 
	      \[
	      w_{i,j;p,q}=P(x_i,a_1,\dots,a_n)+P(x_j,a_1,\dots,a_n)-P(x_p,a_1,\dots,a_n)-P(x_q,a_1,\dots,a_n).
	      \]
	\item $C_P(\Gamma)$ is a Koszul matrix factorization of 
	\[
	w= \sum_{x_i \text{ is assigned to an endpoint.}} \pm P(x_i,a_1,\dots,a_n),
	\]
	where the sign is positive if $\Gamma$ points towards the corresponding endpoint and is negative if $\Gamma$ points away from the corresponding endpoint. In particular, $w=0$ if $\Gamma$ is closed.
\end{enumerate}

\begin{definition}\label{def-H-P-Gamma}
For a closed MOY graph $\Gamma$, define 
\begin{enumerate}
	\item $H_P(\Gamma)$ to be the homology of $C_P(\Gamma)$, which inherits the polynomial grading, the $x$-filtration and the $\zed_2$-grading of $C_P(\Gamma)$,
	\item $H_N(\Gamma)$ to be the homology of $C_N(\Gamma)=C_P(\Gamma)/(a_1,\dots,a_n)\cdot C_P(\Gamma)$, which inherits the polynomial grading and the $\zed_2$-grading of $C_P(\Gamma)$,
	\item $\hat{H}_P(\Gamma)$ to be the homology of $\hat{C}_P(\Gamma)=C_P(\Gamma)/(a_1-1,\dots,a_n-1)\cdot C_P(\Gamma)$, which inherits the $x$-filtration and the $\zed_2$-grading of $C_P(\Gamma)$.
\end{enumerate}

\end{definition}

\subsection{The chain complex associated to a link diagram}  A marking of a link diagram $D$ consists of
\begin{enumerate}
	\item a finite collection of marked points on $D$ such that none of the crossings are marked and every arc between two crossings contains at least one marked point,
	\item an assignment that assigns to each marked point a different homogeneous variable of degree $2$.
\end{enumerate}

Let $D$ be an oriented link diagram with a marking. Assume $x_1,\dots,x_m$ are the variables assigned to the marked points of $D$. Define $R=\C[x_1,\dots,x_m,a_1,\dots,a_n]$, where $a_j$ is a homogeneous variable of degree $2k_j$. Cut $D$ at its marked points. This cuts $D$ into simple pieces $D_1,\dots,D_l$ of the types shown in Figure \ref{fig-D-pieces}.

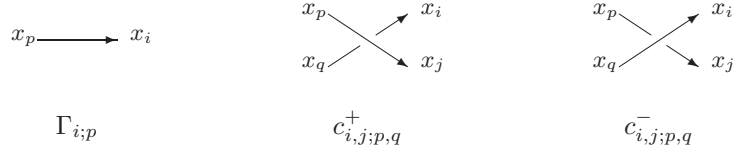
\begin{figure}[ht]
\[
\xymatrix@R=3pt{
\setlength{\unitlength}{1pt}
\begin{picture}(50,20)(-10,0)

\put(-10,10){\small{$x_p$}}

\put(0,10){\vector(1,0){30}}

\put(35,10){\small{$x_i$}}

\end{picture} && \setlength{\unitlength}{1pt}
\begin{picture}(50,20)(-10,0)

\put(-10,20){\small{$x_p$}}

\put(-10,0){\small{$x_q$}}

\put(0,20){\vector(3,-2){30}}

\put(0,0){\line(3,2){12}}

\put(18,12){\vector(3,2){12}}

\put(35,20){\small{$x_i$}}

\put(35,0){\small{$x_j$}}

\end{picture} && \setlength{\unitlength}{1pt}
\begin{picture}(50,20)(-10,0)

\put(-10,20){\small{$x_p$}}

\put(-10,0){\small{$x_q$}}

\put(0,20){\line(3,-2){12}}

\put(18,8){\vector(3,-2){12}}

\put(0,0){\vector(3,2){30}}

\put(35,20){\small{$x_i$}}

\put(35,0){\small{$x_j$}}

\end{picture} \\
\Gamma_{i;p} && c^+_{i,j;p,q} && c^-_{i,j;p,q}
}
\]

\caption{Pieces of $D$}\label{fig-D-pieces}

\end{figure}

We define the chain complex $C_P(\Gamma_{i;p})$ to be
\begin{equation}\label{eq-def-complex-arc}
C_P(\Gamma_{i;p}) = 0 \rightarrow C_P(\Gamma_{i;p})\|0\| \rightarrow 0,
\end{equation}
where the term ``$C_P(\Gamma_{i;p})$" on the right hand side is the Koszul matrix factorization defined in Definition \ref{def-mf-MOY} and ``$\|0\|$" means this term is at homological degree $0$.

To define $C_P(c^\pm_{i,j;p,q})$, we need the following lemma.

\begin{figure}[ht]
\[
\xymatrix@R=3pt{
\setlength{\unitlength}{1pt}
\begin{picture}(50,20)(-10,0)

\put(-10,20){\small{$x_p$}}

\put(-10,0){\small{$x_q$}}

\put(0,0){\vector(1,0){30}}

\put(0,20){\vector(1,0){30}}

\put(35,20){\small{$x_i$}}

\put(35,0){\small{$x_j$}}

\end{picture} &\ar@<3ex>[rr]^{\chi_0} && \ar@<-1ex>[ll]^{\chi_1}&  \setlength{\unitlength}{1pt}
\begin{picture}(50,20)(-10,0)

\put(-10,20){\small{$x_p$}}

\put(-10,0){\small{$x_q$}}

\put(0,0){\vector(1,1){10}}

\put(0,20){\vector(1,-1){10}}

\put(20,10){\vector(1,1){10}}

\put(20,10){\vector(1,-1){10}}

\put(35,20){\small{$x_i$}}

\put(35,0){\small{$x_j$}}

\linethickness{3pt}

\put(10,10){\line(1,0){10}}

\end{picture}   \\
\Gamma_{i;p}\sqcup \Gamma_{j;q} &&&& \Gamma_{i,j;p,q}
}
\]

\caption{Homomorphisms $\chi_0$ and $\chi_1$}\label{fig-chi}

\end{figure}
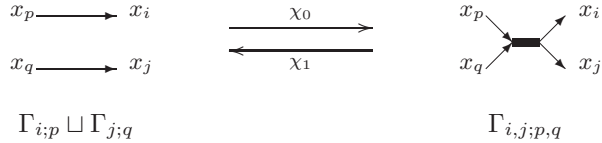

\begin{lemma}\cite{KR1,Krasner,Wu-color-equi}\label{lemma-def-chi}
Up to homotopy and scaling, there is a unique homotopically non-trivial homomorphism $\chi_0: C_P(\Gamma_{i;p}\sqcup \Gamma_{j;q})\rightarrow C_P(\Gamma_{i,j;p,q})$ with $\deg \chi_0 =  1$. And, up to homotopy and scaling, there is a unique homotopically non-trivial homomorphism $\chi_1: C_P(\Gamma_{i,j;p,q})\rightarrow C_P(\Gamma_{i;p}\sqcup \Gamma_{j;q})$ with $\deg \chi_1 = 1$. Moreover, these homomorphisms satisfy $\deg_x \chi_0 = \deg_x \chi_1 = 1$ and, up to scaling by non-zero scalars, $\chi_1 \circ \chi_0 \simeq (x_p-x_j) \id_{C_P(\Gamma_{i;p}\sqcup \Gamma_{j;q})}$, $\chi_0 \circ \chi_1 \simeq (x_p-x_j) \id_{C_P(\Gamma_{i,j;p,q})}$.
\end{lemma}

\begin{proof}
The uniqueness of $\chi_0$ and $\chi_1$ is proved in a more general setting in \cite[Lemma 4.13]{Wu-color-equi}. Here, we only recall the constructions of $\chi_0$ and $\chi_1$ given by Krasner in \cite{Krasner}, which is a straightforward generalization of the corresponding construction by Khovanov and Rozansky in \cite{KR1}.

Recall that 
\[
C_P(\Gamma_{i;p}\sqcup \Gamma_{j;q}) = \left(%
\begin{array}{cc}
  v_{i;p} & x_i-x_p \\
  v_{j;q} & x_j-x_q \\
\end{array}%
\right)_{R} = \left[%
\begin{array}{c}
  R \\
  R\{2-2n\} \\
\end{array}%
\right] \xrightarrow{d^0}
\left[%
\begin{array}{c}
  R\{1-n\} \\
  R\{1-n\} \\
\end{array}%
\right] \xrightarrow{d^1}
\left[%
\begin{array}{c}
  R \\
  R\{2-2n\} \\
\end{array}%
\right],
\]
where 
$
d^0=\left(%
\begin{array}{cc}
  v_{i;p} & x_j-x_q \\
  v_{j;q} & -x_i+x_p \\
\end{array}%
\right)$, 
$d^1=\left(%
\begin{array}{cc}
  x_i-x_p & x_j-x_q \\
  v_{j;q} & -v_{i;p} \\
\end{array}%
\right)
$, and that
\[
 C_P(\Gamma_{i,j;p,q}) =  
\left(%
\begin{array}{cc}
  u_{i,j;p,q}^{\left\langle 1\right\rangle} & x_i+x_j-x_p-x_q \\
  u_{i,j;p,q}^{\left\langle 2\right\rangle} & x_ix_j-x_px_q \\
\end{array}%
\right)_{R}\{-1\} = 
\left[%
\begin{array}{c}
  R\{-1\} \\
  R\{3-2n\} \\
\end{array}%
\right] \xrightarrow{\delta^0}
\left[%
\begin{array}{c}
  R\{-n\} \\
  R\{2-n\} \\
\end{array}%
\right] \xrightarrow{\delta^1}
\left[%
\begin{array}{c}
  R\{-1\} \\
  R\{3-2n\} \\
\end{array}%
\right],
\]
where
$\delta^0=\left(%
\begin{array}{cc}
  u_{i,j;p,q}^{\left\langle 1\right\rangle} & x_ix_j-x_px_q \\
  u_{i,j;p,q}^{\left\langle 2\right\rangle} & -x_i-x_j+x_p+x_q \\
\end{array}%
\right)$,$\delta^1=\left(%
\begin{array}{cc}
  x_i+x_j-x_p-x_q & x_ix_j-x_px_q \\
  u_{i,j;p,q}^{\left\langle 2\right\rangle} & -u_{i,j;p,q}^{\left\langle 1\right\rangle} \\
\end{array}%
\right)$.

In the above explicit forms of $C_P(\Gamma_{i;p}\sqcup \Gamma_{j;q})$ and $C_P(\Gamma_{i,j;p,q})$, define $\chi_0: C_P(\Gamma_{i;p}\sqcup \Gamma_{j;q})\rightarrow C_P(\Gamma_{i,j;p,q})$ by the matrices $\chi_0^0=\left(%
\begin{array}{cc}
  x_p-x_j & 0 \\
  z & 1 \\
\end{array}%
\right)$ and $\chi_0^1=\left(%
\begin{array}{cc}
  x_p & -x_j \\
  -1 & 1 \\
\end{array}%
\right)$, and define $\chi_1: C_P(\Gamma_{i,j;p,q})\rightarrow C_P(\Gamma_{i;p}\sqcup \Gamma_{j;q})$ by the matrices $\chi_1^0=\left(%
\begin{array}{cc}
  1 & 0 \\
  -z & x_p-x_j \\
\end{array}%
\right)$ and $\chi_1^1=\left(%
\begin{array}{cc}
  1 & x_j \\
  1 & x_p \\
\end{array}%
\right)$, where $z= -u_{i,j;p,q}^{\left\langle 2\right\rangle}+\frac{u_{i,j;p,q}^{\left\langle 1\right\rangle}+x_iu_{i,j;p,q}^{\left\langle 2\right\rangle}-v_{j;q}}{x_i-x_p}$.
It is straightforward to verify that $\chi_0$ and $\chi_1$ satisfy all the properties in the lemma.
\end{proof}

We define 
\begin{eqnarray}
\label{eq-def-complex-crossing+} C_P(c^+_{i,j;p,q}) & = & 0 \rightarrow C_P(\Gamma_{i,j;p,q})\|-1\|\{N\} \xrightarrow{\chi_1} C_P(\Gamma_{i;p}\sqcup \Gamma_{j;q})\|0\|\{N-1\} \rightarrow 0, \\
\label{eq-def-complex-crossing-} C_P(c^-_{i,j;p,q}) & = & 0 \rightarrow C_P(\Gamma_{i;p}\sqcup \Gamma_{j;q})\|0\|\{1-N\} \xrightarrow{\chi_0} C_P(\Gamma_{i,j;p,q})\|1\|\{-N\} \rightarrow 0.
\end{eqnarray}

\begin{definition}\label{def-complex-link}
$C_P(D) = C_P(D_1) \otimes_R\cdots\otimes_R C_P(D_l)$, where $C_P(D_i)$ is defined in \eqref{eq-def-complex-arc}, \eqref{eq-def-complex-crossing+} and \eqref{eq-def-complex-crossing-}.
\end{definition}

We call the resolution $c^\pm_{i,j;p,q} \leadsto \Gamma_{i;p}\sqcup \Gamma_{j;q}$ a $0$-resolution and $c^\pm_{i,j;p,q} \leadsto \Gamma_{i,j;p,q}$ a $(\pm1)$-resolution. If we choose a $0$- or $(\pm1)$-resolution for every crossing in $D$, then we get a MOY graph, which we call a MOY resolution of $D$. Of course, the marking of $D$ induces a marking of each MOY resolution of $D$. Let $\mathcal{MOY}(D)$ be the set of all MOY resolutions of $D$. Denote by $w$ the writhe of $D$. For each $\Gamma \in \mathcal{MOY}(D)$, let 
\[
\hbar(\Gamma)= (\#\text{ of } (+1)\text{-resolutions in } \Gamma) -  (\#\text{ of } (-1)\text{-resolutions in } \Gamma).
\]
Then, as $\zed^{\oplus2}$-graded $R$-modules, 
\begin{equation}\label{eq-C_P-D-mf}
C_P(D) \cong \bigoplus_{\Gamma \in \mathcal{MOY}(D)} C_P(\Gamma)\|-\hbar(\Gamma)\|\{(N-1)w+\hbar(\Gamma)\}.
\end{equation}

Note that every MOY resolution $\Gamma$ of $D$ is a closed MOY graph. So $C_P(\Gamma)$ is a Koszul matrix factorization of $0$ and, therefore, a $\zed_2$-graded chain complex.\footnote{$C_P(D)$ and $H_P(D)$ both inherit this $\zed_2$-grading. But this $\zed_2$-grading on $H_P(D)$ is always pure and equal to the number of Seifert circles of $D$. So, unless otherwise specified, we do not keep track of this grading.} Thus, the differential maps of the matrix factorizations of the MOY resolutions of $D$ give rise to a differential map $d_{mf}$ on $C_P(D)$ satisfying:
\begin{itemize}
	\item $d_{mf}$ is homogeneous with  $\deg d_{mf} = \deg_x d_{mf} =N+1$,
	\item $d_{mf}$ preserves the homological grading.
\end{itemize}
The differential maps of $C_P(D_i)$ give rise to a differential map $d_\chi$ of $C_P(D)$ satisfying:
\begin{itemize}
	\item $d_{\chi}$ is homogeneous with $\deg d_{\chi} = \deg_x d_{\chi} =0$,
	\item $d_{\chi}$ raises the homological grading by $1$.
\end{itemize}

As in \eqref{def-H_P}, $H_P(D)$ is defined to be $H_P(D) = H(H(C_P(D), d_{mf}), d_\chi)$, which inherits both $\zed$-gradings and the $x$-filtration of $C_P(D)$. The invariance of $H_P(D)$ is stated in Theorem \ref{thm-H_P-invariance}. 

\section{The Lee-Gornik Spectral Sequence}\label{sec-spectral-sequence}

Now we review the construction of $\hat{E}_r(L)$ and construct $E_r(L)$. In this section, $P=P(x,a_1,\dots,a_n)$ is the polynomial defined in \eqref{def-P-general} and $H_P$ is the corresponding equivariant $\slmf(N)$ Khovanov-Rozansky homology over $\C[a_1,\dots,a_n]$.

\subsection{Structure of $H_P(\Gamma)$} Let $\Gamma$ be a closed MOY graph with a marking. As before, assume $x_1,\dots,x_m$ are the variables assigned to the marked points of $\Gamma$ and define $R=\C[x_1,\dots,x_m,a_1,\dots,a_n]$, where $a_j$ is a homogeneous variable of degree $2k_j$. $P(x,a_1,\dots,a_n)$ is the homogeneous polynomial given in \eqref{def-P-general}.

If we replace every wide edge in $\Gamma$ by a pair of parallel regular edges, that is, change $\Gamma_{i,j;p,q}$ to $\Gamma_{i;p}\sqcup \Gamma_{j;q}$ in Figure \ref{fig-chi}, then we change $\Gamma$ into a collection of oriented circles embedded in the plane. Denote by $\ve$ the total rotation number of this collection and call $\ve$ the rotation number of $\Gamma$. Furthermore, denote by $H_P^\ve(\Gamma)$ (resp. $\hat{H}_P^\ve(\Gamma)$) the component of $H_P(\Gamma)$ (resp. $\hat{H}_P(\Gamma)$) of $\zed_2$-grading $\ve$ and by $H_N^{\ve,p}(\Gamma)$ the component of $H_N(\Gamma)$ of $\zed_2$-grading $\ve$ and polynomial grading $p$.

\begin{lemma}\cite[Proposition 3.2]{Gornik}\label{lemma-hat-H_P-Gamma-fil}
As $\C$-linear spaces,
\begin{eqnarray}
\label{eq-graph-hat-homology-free-1} \hat{H}_P^{\ve+1}(\Gamma) & \cong & 0, \\
\label{eq-graph-hat-homology-free-2} \fil^p_x \hat{H}_P^\ve(\Gamma) / \fil^{p-1}_x \hat{H}_P^\ve(\Gamma) & \cong & H_N^{\ve,p}(\Gamma).
\end{eqnarray}
\end{lemma}

See for example \cite[Proposition 2.19]{Wu7} for a complete proof of Lemma \ref{lemma-hat-H_P-Gamma-fil}. Slightly modifying this proof, we get Lemma \ref{lemma-H_P-Gamma-fil}, which is mentioned in \cite{Wu-2braids} without proof. Since certain technical aspects of its proof are needed later on, we prove Lemma \ref{lemma-H_P-Gamma-fil} in details here.

\begin{lemma}\label{lemma-H_P-Gamma-fil}
As graded $\C[a_1,\dots,a_n]$-modules, 
\begin{eqnarray}
\label{eq-graph-homology-free-1} H_P^{\ve+1}(\Gamma) & \cong & 0, \\
\label{eq-graph-homology-free-2} \fil^p_x H_P^\ve(\Gamma) / \fil^{p-1}_x H_P^\ve(\Gamma) & \cong & H_N^{\ve,p}(\Gamma) \otimes_\C \C[a_1,\dots,a_n].
\end{eqnarray}
\end{lemma}

\begin{proof}
Note that $C_P(\Gamma)$ is also a graded $\C[x_1,\dots,x_m]$-module and $\fil_x$ is the increasing filtration induced by this grading structure. We call the grading of the graded $\C[x_1,\dots,x_m]$-module $C_P(\Gamma)$ the $x$-grading of $C_P(\Gamma)$. Denote by $d_0$ and $d_1$ the two differential maps of the matrix factorization $C_P(\Gamma)$. We decompose $d_0$ and $d_1$ into sums of homogeneous components with respect to the $x$-grading. That is,
\begin{eqnarray}
\label{eq-d-0-decomp} d_0 = \sum_{l=0}^{N} d_0^{(l)}, \\
\label{eq-d-1-decomp} d_1 = \sum_{l=0}^{N} d_1^{(l)},
\end{eqnarray}
where $d_1^{(l)}$ and $d_0^{(l)}$ are $R$-module homomorphisms and satisfy:
\begin{itemize}
	\item they are homogeneous with respect to both the polynomial grading and the $x$-grading,
	\item $\deg d_0^{(l)} = \deg d_1^{(l)} =N+1$, $\deg_x d_0^{(l)} = \deg_x d_1^{(l)} = N+1-2l$.
\end{itemize}
Recall that $C_P(\Gamma)$ is a matrix factorization of $0$. So $d_0 \circ d_1 =0$ and $d_1\circ d_0=0$. Comparing the homogeneous parts with respect to the $x$-grading, one gets that, for any $l\geq 0$,
\begin{eqnarray}
\label{eq-d-0-1-vanish} \sum_{i=0}^l d_0^{(i)}\circ d_{1}^{(l-i)}=0, \\
\label{eq-d-1-0-vanish} \sum_{i=0}^l d_1^{(i)}\circ d_{0}^{(l-i)}=0,
\end{eqnarray}
where we use the convention that $d_0^{(i)}=0$ and $d_1^{(i)}=0$ if $i >N$.

By the definition of $C_N(\Gamma)$, there is an isomorphism of $\zed_2$-periodic chain complexes of $\C[a_1,\dots,a_n]$-modules
\[
C_N(\Gamma)\otimes_\C \C[a_1,\dots,a_n] \cong C_P^0(\Gamma) \xrightarrow{d_0^{(0)}} C_P^1(\Gamma) \xrightarrow{d_1^{(0)}} C_P^0(\Gamma),
\]
that preserves the $\zed_2$-grading, the polynomial grading and the $x$-grading. So there is an isomorphism of $\C[a_1,\dots,a_n]$-modules 
\begin{equation}\label{eq-isomorphism-H_N-d-0}
H_N(\Gamma)\otimes_\C \C[a_1,\dots,a_n] \cong H(C_P(\Gamma), d^{(0)})
\end{equation} 
preserving the $\zed_2$-grading, the polynomial grading and the $x$-grading.

Now we are ready to prove that $H_P^{\ve+1}(\Gamma) = 0$. From \cite{KR1}, we know that $H_N^{\ve+1}(\Gamma)=0$. By \eqref{eq-isomorphism-H_N-d-0}, this means $H^{\ve+1}(C_P(\Gamma), d^{(0)})=0$. That is, $\im (d_{\ve}^{(0)}) =\ker (d_{\ve+1}^{(0)})$. Let $\alpha$ be any element in $\ker d_{\ve+1}$ that is homogeneous with respect to the polynomial grading with $\deg \alpha = g$. Decomposing $\alpha$ according to the $x$-grading, we get $\alpha = \sum_{i=-\infty}^{\infty} \alpha_l$, where $\alpha_l$ is homogeneous with respect to both the polynomial grading and the $x$-grading with $\deg \alpha_l =g$, $\deg_x \alpha_l = g-2l$. Of course,  $\alpha_l=0$ for $l<0$ and $l\gg1$ since the $x$-grading is bounded below.

Next, we construct inductively a sequence $\{\beta_l\}_{-\infty}^{\infty} \subset C_P^\ve(\Gamma)$ such that
\begin{enumerate}
	\item $\beta_l=0$ for $l<0$;
	\item each $\beta_l$ is homogeneous with respect to both the polynomial grading and the $x$-grading;
	\item $\deg \beta_l = g-N-1$ and $\deg_x \beta_l = g - 2l - N-1$;
	\item $\alpha_l=\sum_{i=0}^{N} d_\ve^{(i)}\beta_{l-i}$ for all $l \in \zed$.
\end{enumerate}
Again, since the $x$-grading is bounded below, $\beta_l=0$ for $l\gg1$. Note that $\{\beta_l\}_{-\infty}^{-1}$ is the zero sequence and satisfies conditions (1-4) for $l$ up to $-1$. Now assume that, for some $l\geq0$, there is a sequence $\{\beta_l\}_{-\infty}^{l-1}$ satisfies conditions (1-4) up to $l-1$. Let us construct $\beta_l$. In the equation $d_{\ve+1} \alpha=0$, compare the homogeneous parts with respect to the $x$-grading of $x$-degree $N+1+g-2l$. This gives us
\begin{eqnarray*}
0 & = & \sum_{j=0}^{N}  d_{\ve+1}^{(j)} \alpha_{l-j} =  d_{\ve+1}^{(0)} \alpha_{l} + \sum_{j=1}^{N}  d_{\ve+1}^{(j)} \alpha_{l-j} \\
& = &  d_{\ve+1}^{(0)} \alpha_{l} + \sum_{j=1}^{N}  d_{\ve+1}^{(j)} \sum_{i=0}^{N}  d_\ve^{(i)}\beta_{l-j-i} \\
& = &  d_{\ve+1}^{(0)} \alpha_{l} + \sum_{q=1}^{2N}  (\sum_{j=1}^q d_{\ve+1}^{(j)}d_\ve^{(q-j)})\beta_{l-q} \\
\text{(by \eqref{eq-d-0-1-vanish} and \eqref{eq-d-1-0-vanish}) } & = &  d_{\ve+1}^{(0)} \alpha_{l} - \sum_{q=1}^{N}  d_{\ve+1}^{(0)}d_\ve^{(q)}\beta_{l-q} = d_{\ve+1}^{(0)} (\alpha_{l} - \sum_{q=1}^{N}  d_\ve^{(q)}\beta_{l-q}). \\
\end{eqnarray*}
So $\alpha_{l} - \sum_{q=1}^{N} d_\ve^{(q)}\beta_{l-q} \in \ker (d_{\ve+1}^{(0)})= \im (d_{\ve}^{(0)})$. Therefore, there is a $\beta_l$ satisfying conditions (1-3) such that $d_{\ve}^{(0)} \beta_l = \alpha_{l} - \sum_{q=1}^{N} d_\ve^{(q)}\beta_{l-q}$. Thus, $\{\beta_l\}_{-\infty}^{l}$ satisfies conditions (1-4) above. This completes the inductive construction. Note that $\sum_{l=-\infty}^{\infty} \beta_l$ is a finite sum and therefore a well defined element of $C_P^\ve(\Gamma)$. We have 
\begin{eqnarray*}
\alpha & = & \sum_{l=-\infty}^{\infty}\alpha_l = \sum_{l=-\infty}^{\infty}\sum_{i=0}^{N} d_\ve^{(i)}\beta_{l-i} \\
       & = & \sum_{i=0}^{N} d_\ve^{(i)}(\sum_{l=-\infty}^{\infty}\beta_{l-i}) = \sum_{i=0}^{N} d_\ve^{(i)}(\sum_{l=-\infty}^{\infty}\beta_{l}) \\
       & = & d_\ve(\sum_{l=-\infty}^{\infty}\beta_{l}).
\end{eqnarray*}
So $\alpha \in \im d_\ve$. This shows $\im (d_{\ve}) =\ker (d_{\ve+1})$ and therefore $H_P^{\ve+1}(\Gamma) = 0$.

It remains to prove \eqref{eq-graph-homology-free-2}. According to \eqref{eq-isomorphism-H_N-d-0}, we only need to show that
\begin{equation}\label{eq-graph-homology-free-3}
\fil^p_x H_P^\ve(\Gamma) / \fil^{p-1}_x H_P^\ve(\Gamma) \cong H^{\ve,p}(C_P(\Gamma), d^{(0)}),
\end{equation} 
where $H^{\ve,p}(C_P(\Gamma), d^{(0)})$ is the direct sum component of the free $\C[a]$-module $H(C_P(\Gamma), d^{(0)})$ consisting of homogeneous elements of $\zed_2$-grading $\ve$ and $x$-grading $p$. 

Denote by $(\ker d_{\ve}^{(0)})_p$ the $\C[a]$-submodule of $\ker d_{\ve}^{(0)}$ consisting of elements homogeneous with respective to the $x$-grading of $x$-degree $p$. Next, for every $\alpha\in (\ker{d_\ve^{(0)}})_p$, we construct by induction a sequence $\{\alpha_l\}_{0}^{\infty}\subset C_P^\ve(\Gamma)$, such that $\alpha_0=\alpha$, $\alpha_l$ is a homogeneous element with respective to the $x$-grading of $x$-degree $p-2l$, and 
\begin{equation}\label{eq-def-alpha-l-ind}
\sum_{j=0}^{l} d_\ve^{(j)}\alpha_{l-j}=0, ~\forall ~l\in\zed,
\end{equation}
where we use the convention that $d_\ve^{(j)}=0$ for $j>N$. Again, since the $x$-grading is bounded below, $\alpha_l=0$ for $l\gg1$. Clearly, $\{\alpha_l\}_{0}^{0}$ with $\alpha_0=\alpha$ satisfies the initial condition and equation \eqref{eq-def-alpha-l-ind} up to $l=0$. Assume that, for some $l\geq 1$, $\{\alpha_l\}_{0}^{l-1}$ is constructed and satisfies the initial condition and equation \eqref{eq-def-alpha-l-ind} up to $l-1$. Let us find an $\alpha_l$. Note that
\begin{eqnarray*}
d_{\ve+1}^{(0)}(\sum_{j=1}^{l} d_\ve^{(j)}\alpha_{l-j}) & = & \sum_{j=1}^{l} d_{\ve+1}^{(0)}d_\ve^{(j)}\alpha_{l-j} \\
\text{(by \eqref{eq-d-0-1-vanish} and \eqref{eq-d-1-0-vanish}) }& = & -\sum_{j=1}^{l} \sum_{i=0}^{j-1}d_{\ve+1}^{(j-i)}d_\ve^{(i)}\alpha_{l-j} \\
                      (\text{setting }q=l-j+i)         & = & -\sum_{q=0}^{l-1} \sum_{i=0}^q d_{\ve+1}^{(l-q)}d_\ve^{(i)}\alpha_{q-i} \\
                                                   & = & -\sum_{q=0}^{l-1} d_{\ve+1}^{(l-q)}(\sum_{i=0}^q d_\ve^{(i)}\alpha_{q-i}) \\
             (\text{by induction hypothesis})      & = & 0
\end{eqnarray*}
But $H^{\ve+1}(C_P(\Gamma), d^{(0)})=0$, that is, $\im (d_{\ve}^{(0)}) =\ker (d_{\ve+1}^{(0)})$. So there is an $\alpha_l \in C_P^\ve(\Gamma)$ homogeneous with respective to the $x$-grading of $x$-degree $n-2l$ satisfying $d_{\ve}^{(0)} \alpha_l = -\sum_{j=1}^{l} d_\ve^{(j)}\alpha_{l-j}$. Thus, the sequence $\{\alpha_l\}_{0}^{l}$ satisfies the initial condition and equation \eqref{eq-def-alpha-l-ind} up to $l$. This completes the induction and we have the sequence $\{\alpha_l\}_{0}^{\infty}$. As explained above, $\sum_{l=0}^{\infty}\alpha_l$ is in fact a finite sum and therefore a well defined element of $C_P^\ve(\Gamma)$. Note that the homogeneous part of $d_\ve (\sum_{l=0}^{\infty}\alpha_l)$ with respect to the $x$-grading of $x$-degree $N+1+p-2l$ is $\sum_{j=0}^{l}d_\ve^{(j)}\alpha_{l-j}=0$ by equation \eqref{eq-def-alpha-l-ind}. This implies that  $d_\ve (\sum_{l=0}^{\infty}\alpha_l)=0$, that is, $\sum_{l=0}^{\infty}\alpha_l$ is a cycle in $(C_P(\Gamma),d)$.

Define $\tilde{\phi}_p: (\ker{d_\ve^{(0)}})_p \rightarrow \fil^p_x H_P^\ve(\Gamma) / \fil^{p-1}_x H_P^\ve(\Gamma)$ by $\alpha \mapsto [\sum_{l=0}^{\infty}\alpha_l]$. Since the top homogeneous component of $\sum_{l=0}^{\infty}\alpha_l$ with respect to the $x$-grading is $\alpha_0=\alpha$, one can see that $\tilde{\phi}_p(\alpha)$ does not depend on the choice of $\{\alpha_l\}_{0}^{\infty}$ and is well defined. It is also easy to verify that $\tilde{\phi}_p$ is a $\C[a_1,\dots,a_n]$-module homomorphism preserving the polynomial grading. Moreover, $\tilde{\phi}_p$ is surjective. To see this, note that any element of $\fil^p_x H_P^\ve(\Gamma) / \fil^{p-1}_x H_P^\ve(\Gamma)$ can be expressed as $[\sum_{l=0}^{\infty}\alpha_l]$, where $\alpha_l$ is a homogeneous element with respective to the $x$-grading of $x$-degree $p-2l$ and $d_\ve \sum_{l=0}^{\infty}\alpha_l =0$. Comparing the top homogeneous parts with respect to the $x$-grading on both sides of this equation, one gets $d_\ve^{(0)} \alpha_0=0$, which means $\alpha_0 \in (\ker{d_\ve^{(0)}})_p$. By the definition of $\tilde{\phi}_n$, one easily sees that $\tilde{\phi}_p(\alpha_0) = [\sum_{l=0}^{\infty}\alpha_l]$.

Denote by $(\im d_{\ve+1}^{(0)})_p$ the homogeneous component of $\im d_{\ve+1}^{(0)}$ with respect to the $x$-grading of $x$-degree $p$. We prove isomorphism \eqref{eq-graph-homology-free-3} by showing that $\ker \tilde{\phi}_p = (\im d_{\ve+1}^{(0)})_p$. Assume $\alpha \in \ker \tilde{\phi}_p$ and $\{\alpha_l\}_{0}^{\infty}$ is a sequence given by the above inductive construction. Then 
\begin{equation}\label{eq-comp-ker-tilde-phi-1}
\sum_{l=0}^{\infty}\alpha_l = d_{\ve+1}\beta+\gamma, 
\end{equation}
where $\gamma$ is a cycle in $\fil_x^{p-1}C_P^\ve(\Gamma)$, and $\beta\in C_P^{\ve+1}(\Gamma)$ satisfying $d_{\ve+1}\beta\in\fil_x^p C_P^\ve(\Gamma)$. We claim that we can choose $\beta$ so that $\deg_x \beta\leq p-N-1$. Assume that $\deg_x \beta=g> p-N-1$ and denote by $\beta_0$ the top homogeneous part of $\beta$ with respect to the $x$-grading. Comparing the top homogeneous parts with respect to the $x$-grading on both sides of equation \eqref{eq-comp-ker-tilde-phi-1}, we have $d_{\ve+1}^{(0)}\beta_0=0$. So there exists a homogeneous element $\theta\in C_P^\ve(\Gamma)$ of degree $g-N-1$ such that $d_\ve^{(0)}\theta=\beta_0$. Let $\beta'=\beta-d_\ve\theta$. Then $\beta'$ also satisfies the above equation, and $\deg_x \beta' < \deg_x \beta$. Repeat this process. Within finite steps, we get a $\hat{\beta} \in C_P^{\ve+1}(\Gamma)$ with $\deg_x\hat{\beta}\leq p-N-1$ and
\begin{equation}\label{eq-comp-ker-tilde-phi-2}
\sum_{l=0}^{\infty}\alpha_l = d_{\ve+1}\hat{\beta}+\gamma.
\end{equation}
Let $\hat{\beta}_0$ be the homogeneous part of $\hat{\beta}$ with respect to the $x$-grading of $x$-degree $p-N-1$. Comparing the top homogeneous parts with respect to the $x$-grading on both sides of equation \eqref{eq-comp-ker-tilde-phi-2}, one can see that $\alpha=\alpha_0=d_{\ve+1}^{(0)}\hat{\beta}_0$. This shows $\alpha\in(\im{d_{\ve+1}^{(0)}})_p$. So $\ker{\tilde{\phi}_p}\subset(\im{d_{\ve+1}^{(0)}})_p$. On the other hand, if $\alpha \in (\im{d_{\ve+1}^{(0)}})_p$, then $\alpha= d_{\ve+1}^{(0)} \beta$ for some $\beta \in  C_P^{\ve+1}(\Gamma)$ homogeneous with respect to the $x$-grading of $x$-degree $p-N-1$. So 
\[
\sum_{l=0}^{\infty}\alpha_l = d_{\ve+1}^{(0)} \beta + \sum_{l=1}^{\infty}\alpha_l = d_{\ve+1} \beta + (\sum_{l=1}^{\infty}\alpha_l - \sum_{j=1}^{N} d_{\ve+1}^{(j)} \beta) \in \ker{\tilde{\phi}_p}.
\]
Thus, $(\im{d_{\ve+1}^{(0)}})_p\subset\ker{\tilde{\phi}_p}$. This shows $(\im{d_{\ve+1}^{(0)}})_p=\ker{\tilde{\phi}_p}$ and, therefore, $\tilde{\phi}_p$ induces a $\C[a_1,\dots,a_n]$-module isomorphism $\phi_p: H^{\ve,p}(C_P(\Gamma), d^{(0)}) \rightarrow \fil^p_x H_P^\ve(\Gamma) / \fil^{p-1}_x H_P^\ve(\Gamma)$ preserving the polynomial grading.
\end{proof}

\begin{corollary}\label{cor-graph-homology-free}
Let $\Gamma$ be a closed MOY graph. Then 
\begin{enumerate}
  \item $\hat{H}_P(\Gamma)$ is a finite dimensional $\C$-space and its $x$-filtration is bounded and exhaustive.
	\item $H_P(\Gamma)$ is a finitely generated graded-free $\C[a_1,\dots,a_n]$-module and its $x$-filtration is bounded and exhaustive, where ``graded-free" means $H_P(\Gamma)$  is graded, free and admits a homogeneous basis.
\end{enumerate}
\end{corollary}

\begin{proof}
From \cite{KR1}, we know that $H_N(\Gamma)$ is finite dimensional and its polynomial grading is bound above and below. In addition, by their definitions, we know that the $x$-filtrations of $\hat{H}_P(\Gamma)$ and $H_P(\Gamma)$ are bounded below and exhaustive. Using Lemmas \ref{lemma-hat-H_P-Gamma-fil} and \ref{lemma-H_P-Gamma-fil}, one can inductively prove that, for every $p$,
\begin{itemize}
	\item $\fil_x^p \hat{H}_P(\Gamma)$ is a finite dimensional $\C$-space,
	\item $\fil_x^p H_P(\Gamma)$ is a finitely generated free $\C[a_1,\dots,a_n]$-module.
\end{itemize}
Lemmas \ref{lemma-hat-H_P-Gamma-fil} and \ref{lemma-H_P-Gamma-fil} also imply that the $x$-filtrations of $\hat{H}_P(\Gamma)$ and $H_P(\Gamma)$ are bounded above. Thus, 
\begin{itemize}
	\item $\hat{H}_P(\Gamma)$ is itself a finite dimensional $\C$-space,
	\item $H_P(\Gamma)$ is itself a finitely generated free $\C[a_1,\dots,a_n]$-module.
\end{itemize}
Finally, since the polynomial grading of $H_P(\Gamma)$ is bounded below, we know that $H_P(\Gamma)$ is a graded-free $\C[a_1,\dots,a_n]$-module by, for example, \cite[Lemma 3.3]{Wu-color}.
\end{proof}

\subsection{$E_r(L)$ and $\hat{E}_r(L)$} Using Lemmas \ref{lemma-hat-H_P-Gamma-fil}, \ref{lemma-H_P-Gamma-fil} and Corollary \ref{cor-graph-homology-free}, it is straightforward to prove Theorem \ref{thm-spectral-sequence}. We summarize the key observation in the proof as the following lemma.

\begin{lemma}\label{lemma-chi-commute}
Suppose $\Gamma_1$ is a closed MOY graph and $\Gamma_0$ is obtained from $\Gamma_1$ by replacing a wide edge by a pair of parallel regular edges.\footnote{That is, replacing a piece of $\Gamma_1$ of the form $\Gamma_{i,j;p,q}$ in Figure \ref{fig-chi} by $\Gamma_{i;p}\sqcup \Gamma_{j;q}$ in the same figure.} Denote by $\xymatrix{C_P(\Gamma_0) \ar@<.5ex>[r]^{\chi_0} & C_P(\Gamma_1) \ar@<.5ex>[l]^{\chi_1}}$ the homomorphisms induced by this local change and by $\chi_0^{(0)}$, $\chi_1^{(0)}$ the top homogeneous parts of $\chi_0$, $\chi_1$ with respect to the $x$-grading. In addition, we denote by $d_{mf}^{(0)}$ the top homogeneous parts of the differential maps of $C_P(\Gamma_0)$ and $C_P(\Gamma_1)$ with respect to the $x$-grading. Then
\begin{itemize}
	\item $\xymatrix{(C_P(\Gamma_0), d_{mf}^{(0)}) \ar@<.5ex>[r]^{\chi_0^{(0)}} & (C_P(\Gamma_0), d_{mf}^{(0)}) \ar@<.5ex>[l]^{\chi_1^{(0)}}}$ are homomorphisms of matrix factorizations of $0$.
	\item The following squares commute, where $\ve$ is the rotation number of $\Gamma_0$ and $\Gamma_1$, and $\phi_{p,\Gamma_0}$, $\phi_{p,\Gamma_1}$ are the isomorphisms constructed in the proof of Lemma \ref{lemma-H_P-Gamma-fil}.
	\[
	\xymatrix{
	H^{\ve,p}(C_P(\Gamma_0), d_{mf}^{(0)}) \ar[d]^{\chi_0^{(0)}} \ar[r]^{\phi_{p,\Gamma_0}} & \fil^p_x H_P^\ve(\Gamma_0) / \fil^{p-1}_x H_P^\ve(\Gamma_0)\ar[d]^{\chi_0} \\
	H^{\ve,p}(C_P(\Gamma_1), d_{mf}^{(0)}) \ar[r]^{\phi_{p,\Gamma_1}} & \fil^p_x H_P^\ve(\Gamma_1) / \fil^{p-1}_x H_P^\ve(\Gamma_1)
	} 
	\hspace{1.5pc}
	\xymatrix{
	H^{\ve,p}(C_P(\Gamma_1), d_{mf}^{(0)}) \ar[d]^{\chi_1^{(0)}} \ar[r]^{\phi_{p,\Gamma_1}} & \fil^p_x H_P^\ve(\Gamma_1) / \fil^{p-1}_x H_P^\ve(\Gamma_1)\ar[d]^{\chi_1} \\
	H^{\ve,p}(C_P(\Gamma_0), d_{mf}^{(0)}) \ar[r]^{\phi_{p,\Gamma_0}} & \fil^p_x H_P^\ve(\Gamma_0) / \fil^{p-1}_x H_P^\ve(\Gamma_0)
	} 
	\]
\end{itemize}
\end{lemma}

\begin{proof}
This lemma follows easily from the constructions of $\chi_0$, $\chi_1$ and $\phi_n$. We leave the details to the reader.
\end{proof}

The part of Theorem \ref{thm-spectral-sequence} about $\hat{E}_r(L)$ is proved in \cite{Gornik,Wu7}. So we only need to prove the part about $\{E_r(L)\}$, which is a special case of the following theorem.

\begin{theorem}\label{thm-spectral-sequence-general}
$x$-filtration $\fil_x$ on the chain complex $(H(C_P(D), d_{mf}), d_\chi)$ induces a spectral sequence $\{E_r(L)\}$ converging to $H_P(L)$ with $E_1(L) \cong H_N(L)\otimes_\C \C[a_1,\dots,a_n]$.
\end{theorem}

\begin{proof}
By Corollary \ref{cor-graph-homology-free}, the $x$-filtration of $H(C_P(D),d_{mf})$ is bounded and exhaustive. So $E_r(L)$ converges to $H_P(L)$. It remains to show that $E_1(L) \cong H_N(L)\otimes_\C \C[a_1,\dots,a_n]$. By Lemma \ref{lemma-chi-commute}, we know that $E_0(L)$ is isomorphic to the chain complex $(H(C_P(D),d_{mf}^{(0)}), d_\chi^{(0)})$, where $d_{mf}^{(0)}$ and $d_\chi^{(0)}$ are the top homogeneous parts of $d_{mf}$ and $d_\chi$ with respect to the $x$-grading of $C_P(D)$. So $E_1(L)\cong H(H(C_P(D),d_{mf}^{(0)}), d_\chi^{(0)})$. On the other hand, by the definition of $H_N(L)$, it is easy to see that $H(H(C_P(D),d_{mf}^{(0)}), d_\chi^{(0)}) \cong H_N(L)\otimes_\C \C[a_1,\dots,a_n]$. So $E_1(L) \cong H_N(L)\otimes_\C \C[a_1,\dots,a_n]$.
\end{proof}

\section{Decomposition Theorems}\label{sec-spectral-sequence-decomp}

Next, we prove Theorems \ref{thm-H_P-decomp}, \ref{thm-spectral-sequence-decomp}, and compute the spectral sequences of the filtered chain complexes $F_{i,s}$, $T_{i,m,s}$, $\hat{F}_{i,s}$ and $\hat{T}_{i,m,s}$ defined in \eqref{def-F-i-s-fil}--\eqref{def-hat-T-i-m-s-fil}. Theorem \ref{thm-exact-couple-lee-gornik} follows easily from these.

In this section, $P=P(x,a)$ is a homogeneous polynomial of form \eqref{def-P} and $H_P$ is the corresponding equivariant $\slmf(N)$ Khovanov-Rozansky homology over $\C[a]$. Recall that $\deg a =2k$.

\subsection{A closer look at $\fil_x$}\label{subsec-fil-decomp} To prove Theorems \ref{thm-H_P-decomp} and \ref{thm-spectral-sequence-decomp}, we need to better understand the relation between the polynomial grading and the $x$-filtration. The goal of this subsection is to show that, for a closed MOY graph $\Gamma$, any direct sum decomposition of $H_P(\Gamma)$ in the category of graded $\C[a]$-modules is also a direct sum decomposition in the category of filtered $\C$-spaces. Theorems \ref{thm-H_P-decomp} and \ref{thm-spectral-sequence-decomp} both follow from this.

In the rest of this subsection, $\Gamma$ is a closed MOY graph with a marking, $x_1,\dots,x_m$ are the variables assigned to the marked points of $\Gamma$ and $R=\C[x_1,\dots,x_m,a]$, where $a$ is a homogeneous variable of degree $2k$. $P(x,a)$ is a homogeneous polynomial of form \eqref{def-P}. Unless otherwise specified, when we say an element is homogeneous, we mean it is homogeneous with respect to the polynomial grading.

We start with simple observations.

\begin{lemma}\label{lemma-degree-a}
Suppose $M$ is a Koszul matrix factorization over $R$ (see Definition \ref{def-mf-Koszul}) and $\rho$ is a homogeneous element of $M$. Then $\deg_x \rho \leq \deg \rho$, and $\deg_x \rho < \deg \rho$ if and only if $\rho \in aM$.
\end{lemma}

\begin{proof}
Let $\{1_{\vec{\ve}}\}$ be the standard basis for $M$ defined in Definition \ref{definition-Koszul-basis}. Then $\rho = \sum_{\vec{\ve}} f_{\vec{\ve}}  1_{\vec{\ve}}$, where $f_{\vec{\ve}}$ is a homogeneous element of $R$ with $\deg f_{\vec{\ve}} = \deg \rho - \deg  1_{\vec{\ve}}$. Note that, for every $f \in R$, $\deg_x f \leq \deg f$, and $\deg_x f < \deg f$ if and only if $f \in aR$. So $\deg_x f_{\vec{\ve}} + \deg_x 1_{\vec{\ve}} \leq \deg f_{\vec{\ve}} + \deg 1_{\vec{\ve}} = \deg \rho$ $\forall ~\vec{\ve}$. This shows that $\deg_x \rho \leq \deg \rho$. Moreover, $\deg_x \rho < \deg \rho$ if and only if $\deg_x f_{\vec{\ve}} < \deg f_{\vec{\ve}}$ $\forall~\vec{\ve}$ if and only if $f_{\vec{\ve}} \in aR$  $\forall~\vec{\ve}$.
\end{proof}

\begin{lemma}\label{lemma-quotient-basis}
Let $M$ be a finitely generated free $\C[a]$-module and $\{v_i\}$ a basis for $M$. For any $\lambda \in \C$, denote by $\pi_{a-\lambda}:M \rightarrow M/(a-\lambda)M$ the standard quotient map. Then $\{\pi_{a-\lambda}(v_i)\}$ is a basis for the $\C$-space $M/(a-\lambda)M$.
\end{lemma}

\begin{proof}
Since $\{v_i\}$ spans $M$, we know that $\{\pi_{a-\lambda}(v_i)\}$ spans $M/(a-\lambda)M$. It remains to show that $\{\pi_{a-\lambda}(v_i)\}$ is linearly independent. Suppose $\{c_i\}\subset \C$ satisfies $\sum_i c_i \pi_{a-\lambda}(v_i)=0$. Then $\sum_i c_i v_i \in (a-\lambda)M$. Therefore, there are $f_i \in R$ such that $\sum_i c_i v_i = (a-\lambda) \sum_i f_i v_i$. That is, $\sum_i (c_i-(a-\lambda)f_i)v_i=0$. Since $\{v_i\}$ a basis for $M$, this means $c_i-(a-\lambda)f_i=0$ and, therefore, $c_i=f_i=0$ for every $i$. 
\end{proof}

\begin{lemma}\label{lemma-ring-components-degree}
For $f \in R$, denote by $f_i$ the homogeneous component of $f$ with $\deg f_i = i$. Then $\deg_x f_i \leq \deg_x f$ for every $i$.
\end{lemma}

\begin{proof}
Obvious.
\end{proof}

\begin{lemma}\label{lemma-homology-components-degree}
\begin{enumerate}
	\item If $u \in \fil_x^n C_P(\Gamma)$, then all homogeneous components of $u$ are also in $\fil_x^n C_P(\Gamma)$.
	\item If $[u]\in \fil_x^n H_P(\Gamma)$, then all homogeneous components of $[u]$ are also in $\fil_x^n H_P(\Gamma)$.
\end{enumerate}
\end{lemma}

\begin{proof}
$C_P(\Gamma)$ is a Koszul matrix factorization. Denote by $\{1_{\vec{\ve}}\}$ the standard basis for $C_P(\Gamma)$ given in Definition \ref{definition-Koszul-basis}. Recall that $\{1_{\vec{\ve}}\}$ is a homogeneous basis with respect to the polynomial grading.

To prove Part (1) of the lemma, assume $u \in \fil_x^n C_P(\Gamma)$ and denote by $u_i$ the homogeneous component of $u$ with $\deg u_i =i$. Every $u_i$ can be uniquely expressed as $u_i = \sum_{\vec{\ve}} g_{i,\vec{\ve}} 1_{\vec{\ve}}$, where $g_{i,\vec{\ve}}$ is a homogeneous element of $R$ with $\deg g_{i,\vec{\ve}} = i - \deg 1_{\vec{\ve}}$. Then $u = \sum_i u_i = \sum_{\vec{\ve}} (\sum_i g_{i,\vec{\ve}}) 1_{\vec{\ve}}$. Since $u \in \fil_x^n C_P(\Gamma)$, we have $\deg_x \sum_i g_{i,\vec{\ve}} \leq n - \deg_x 1_{\vec{\ve}}$ for every $\vec{\ve}$. Note that $g_{i,\vec{\ve}}$ is the homogeneous component of $\sum_i g_{i,\vec{\ve}}$ of polynomial degree $i - \deg 1_{\vec{\ve}}$. Thus, by Lemma \ref{lemma-ring-components-degree}, we have $\deg_x g_{i,\vec{\ve}} \leq n-\deg_x 1_{\vec{\ve}}$. So $\deg_x u_i =\deg_x \sum_{\vec{\ve}} g_{i,\vec{\ve}} 1_{\vec{\ve}} \leq n$. This proves Part (1).

To prove Part (2), note that $[u]$ is represented by a cycle $u \in \fil_x^n C_P(\Gamma)$. Denote by $u_i$ the homogeneous component of $u$ with $\deg u_i =i$. Since the differential of $C_P(\Gamma)$ is homogeneous, each $u_i$ is a cycle, and $[u_i]$ is the homogeneous component of $[u]$ with $\deg [u_i]=i$. Then Part (2) of the lemma follows from Part (1). 
\end{proof}

By Corollary \ref{cor-graph-homology-free}, $H_P(\Gamma)$ is a finitely generated graded-free $\C[a]$-module. The next lemma determines the $x$-filtration degrees of elements of homogeneous bases for $H_P(\Gamma)$.

\begin{lemma}\label{lemma-homology-basis-degree}
Let $\{[u_j]\}$ be any homogeneous basis for the free $\C[a]$-module $H_P(\Gamma)$. Then $\deg_x [u_j] = \deg [u_j]$ for every $j$.
\end{lemma}

\begin{proof}
By Lemma \ref{lemma-degree-a}, we know that $\deg_x [u_j] \leq \deg [u_j]$ for every $j$. Assume $n=\deg_x [u_j] < \deg [u_j] =l$ for a certain $j$. Then $[u_j]$ is represented by a cycle $u_j \in \fil_x^n C_P(\Gamma)$. Denote by $u_{j,i}$ the homogeneous component of $u_j$ with $\deg u_{j,i} = i$. By Lemma \ref{lemma-homology-components-degree}, we know that $u_{j,i} \in \fil_x^n C_P(\Gamma)$ for every $i$. Also, since the differential of $C_P(\Gamma)$ is homogeneous, each $u_{j,i}$ is itself a cycle. By comparing the homogeneous components in $[u_j] = \sum_i [u_{j,i}]$, we get $[u_j] = [u_{j,l}]$ and $[u_{j,i}] = 0$ if $i \neq l$. Note that $\deg_x u_{j,l} \leq n < l= \deg u_{j,l}$. So, by Lemma \ref{lemma-degree-a}, $u_{j,l} = a v$ for some $v \in C_P(\Gamma)$. It is easy to see that $v$ is a homogeneous cycle in $C_P(\Gamma)$ and that $\{[v]\}\cup \{u_i~|~i\neq j\}$ spans $H_P(\Gamma)$ and is $\C[a]$-linearly independent. In other words, $\{[v]\}\cup \{[u_i]~|~i\neq j\}$ is also a basis for $H_P(\Gamma)$. But this is impossible because, if this is true, then the determinant of the change-of-coordinates matrix from the basis $\{[u_j]\}$ to the basis $\{[v]\}\cup \{[u_i]~|~i\neq j\}$ is $a$, which is not invertible in $\C[a]$. 
\end{proof}

\begin{definition}\label{def-pi-0}
Denote by $\pi_a:C_P(\Gamma) \rightarrow C_N(\Gamma) (= C_P(\Gamma)/aC_P(\Gamma))$ the standard quotient map. To keep notations simple, we denote again by $\pi_a:H_P(\Gamma) \rightarrow H_N(\Gamma)$ the homomorphism induced by the quotient map $\pi_a$. 
\end{definition}

Recall that $C_N(\Gamma)$ inherits the polynomial grading of $C_P(\Gamma)$ via $\pi_a$, which makes $\pi_a$ a homogeneous map of degree $0$. Moreover, $C_N(\Gamma)$ also inherits the $x$-filtration of $C_P(\Gamma)$ via $\pi_a$. It is easy to see the $x$-filtration of $C_N(\Gamma)$ is the increasing filtration induced by its polynomial grading.

\begin{lemma}\label{lemma-basis-pi-0}
$\pi_a:H_P(\Gamma) \rightarrow H_N(\Gamma)$ is a surjective homogeneous homomorphism with $\deg \pi_a =0$ and $\ker \pi_a = a H_P(\Gamma)$. Moreover, any homogeneous basis for the free $\C[a]$-module $H_P(\Gamma)$ is mapped by $\pi_a$ to a homogeneous basis for the $\C$-space $H_N(\Gamma)$.
\end{lemma}

\begin{proof}
By its definition, we know that $\pi_a:H_P(\Gamma) \rightarrow H_N(\Gamma)$ is a homogeneous homomorphism with $\deg \pi_a =0$. The short exact sequence $0\rightarrow C_P(\Gamma) \xrightarrow{a} C_P(\Gamma) \xrightarrow{\pi_a} C_N(\Gamma) \rightarrow 0$ induces a long exact sequence 
\[
\cdots\rightarrow H_N^{\ve+1}(\Gamma) \rightarrow H_P^{\ve}(\Gamma) \xrightarrow{a} H_P^\ve(\Gamma) \xrightarrow{\pi_a} H_N^\ve(\Gamma) \rightarrow H_P^{\ve+1}(\Gamma) \xrightarrow{a} H_P^{\ve+1}(\Gamma) \xrightarrow{\pi_a} H_N^{\ve+1}(\Gamma)  \rightarrow \cdots
\]
with $\zed_2$ homological grading, where $\ve$ is the rotation number of $\Gamma$. From \cite{KR1}, we know that $H_N^{\ve+1}(\Gamma)\cong 0$. By Lemma \ref{lemma-H_P-Gamma-fil}, we know that $H_P^{\ve+1}(\Gamma) \cong 0$. So the above long exact sequence becomes a short exact sequence
\[
0\rightarrow H_P(\Gamma) \xrightarrow{a} H_P(\Gamma) \xrightarrow{\pi_a} H_N(\Gamma) \rightarrow 0.
\] 
It follows from this that $\pi_a:H_P(\Gamma) \rightarrow H_N(\Gamma)$ is surjective and $\ker \pi_a = a H_P(\Gamma)$. The statement about bases follows then from Lemma \ref{lemma-quotient-basis}.
\end{proof}

\begin{lemma}\label{lemma-a-fixes-filtration}
For any $[u]\in H_P(\Gamma)$, $[u] \in \fil_x^n H_P(\Gamma)$ if and only if $a[u] \in \fil_x^n H_P(\Gamma)$.
\end{lemma}

\begin{proof}
Since the map $C_P(\Gamma) \xrightarrow{a}C_P(\Gamma)$ preserves the $x$-filtration, so does the map $H_P(\Gamma) \xrightarrow{a} H_P(\Gamma)$. Therefor, $a[u] \in \fil_x^n H_P(\Gamma)$ if $[u] \in \fil_x^n H_P(\Gamma)$.

Now assume $[u] \notin \fil_x^n H_P(\Gamma)$. Since the $x$-filtration is exhaustive, there is an $l>n$ such that  $[u] \in \fil_x^l H_P(\Gamma)$ and  $[u] \notin \fil_x^{l-1} H_P(\Gamma)$. Denote by $\ve$ the rotation number of $\Gamma$. Recall that, by Lemma \ref{lemma-H_P-Gamma-fil}, we have $H_P^{\ve+1}(\Gamma)=0$. Moreover, in the proof of Lemma \ref{lemma-H_P-Gamma-fil}, we constructed an isomorphism $\phi_l: H_N^{\ve,l}(\Gamma)\otimes_\C \C[a] \rightarrow \fil^l_x H_P^\ve(\Gamma) / \fil^{l-1}_x H_P^\ve(\Gamma)$ of $\C[a]$-modules. So $[u]\in H_P^\ve(\Gamma)$ and $\phi_l^{-1}([u]) \neq 0$. But $ H_N^{\ve,l}(\Gamma)\otimes_\C \C[a]$ is a free $\C[a]$-module. So $\phi_l^{-1}(a[u])= a\phi_l^{-1}([u]) \neq 0$. Thus, $a[u] \notin \fil_x^{l-1} H_P(\Gamma)$ and, therefore, $a[u] \notin \fil_x^n H_P(\Gamma)$.
\end{proof}

\begin{lemma}\label{lemma-homogeneous-basis-fil}
Let $\{[u_j]\}$ be a homogeneous basis for $H_P(\Gamma)$. For any $\{f_j\}\subset \C[a]$ and any $l \in \zed$, $\sum_j f_j [u_j] \in \fil_x^l H_P(\Gamma)$ if and only if $f_j=0$ whenever $\deg [u_j]>l$.
\end{lemma}

\begin{proof}
By Lemma \ref{lemma-homology-basis-degree}, we have $\deg_x [u_j] = \deg [u_j]$ for every $j$. If $f_j=0$ whenever $\deg [u_j]>l$, then, by Lemma \ref{lemma-a-fixes-filtration}, we know that $\sum_j f_j [u_j] \in \fil_x^l H_P(\Gamma)$.

Now assume $\sum_j f_j [u_j] \in \fil_x^l H_P(\Gamma)$. We prove by contradiction that $f_j=0$ whenever $\deg [u_j]>l$. Assuming the conclusion is not true, then $n:=\{\deg [u_j] ~|~ f_j\neq 0\} > l$. Without loss of generality, we assume
\[
n_j:=\deg_x [u_j]=\deg [u_j] \begin{cases}
=n & \text{if } 1\leq j \leq p, \\
<n & \text{if } p+1\leq j \leq p+q, \\
>n & \text{otherwise.}
\end{cases}
\]
By the definition of $n$, we know that $f_j=0$ unless $1\leq j\leq p+q$. For $1\leq j\leq p+q$, write $f_j=\sum_{i\geq0}c_{j,i}a^i$. Define $[v_i]= \sum_{j=1}^p c_{j,i} [u_j] +\sum_{j=p+1}^{p+q} c_{j,n+i-n_j} a^{n-n_j} [u_j]$. Then the homogeneous component of $\sum_j f_j [u_j]$ of polynomial degree $n+i$ is $a^i[v_i]$. By Lemma \ref{lemma-homology-components-degree}, we have $a^i[v_i]\in \fil_x^l H_P(\Gamma)$. Therefore, by Lemma \ref{lemma-a-fixes-filtration}, we have $[v_i]\in \fil_x^l H_P(\Gamma)$. Consider $\pi_a([v_i])= \sum_{j=1}^p c_{j,i} \pi_a([u_j]) \in H_N(\Gamma)$. On the one hand, we have that $\pi_a([v_i]) \in \fil_x^l H_N(\Gamma)$. On the other hand, we know that $\deg \pi_a([u_1])=\cdots=\pi_a([u_p])=n>l$. But the $x$-filtration on $H_N(\Gamma)$ is induced by the polynomial grading. We must have $\pi_a([v_i])= \sum_{j=1}^p c_{j,i} \pi_a([u_j])=0$. Lemma \ref{lemma-basis-pi-0} tells us that $\{\pi_a([u_1]),\dots,\pi_a([u_p])\}$ is linearly independent. So $c_{j,i}=0$ for $i\geq 0$ and $1\leq j\leq p$. In other words, $f_j=0$ for $1\leq j\leq p$. This contradicts the definition of $n$.
\end{proof}

\begin{definition}\label{def-pi-1}
Denote by $\pi_{a-1}:C_P(\Gamma) \rightarrow \hat{C}_P(\Gamma) (= C_P(\Gamma)/(a-1)C_P(\Gamma))$ the standard quotient map. To keep notations simple, we denote again by $\pi_{a-1}:H_P(\Gamma) \rightarrow \hat{H}_P(\Gamma)$ the homomorphism induced by the quotient map $\pi_{a-1}$. 
\end{definition}

Note that $\hat{C}_P(\Gamma)$ inherits the $x$-filtration of $C_P(\Gamma)$ through $\pi_{a-1}$.

\begin{lemma}\label{lemma-homogeneous-basis-fil-pi-1}
$\pi_{a-1}:H_P(\Gamma) \rightarrow \hat{H}_P(\Gamma)$ is a surjective homomorphism preserving the $x$-filtration with $\ker \pi_{a-1} = (a-1)H_P(\Gamma)$. Moreover, for any homogeneous basis $\{[u_j]\}$ for the free $\C[a]$-module $H_P(\Gamma)$, we have:
\begin{itemize}
	\item $\{\pi_{a-1}([u_j])\}$ is a basis for the $\C$-space $\hat{H}_P(\Gamma)$.
	\item $\deg_x \pi_{a-1}([u_j]) = \deg [u_j]$ for every $j$.
	\item For any $\{c_j\} \subset \C$, $\sum_j c_j \pi_{a-1}([u_j]) \in \fil_x^l\hat{H}_P(\Gamma)$ if and only if $c_j=0$ whenever $\deg [u_j] >l$.
\end{itemize}
\end{lemma}

\begin{proof}
By its definition, we know that $\pi_{a-1}:H_P(\Gamma) \rightarrow \hat{H}_P(\Gamma)$ preserves the $x$-filtration. The short exact sequence $0\rightarrow C_P(\Gamma) \xrightarrow{a-1} C_P(\Gamma) \xrightarrow{\pi_{a-1}} \hat{C}_P(\Gamma) \rightarrow 0$ induces a long exact sequence 
\[
\cdots\rightarrow \hat{H}_P^{\ve+1}(\Gamma) \rightarrow H_P^{\ve}(\Gamma) \xrightarrow{a-1} H_P^\ve(\Gamma) \xrightarrow{\pi_{a-1}} \hat{H}_P^\ve(\Gamma) \rightarrow H_P^{\ve+1}(\Gamma) \xrightarrow{a-1} H_P^{\ve+1}(\Gamma) \xrightarrow{\pi_{a-1}} \hat{H}_P^{\ve+1}(\Gamma)  \rightarrow \cdots
\]
with $\zed_2$ homological grading, where $\ve$ is the rotation number of $\Gamma$. By Lemmas \ref{lemma-hat-H_P-Gamma-fil} and \ref{lemma-H_P-Gamma-fil}, we know that $\hat{H}_P^{\ve+1}(\Gamma)\cong 0$ and $H_P^{\ve+1}(\Gamma)\cong0$. So the above long exact sequence becomes a short exact sequence
\[
0\rightarrow H_P(\Gamma) \xrightarrow{a-1} H_P(\Gamma) \xrightarrow{\pi_{a-1}} \hat{H}_P(\Gamma) \rightarrow 0.
\] 
Thus, $\pi_{a-1}:H_P(\Gamma) \rightarrow \hat{H}_P(\Gamma)$ is surjective and $\ker \pi_{a-1} = (a-1) H_P(\Gamma)$.

Now assume $\{[u_j]\}$ is a homogeneous basis for the free $\C[a]$-module $H_P(\Gamma)$ with $\deg [u_j] = n_j$. It follows from the above and Lemma \ref{lemma-quotient-basis} that $\{\pi_{a-1}([u_j])\}$ is a basis for the $\C$-space $\hat{H}_P(\Gamma)$. 

Since $\pi_{a-1}:H_P(\Gamma) \rightarrow \hat{H}_P(\Gamma)$ preserves the $x$-filtration, we get from Lemma \ref{lemma-homology-basis-degree} that $\deg_x \pi_{a-1}([u_j]) \leq \deg_x [u_j] = \deg [u_j] =n_j$. Next, we prove that $\deg_x \pi_{a-1}([u_j])=n_j$. Note that $[u_j]$ is represented by a homogeneous cycle $u_j$ in $C_P(\Gamma)$ with $\deg u_j = n_j$. Recall that the $x$-filtration on $C_P(\Gamma)$ is the increasing filtration associated to an $x$-grading on $C_P(\Gamma)$. Denote by $u_{j,i}$ the homogeneous component of $u_j$ with respect to this $x$-grading. Then $u_{j,i}=0$ if $i> n_j$ and, by Lemma \ref{lemma-degree-a}, $u_{j,i} \in aC_P(\Gamma)$ if $i<n_j$. So $\pi_a(u_{j,n_j})=\pi_a(u_j)$ is a homogeneous cycle in $C_N(\Gamma)$ of polynomial degree $n_j$ representing the homology class $\pi_a([u_j])$. Lemma \ref{lemma-basis-pi-0} implies that $\pi_a([u_j])\neq 0$. Recall that Lemma \ref{lemma-hat-H_P-Gamma-fil} is proved in \cite{Wu7} by a construction very similar to the proof of Lemma \ref{lemma-H_P-Gamma-fil}. To summarize, we know that 
\begin{itemize}
	\item $\hat{H}_P^{\ve+1}(\Gamma) \cong H_N^{\ve+1}(\Gamma)\cong 0$,
	\item every homogeneous cycle in $C_N^{\ve}(\Gamma)$ can be completed to a cycle in $\hat{C}_P^\ve(\Gamma)$ by adding terms with strictly lower polynomial degrees, and this correspondence gives rise to a well defined isomorphism $\hat{\phi}_n: H_N^{\ve,n}(\Gamma) \rightarrow \fil_x^n \hat{H}_P^\ve (\Gamma) / \fil_x^{n-1} \hat{H}_P^\ve (\Gamma)$.
\end{itemize}
Clearly, $\pi_{a-1}(u_j)$ is a cycle in $\hat{C}_P^\ve(\Gamma)$ obtained from $\pi_a(u_{j,n_j})$ by adding terms with polynomial degrees strictly less than $\deg \pi_a(u_{j,n_j}) = n_j$. Denote by $\pi^{(n)}: \fil_x^{n} \hat{H}_P^\ve (\Gamma) \rightarrow \fil_x^{n} \hat{H}_P^\ve (\Gamma) / \fil_x^{n-1} \hat{H}_P^\ve (\Gamma)$ the standard quotient map. We have a commutative diagram
\begin{equation}\label{diagram-quotients-commute}
\xymatrix{
H_P^{\ve,n}(\Gamma) \ar[rr]^{\pi_{a-1}} \ar[d]^{\pi_a} && \fil_x^n\hat{H}_P^{\ve}(\Gamma) \ar[d]^{\pi^{(n)}} \\
H_N^{\ve,n}(\Gamma) \ar[rr]^<<<<<<<<<<{\hat{\phi}_n}_<<<<<<<<<<{\cong} && \fil_x^n\hat{H}_P^{\ve}(\Gamma)/\fil_x^{n-1}\hat{H}_P^{\ve}(\Gamma)
},
\end{equation}
where $H_P^{\ve,n}(\Gamma)$ (resp. $H_N^{\ve,n}(\Gamma)$) is the component of $H_P(\Gamma)$ (resp. $H_N(\Gamma)$) with $\zed_2$-degree $\ve$ and polynomial degree $n$. Thus, $\pi^{(n_j)} (\pi_{a-1}([u_j]))= \hat{\phi}_{n_j}([\pi_a(u_{j,n_j})]) = \hat{\phi}_{n_j}(\pi_a([u_j]))\neq 0$. So $\pi_{a-1}([u_j]) \notin \fil_x^{n_j-1} \hat{H}_P^\ve (\Gamma)$ and, therefore, $\deg_x \pi_{a-1}([u_j])=n_j$. 

It remains to show that $\sum_j c_j \pi_{a-1}([u_j]) \in \fil_x^l\hat{H}_P(\Gamma)$ if and only if $c_j=0$ whenever $n_j >l$. Recall that $\deg_x \pi_{a-1}([u_j])=n_j$. If $c_j=0$ whenever $n_j >l$, then we clearly have that $\sum_j c_j \pi_{a-1}([u_j]) \in \fil_x^l\hat{H}_P(\Gamma)$. Now assume that there is at least one $j$ such that $n_j>l$ and $c_j \neq 0$. Then $n:=\max \{n_j~|~c_j\neq 0\}>l$. Consider $[u]:=\sum_j c_j a^{n-n_j} [u_j]$. $[u]$ is a homogeneous element of $H_P(\Gamma)$ of polynomial degree $n$, and $\pi_{a-1}([u]) = \sum_j c_j \pi_{a-1}([u_j])$. From diagram \eqref{diagram-quotients-commute}, we have
\[
\pi^{(n)} (\sum_j c_j \pi_{a-1}([u_j])) = \pi^{(n)} (\pi_{a-1}([u])) = \hat{\phi}_n (\pi_a([u])) = \hat{\phi}_n( \sum_{n_j=n} c_j \pi_a([u_j])). 
\]
By Lemma \ref{lemma-basis-pi-0}, $\{\pi_a([u_j])\}$ is a basis for $H_N(\Gamma)$. By the definition of $n$, we know that $n_j=n$ and $c_j\neq 0$ for at least one $j$. Thus, $\sum_{n_j=n} c_j \pi_a([u_j]) \neq 0$. But $\hat{\phi}_n: H_N^{\ve,n}(\Gamma) \rightarrow \fil_x^n \hat{H}_P^\ve (\Gamma) / \fil_x^{n-1} \hat{H}_P^\ve (\Gamma)$ is an isomorphism. So $\pi^{(n)} (\pi_{a-1}([u]))=\hat{\phi}_n( \sum_{n_j=n} c_j \pi_a([u_j])) \neq 0$. Thus $\sum_j c_j \pi_{a-1}([u_j]) = \pi_{a-1}([u]) \notin \fil_x^{l} \hat{H}_P (\Gamma) (\subset \fil_x^{n-1} \hat{H}_P (\Gamma))$.
\end{proof}

\begin{remark}
In conclusion of this subsection, we note that Lemmas \ref{lemma-homogeneous-basis-fil} and \ref{lemma-homogeneous-basis-fil-pi-1} imply that any direct sum decomposition of $H_P(\Gamma)$ in the category of graded $\C[a]$-modules is also a direct sum decomposition in the category of filtered $\C$-spaces and induces a direct sum decomposition of $\hat{H}_P(\Gamma)$ in the category of filtered $\C$-spaces.
\end{remark}

\subsection{Decomposition of $H_P(L)$}\label{subsec-homology-decomp} In this subsection, we give a proof of Lobb's decomposition theorem (Theorem \ref{thm-H_P-decomp}.)

As before, $a$ is a homogeneous variable of degree $2k$.

\begin{lemma}\cite{Lobb-2-twists}\label{lemma-graded-chain-decomp}
Assume that $(C^\ast,d^\ast)= \cdots\xrightarrow{d^{n-1}} C^n \xrightarrow{d^{n}} C^{n+1} \xrightarrow{d^{n+1}}\cdots$ is a bounded chain complex of finitely generated graded-free $\C[a]$-module and its differential $d^\ast$ preserves the grading of $C^\ast$. Then, in the category of chain complexes of graded $\C[a]$-modules, $(C^\ast,d^\ast)$ is a direct sum of chain complexes of the forms $F_{i,s}$ and $T_{i,m,s}$ given in \eqref{def-F-i-s}--\eqref{def-T-i-m-s}.
\end{lemma}

\begin{proof}
We prove the lemma by an induction on the total rank of $C^\ast$. If $\rank C^\ast = 0$ or $1$, then the lemma is trivially true. Assume $\rank C^\ast = K \geq 2$ and the lemma is true if the rank of the chain complex is less than $K$. There is an $n$ such that $C^n\neq 0$ and $C^j=0$ $\forall~j<n$. Consider the section $0\rightarrow C^n \xrightarrow{d^{n}} C^{n+1} \xrightarrow{d^{n+1}}\cdots$ of $C^\ast$. Let $\{u_1,\dots,u_p\}$ be a homogeneous basis for $C^n$ with $\deg u_1 \leq \cdots\leq \deg u_p$, and $\{v_1,\dots,v_q\}$ a homogeneous basis for $C^{n+1}$ with $\deg v_1 \geq \cdots \geq \deg v_q$. For each $1\leq j \leq p$, we have $d^{n}(u_j) = \sum _{i=1}^q f_{i,j} v_i$, where each $f_{i,j}$ is a monomial in $a$ of degree $\deg f_{i,j} = \deg u_j -\deg v_i$. Note that $\deg f_{i,j}$ is increasing with respect to both $i$ and $j$. 

If $(f_{1,1},\dots,f_{q,1})=0$, then $C^\ast$ has a direct sum component $0\rightarrow \C[a]\cdot u_1 \rightarrow 0 \cong F_{n,\deg u_1}$. Thus, by induction hypothesis, the lemma is true for $C^\ast$. 

If $(f_{1,1},\dots,f_{q,1})\neq 0$, then there is an $l$ such that $f_{l,1} \neq 0$ and $f_{i,1}=0$ $\forall~1\leq i<l$. Define a $q\times q$ matrix $\Xi=(\xi_{i,j})$ by
\[
\xi_{i,j} = \begin{cases}
1& \text{if } i=j, \\
\frac{f_{i,1}}{f_{l,1}} & \text{if } i>l \text{ and } j=l, \\
0 & \text{otherwise.}
\end{cases}
\]
Then $\Xi$ is invertible and $\Xi^{-1} = (\tilde{\xi}_{i,j})$ is given by 
\[
\tilde{\xi}_{i,j} = \begin{cases}
1& \text{if } i=j, \\
-\frac{f_{i,1}}{f_{l,1}} & \text{if } i>l \text{ and } j=l, \\
0 & \text{otherwise.}
\end{cases}
\]
Note that, for any $i>l$, $\frac{f_{i,1}}{f_{l,1}}$ is a monomial of $a$ of degree $\deg f_{i,1} - \deg f_{l,1}= \deg v_l -\deg v_i$. We have
\[
(d^n u_1,\dots, d^n u_p) = (v_1,\dots,v_q)(f_{i,j})= (v_1,\dots,v_q)\Xi \Xi^{-1}(f_{i,j}).
\]
Let $\tilde{v}_j = \sum_{i=1}^q \xi_{i,j}v_i$ and $g_{i,j} = \sum_{\alpha=1}^q \xi_{i,\alpha} f_{\alpha,j}$. Then $\{\tilde{v}_1,\dots,\tilde{v}_q\}$ is a homogeneous basis for $C^{n+1}$ with $\deg \tilde{v}_i= \deg v_i$, $g_{i,j}$ is a monomial in $a$ with $\deg g_{i,j} = \deg f_{i,j}$, and $(d^n u_1,\dots, d^n u_p) = (\tilde{v}_1,\dots,\tilde{v}_q)(g_{i,j})$.

Note that $g_{l,1} = f_{l,1} \neq 0$ and $f_{i,1}=0$ $\forall ~ i \neq l$. Next, define a $p \times p$ matrix $\Theta= (\theta_{i,j})$ by
\[
\theta_{i,j} = \begin{cases}
1& \text{if } i=j, \\
-\frac{g_{l,j}}{g_{l,1}} & \text{if } i=1 \text{ and } j>1, \\
0 & \text{otherwise,}
\end{cases}
\]
where $\frac{g_{l,j}}{g_{l,1}}$ is a monomial of $a$ of degree $\deg g_{l,j} - \deg g_{l,1} = \deg u_j -\deg u_1$. Let $\tilde{u}_j = \sum_{i=1}^p \theta_{i,j} u_i$ and $h_{i,j} = \sum_{\alpha=1}^p g_{i,\alpha} \theta_{\alpha,j}$. Then $\{\tilde{u}_1,\dots,\tilde{u}_p\}$ is a homogeneous basis for $C^{n}$ with $\deg \tilde{u}_j = \deg u_j$, $h_{i,j}$ is a monomial in $a$ with $\deg h_{i,j}=\deg g_{i,j} =\deg f_{i,j}$, and $(d^n \tilde{u}_1,\dots, d^n \tilde{u}_p) = (\tilde{v}_1,\dots,\tilde{v}_q)(h_{i,j})$.

Note that $h_{l,1}=g_{l,1}=f_{l,1}\neq 0$, $h_{i,1}=0$ $\forall ~i \neq l$ and $h_{l,j}=0$ $\forall~ j \neq 1$. Thus, $C^\ast$ has a direct sum component 
\[
0\rightarrow \C[a]\cdot \tilde{u}_1 \xrightarrow{h_{l,1}} \C[a] \cdot \tilde{v}_l \rightarrow 0 \cong T_{n+1,\frac{\deg \tilde{u}_1 - \deg \tilde{v}_l}{2k},\deg \tilde{v}_l}.
\] 
By the induction hypothesis, the lemma is true for $C^\ast$.
\end{proof}

The existence of the decomposition in Theorem \ref{thm-H_P-decomp} follows from Lemma \ref{lemma-graded-chain-decomp}. To prove the uniqueness of this decomposition, we need the following lemma, which is a slight refinement of the standard invariance theorem for modules over a principal ideal domain.

\begin{lemma}\label{lemma-graded-module-decomp-unique}
Suppose that $\{(m_1,s_1),\dots,(m_p,s_p)\}$ and $\{(n_1,t_1),\dots,(n_q,t_q)\}$ are two sequences in $\zed_{>0}\times \zed$ satisfying:
\begin{itemize}
	\item $m_1\leq\cdots\leq m_p$, $n_1\leq\cdots\leq n_q$.
	\item If $i<j$ and $m_i=m_j$, then $s_i \leq s_j$.
	\item If $i<j$ and $n_i=n_j$, then $t_i \leq t_j$.
	\item As graded $\C[a]$-modules,
\[
\bigoplus_{i=1}^p (\C[a]/(a^{m_i}))\{s_i\} \cong \bigoplus_{j=1}^p (\C[a]/(a^{n_j}))\{t_j\}.
\]
\end{itemize}
Then $p=q$ and $m_i=n_i$, $s_i=t_i$ for every $1 \leq i \leq p$.
\end{lemma}

\begin{proof}
We adapt the proof of the invariance theorem in \cite[Section 3.9]{Jacobson-Basic-algebra-I} to prove Lemma \ref{lemma-graded-module-decomp-unique}. The only change is that, instead of counting dimensions, we count graded dimensions. Recall that, for a finite dimensional graded $\C$-space $V=\bigoplus_i V_i$, where $V_i$ is the homogeneous component of $V$ of degree $i$, the graded dimension of $V$ is $\gdim_\C V := \sum_i \beta^i \dim_\C V_i$, where $\beta$ is a homogeneous variable of degree $1$.

Let
\[
M:=\bigoplus_{i=1}^p (\C[a]/(a^{m_i}))\{s_i\} \cong \bigoplus_{j=1}^p (\C[a]/(a^{n_j}))\{t_j\}.
\]
Denote by $z_i$ the multiplicative unit $1$ in $(\C[a]/(a^{m_i}))\{s_i\}$, which is a homogeneous element of $M$ of degree $s_i$. For any $l \geq 0$, define $M^{(l)}:=a^l M/a^{l+1}M$. Then $M^{(l)}$ is a finite dimensional graded $\C$-space. If $l \geq m_p$, then $M^{(l)}=0$ and $\gdim_\C M^{(l)}=0$. If $0\leq l < m_p$, then there is a unique $j$ such that $m_1\leq\cdots\leq m_j \leq l <m_{j+1}\leq \cdots\leq m_p$. One can see that $a^lM= \bigoplus_{i=j+1}^p \C[a] a^lz_i$ and $\{a^lz_{j+1}+a^{l+1}M,\dots, a^lz_{p}+a^{l+1}M\}$ is a homogeneous basis for the $\C$-space $M^{(l)}$. So $\gdim_\C M^{(l)} = \beta^{2kl} \sum_{i=j+1}^p \beta^{s_i}$. For any integer $c$, define $S_{l,c} = \{i~|~ 1\leq i\leq p, ~s_i=c,~m_i>l\}$. We observe that the coefficient of $\beta^{c+2kl}$ in $\gdim_\C M^{(l)}$ is equal to the cardinality of $S_{l,c}$. Similarly, defining $S'_{l,c} = \{i~|~ 1\leq i\leq q, ~t_i=c,~n_i>l\}$, we have that the coefficient of $\beta^{c+2kl}$ in $\gdim_\C M^{(l)}$ is equal to the cardinality of $S'_{l,c}$. Thus, for any $(l,c) \in \zed_{\geq 0}\times \zed$, the cardinalities of $S_{l,c}$ and $S'_{l,c}$ are equal. The lemma follows from this.
\end{proof}

It is now very easy to prove Theorem \ref{thm-H_P-decomp} and Corollary \ref{cor-H_N-decomp}.

\begin{proof}[Proof of Theorem \ref{thm-H_P-decomp} and Corollary \ref{cor-H_N-decomp}]
Fix a diagram $D$ of $L$ and a marking of $D$. Let $x_1,\dots,x_m$ be the variables assigned to the marked points of $\Gamma$ and $R=\C[x_1,\dots,x_m,a]$. By Corollary \ref{cor-graph-homology-free}, $H(C_P(D),d_{mf})$ with its polynomial grading is a finitely generated graded-free $\C[a]$-module. By the definition of $d_\chi$, we know that it preserves the polynomial grading.

According to Lemma \ref{lemma-graded-chain-decomp}, in the category of chain complexes of graded $\C[a]$-modules, $(H(C_P(D),d_{mf}),d_\chi)$ decomposes into a direct sum of chain complexes of the forms $F_{i,s}$ and $T_{i,m,s}$ given in \eqref{def-F-i-s}--\eqref{def-T-i-m-s}. Note that each factor of $F_{i,s}$ in this decomposition contributes a direct sum component $\C[a]\|i\|\{s\}$ to $H_P(L)$ and each factor of $T_{i,m,s}$ contributes a direct sum component $(\C[a]/(a^m))\|i\|\{s\}$ to $H_P(L)$. To prove the existence of decomposition \eqref{eq-H_P-decomp}, it remains to determine the free part of $H_P(L)$. By Lemma \ref{lemma-homogeneous-basis-fil-pi-1}, the above decomposition of $(H(C_P(D),d_{mf}),d_\chi)$ induces a decomposition of $(H(\hat{C}_P(D),d_{mf}),d_\chi)$ in the category of filtered chain complexes of $\C$-spaces. Each factor of $F_{i,s}$ (resp. $T_{i,m,s}$) in the decomposition of $(H(C_P(D),d_{mf}),d_\chi)$ corresponds to a factor of $\hat{F}_{i,s}$ (resp. $\hat{T}_{i,m,s}$) in the decomposition of $(H(\hat{C}_P(D),d_{mf}),d_\chi)$, and the grading of $F_{i,s}$ (resp. $T_{i,m,s}$) induces the filtration on the corresponding $\hat{F}_{i,s}$ (resp. $\hat{T}_{i,m,s}$). The homology of $\hat{T}_{i,m,s}$ vanishes and, as a filtered $\C$-space, the homology of $\hat{F}_{i,s}$ is $\C\|i\|\{s\} \cong (\C[a]/(a-1))\|i\|\{s\}$. One can see from this that, as a graded $\C[a]$-module, the free part of $H_P(L)$ is isomorphic to $\hat{\mathcal{H}}_P(L) \otimes_\C \C[a]$. Thus, we have proved the existence of decomposition \eqref{eq-H_P-decomp}. 
 
The uniqueness of decomposition \eqref{eq-H_P-decomp} follows from Lemma \ref{lemma-graded-module-decomp-unique}. This completes the proof of Theorem \ref{thm-H_P-decomp}.

To prove Corollary \ref{cor-H_N-decomp}, note that, by Lemma \ref{lemma-basis-pi-0}, $(H(C_N(D),d_{mf}), d_\chi) \cong H(C_P(D),d_{mf})/aH(C_P(D),d_{mf})$ as chain complexes of graded $\C$-spaces. Moreover, as chain complexes of graded $\C$-spaces, 
\begin{eqnarray*}
F_{i,s}/aF_{i,s} & \cong & 0\rightarrow \C\|i\| \rightarrow 0, \\
T_{i,m,s}/aT_{i,m,s} & \cong & \begin{cases}
 0\rightarrow \C\|i-1\|\{s+2km\} \xrightarrow{0} \C\|i\|\{s\} \rightarrow 0 & \text{if } m\geq 1, \\
 0\rightarrow \C\|i-1\|\{s+2km\} \xrightarrow{1} \C\|i\|\{s\} \rightarrow 0 & \text{if } m=0.
 \end{cases}
\end{eqnarray*}
So Corollary \ref{cor-H_N-decomp} also follows from the decomposition of the chain complex $(H(C_P(D),d_{mf}),d_\chi)$.
\end{proof}

\subsection{Decompositions of $E_r(L)$ and $\hat{E}_r(L)$}\label{subsec-spectral-sequence-decomp}

In this subsection, we prove the decompositions of $E_r(L)$ and $\hat{E}_r(L)$.

First, we compute the spectral sequences of the filtered chain complexes $F_{i,s}$, $T_{i,m,s}$, $\hat{F}_{i,s}$ and $\hat{T}_{i,m,s}$. We use the notations given in \cite[Section 2.2]{McCleary-users-guide-to-SS}. Since $\fil_x$ is an increasing filtration and the notations in  \cite[Section 2.2]{McCleary-users-guide-to-SS} are for a decreasing filtration, we need to adjust their definitions accordingly. 

Let $(C^\ast, d,\fil)$ be a filtered chain complex such that $d$ raises the homological grading by $1$ and $\fil$ is increasing. Set
\begin{eqnarray}
\label{eq-def-Z-SS} Z_r^{p,q} & = & \fil^p C^{p+q} \cap d^{-1}( \fil^{p-r} C^{p+q+1}), \\
\label{eq-def-B-SS} B_r^{p,q} & = & \fil^p C^{p+q} \cap d(\fil^{p+r} C^{p+q-1}), \\
\label{eq-def-E-SS} E_r^{p,q} & = & Z_r^{p,q}/(Z_{r-1}^{p-1,q+1} + B_{r-1}^{p,q}).
\end{eqnarray}
Then $\{E_r^{p,q}\}$ is the spectral sequence of $(C^\ast, d,\fil)$.

\begin{proof}[Proof of Lemma \ref{lemma-ss-components}]
We only compute $E_r(F_{i,s})$ and $E_r(T_{i,m,s})$. The computation of $E_r(\hat{F}_{i,s})$ and $E_r(\hat{T}_{i,m,s})$ is very similar and left to the reader.

Recall that the differential map of $F_{i,s}$ is $0$. So $B_r^{p,q}(F_{i,s})=0$ and 
\[
Z_r^{p,q}(F_{i,s}) = \fil^p F_{i,s}^{p+q} = \begin{cases}
\C[a]\{s\} & \text{if } p\geq s, ~q=i-p, \\
0 & \text{otherwise.} 
\end{cases}
\]
So
\[
E_r^{p,q}(F_{i,s}) = \fil^p F_{i,s}^{p+q} / \fil^{p-1} F_{i,s}^{p+q} \cong \begin{cases}
\C[a]\{s\} & \text{if } p=s,~q=i-s, \\
0 & \text{otherwise.}
\end{cases}
\]
The observations about $\{E_r(F_{i,s})\}$ in Lemma \ref{lemma-ss-components} follow from this.

Recall that the filtered chain complex $T_{i,m,s}$ is given by 
\[
\fil_x^p T_{i,m,s} = \begin{cases}
0\rightarrow \C[a]\|i-1\|\{s+2km\} \xrightarrow{a^m} \C[a]\|i\|\{s\} \rightarrow 0 & \text{if } p\geq s+2km, \\
0\rightarrow \C[a]\|i\|\{s\} \rightarrow 0 & \text{if } s \leq p< s+2km, \\
0 & \text{if }  p<s.
\end{cases}
\]
Note that $Z_r^{p,q}(T_{i,m,s})=0$ unless $q=i-1-p$ or $i-p$. So $E_r^{p,q}(T_{i,m,s})=0$ unless $q=i-1-p$ or $i-p$.

We compute $E_r^{p,i-1-p}(T_{i,m,s})$ first. Note that
\[
\fil_x^p T_{i,m,s}^{i-1} = \begin{cases}
\C[a]\{s+2km\} & \text{if } p\geq s+2km, \\
0 & \text{if } p<s+2km,
\end{cases}
\]
\[
d^{-1} (\fil_x^{p-r} T_{i,m,s}^{i}) = \begin{cases} 
\C[a]\{s+2km\} & \text{if } p-r \geq s, \\
0 & \text{if } p-r < s.
\end{cases}
\]
Thus,
\[
Z_r^{p,i-1-p}(T_{i,m,s}) = \fil_x^p T_{i,m,s}^{i-1} \cap d^{-1} (\fil_x^{p-r} T_{i,m,s}^{i}) = \begin{cases} 
\C[a]\{s+2km\} & \text{if }  p\geq s+2km \text{ and } p \geq s+r \\
0 & \text{otherwise.}
\end{cases}
\]
Also, $B_r^{p,i-1-p}(T_{i,m,s}) = \fil_x^p T_{i,m,s}^{i-1} \cap d (\fil_x^{p+r} T_{i,m,s}^{i-2}) = 0$. Therefore,
\[
E_r^{p,i-1-p}(T_{i,m,s}) = Z_r^{p,i-1-p}(T_{i,m,s})/ Z_r^{p-1,i-p}(T_{i,m,s}) \cong \begin{cases} 
\C[a]\{s+2km\} & \text{if } p= s+2km \geq s+r \\
0 & \text{otherwise.}
\end{cases}
\]

Next, we compute $E_r^{p,i-p}(T_{i,m,s})$. Note that 
\[
Z_r^{p,i-p}(T_{i,m,s}) = \fil_x^p T_{i,m,s}^{i} \cap d^{-1} (\fil_x^{p-r} T_{i,m,s}^{i+1}) = \fil_x^p T_{i,m,s}^{i} = \begin{cases} 
\C[a]\{s\} & \text{if } p\geq s, \\
0 & \text{if } p< s,
\end{cases}
\]
and 
\[
d (\fil_x^{p+r} T_{i,m,s}^{i-1}) = \begin{cases}
a^m \C[a]\{s\} & \text{if } p+r \geq s+2km, \\
0 & \text{if } p+r < s+2km.
\end{cases}
\]
So
\[
B_r^{p,i-p}(T_{i,m,s}) = \begin{cases}
a^m \C[a]\{s\} & \text{if }p\geq s \text{ and } p+r \geq s+2km, \\
0 & \text{if } p+r < s+2km.
\end{cases}
\]
Recall that $E_r^{p,i-p}(T_{i,m,s}) = Z_r^{p,i-p}(T_{i,m,s})/(Z_r^{p-1,i-p+1}(T_{i,m,s}) + B_r^{p,i-p}(T_{i,m,s}))$. Putting these together, we get
\[
E_r^{p,i-p}(T_{i,m,s})  \cong \begin{cases}
(\C[a]/(a^m))\{s\} & \text{if } p=s \text{ and } r \geq 2km+1, \\
\C[a]\{s\} & \text{if } p=s \text{ and } r \leq 2km, \\
0 & \text{otherwise.}
\end{cases} 
\]
This completes the computation of $\{E_r(T_{i,m,s})\}$. Note that $\{E_r(T_{i,m,s})\}$ collapses exactly at its $E_{2km+1}$-page.
\end{proof}

Next, we prove Theorem \ref{thm-spectral-sequence-decomp}.

\begin{proof}[Proof of Theorem \ref{thm-spectral-sequence-decomp}]
Fix a diagram $D$ of $L$ and a marking of $D$. Let $x_1,\dots,x_m$ be the variables assigned to the marked points of $\Gamma$ and $R=\C[x_1,\dots,x_m,a]$. By Corollary \ref{cor-graph-homology-free}, $H(C_P(D),d_{mf})$ with its polynomial grading is a finitely generated graded-free $\C[a]$-module. By the definition of $d_\chi$, we know it preserves the polynomial grading.

According to Lemma \ref{lemma-graded-chain-decomp}, in the category of chain complexes of graded $\C[a]$-modules, $(H(C_P(D),d_{mf}),d_\chi)$ decomposes into a direct sum of chain complexes of the forms $F_{i,s}$ and $T_{i,m,s}$ given in \eqref{def-F-i-s}--\eqref{def-T-i-m-s}. By Lemma \ref{lemma-homogeneous-basis-fil}, this is also a decomposition in the category of filtered chain complexes, in which the filtrations on $F_{i,s}$ and $T_{i,m,s}$ are given by \eqref{def-F-i-s-fil}--\eqref{def-T-i-m-s-fil}. Thus, the spectral sequence of $(H(C_P(D),d_{mf}),d_\chi)$ is the direct sum of the spectral sequences of its components in this decomposition. Decomposition \eqref{eq-spectral-sequence-decomp-E} in Theorem \ref{thm-spectral-sequence-decomp} then follows from this and Lemma \ref{lemma-ss-components}.

By Lemma \ref{lemma-homogeneous-basis-fil-pi-1}, the above decomposition of $(H(C_P(D),d_{mf}),d_\chi)$ induces a decomposition of \linebreak $(H(\hat{C}_P(D),d_{mf}),d_\chi)$ in the category of filtered chain complexes into a direct sum of chain complexes of the forms $\hat{F}_{i,s}$ and $\hat{T}_{i,m,s}$ given in \eqref{def-hat-F-i-s-fil}--\eqref{def-hat-T-i-m-s-fil}. The spectral sequence of $(H(\hat{C}_P(D),d_{mf}),d_\chi)$ is the direct sum of the spectral sequences of its components in this decomposition. Decomposition \eqref{eq-spectral-sequence-decomp-hat-E} in Theorem \ref{thm-spectral-sequence-decomp} then follows from this and Lemma \ref{lemma-ss-components}.
\end{proof}

\subsection{Exact couples}\label{subsec-couples} Let us first recall the definition of derived couples. 

\begin{lemma}\label{lemma-derived-couple}
Let $(A,E,f,g,h)$ be an exact couple as defined in Definition \ref{def-exact-couple}. Define 
\begin{enumerate}[1.]
	\item $A'=f(A)$,
	\item $E' = H(E,d)=\ker d / d(E)$, where $d=g\circ h:E \rightarrow E$,
	\item $f'=f|_{A'}$,
	\item $g'(\alpha) = g(\beta)$ where $\alpha = f(\beta) \in A'$.
	\item $h'(\eta + d(E)) = h(\eta)$ for any $\eta \in \ker d$,
\end{enumerate}
Then 
\begin{itemize}
	\item $A'$ and $E'$ are $\zed^{\oplus 2}$-graded $\C$-linear spaces,
	\item $A' \xrightarrow{f'} A'$, $A' \xrightarrow{g'} E'$ and $E' \xrightarrow{h'} A'$ are well defined homogeneous homomorphisms of $\zed^{\oplus 2}$-graded $\C$-linear spaces,
	\item the triangle 
	\[
	\xymatrix{
	A' \ar[rr]^{f'} && A'  \ar[ld]^{g'} \\
	& E' \ar[lu]^{h'} &
	}
	\]
	is exact.
\end{itemize}
That is, $(A',E',f',g',h')$ is itself an exact couple. $(A',E',f',g',h')$ is called the derived couple of $(A,E,f,g,h)$. We shall write $(A',E',f',g',h')=(A,E,f,g,h)'$. 
\end{lemma} 

\begin{proof}
See for example \cite[Proposition 2.7]{McCleary-users-guide-to-SS}.
\end{proof}

The following is a simple observation.

\begin{corollary}\label{cor-ec-h-vanish}
In the notations of Lemma \ref{lemma-derived-couple}, $h'=0$ if $h=0$.
\end{corollary}

For a chain complex $C$ of graded-free $\C[a]$-modules, denote by $(A^{(1)}(C),E^{(1)}(C),f^{(1)}_C,g^{(1)}_C,h^{(1)}_C)$ the exact couple 
\[	
\xymatrix{
H(C) \ar[rr]^{a} && H(C)  \ar[ld]^{\pi_a} \\
& H(C/aC) \ar[lu]^{\Delta} &
}
\]
induced by the short exact sequence
\[
0 \rightarrow C \xrightarrow{a} C \xrightarrow{\pi_a} C/aC \rightarrow 0,
\]
where $\pi_a$ is the standard quotient map. Define a sequence $\{(A^{(r)}(C),E^{(r)}(C),f^{(r)}_C,g^{(r)}_C,h^{(r)}_C)\}$ of exact couples such that $(A^{(r)}(C),E^{(r)}(C),f^{(r)}_C,g^{(r)}_C,h^{(r)}_C)=(A^{(r-1)}(C),E^{(r-1)}(C),f^{(r-1)}_C,g^{(r-1)}_C,h^{(r-1)}_C)'$.

\begin{lemma}\label{lemma-ec-components}
Let $F_{i,s}$ and $T_{i,m,s}$ be the chain complexes defined in \eqref{def-T-i-m-s} and \eqref{def-F-i-s}. Then, as $\zed^{\oplus 2}$-graded $\C$-linear spaces, 
\begin{eqnarray}
\label{eq-F-i-s-ec} E^{(r)}(F_{i,s}) & \cong &  \C\|i\|\{s\} ~\forall~ r\geq 1,\\
\label{eq-T-i-m-s-ec} E^{(r)}(T_{i,m,s}) & \cong &  \begin{cases}
\C\|i-1\|\{2km+s\} \oplus \C\|i\|\{s\} & \text{if } 1 \leq r\leq m, \\
0 & \text{if } r \geq m+1.
\end{cases}
\end{eqnarray}
\end{lemma}

\begin{proof}
For the chain complex $F_{i,s}$, note that $h^{(1)}_{F_{i,s}}=0$ in the exact couple $(A^{(1)}(F_{i,s}),E^{(1)}(F_{i,s}),f^{(1)}_{F_{i,s}},g^{(1)}_{F_{i,s}},h^{(1)}_{F_{i,s}})$. By Corollary \ref{cor-ec-h-vanish}, this means $h^{(r)}_{F_{i,s}}=0$ $\forall ~r \geq 1$. So the differential on $E^{(r)}(F_{i,s})$ is $0$ $\forall~r \geq 1$. Thus, $E^{(r)}(F_{i,s})  \cong E^{(1)}(F_{i,s}) \cong \C\|i\|\{s\}$ $\forall ~r \geq 1$. So we have proved isomorphism \eqref{eq-F-i-s-ec}.

Now consider the chain complex $T_{i,m,s}$. Note that $H(T_{i,m,s}) \cong \C[a]/(a^m)\|i\|\{s\}$, $H(T_{i,m,s}/aT_{i,m,s}) \cong \C\|i-1\|\{2km+s\} \oplus \C\|i\|\{s\}$ and the exact couple 
\[
(A^{(1)}(T_{i,m,s}),E^{(1)}(T_{i,m,s}),f^{(1)}_{T_{i,m,s}},g^{(1)}_{T_{i,m,s}},h^{(1)}_{T_{i,m,s}}) = (H(T_{i,m,s}), H(T_{i,m,s}/aT_{i,m,s}),a, \pi_a, \Delta)
\] 
is the exact sequence
\[
0\rightarrow \C\|i-1\|\{2km+s\} \xrightarrow{\Delta} \C[a]/(a^m)\|i\|\{s\} \xrightarrow{a} \C[a]/(a^m)\|i\|\{s\} \xrightarrow{\pi_a} \C[a]\|i\|\{s\} \rightarrow 0,
\]
where the connecting homomorphism $\C\xrightarrow{\Delta} \C[a]/(a^m)$ is given by $\Delta (1) = a^{m-1}$. 

For $1\leq r \leq m-1$ denote by $a^r \cdot \C[a]/(a^m)$ the subspace of $\C[a]/(a^m)$ spanned by $\{a^r, a^{r+1},\dots,a^{m-1}\}$ and by $a^{-r}: a^r \cdot \C[a]/(a^m) \rightarrow \C[a]/(a^m)$ and linear mapping given by $a^{-r} (a^{r+i}) = a^i$ for $i=0,\dots,m-1-r$. A simple induction shows that, for $1 \leq r \leq m$, there is an isomorphism of exact couples
\[
(A^{(r)}(T_{i,m,s}),E^{(r)}(T_{i,m,s}),f^{(r)}_{T_{i,m,s}},g^{(r)}_{T_{i,m,s}},h^{(r)}_{T_{i,m,s}}) \cong (a^{r-1} \cdot H(T_{i,m,s}), H(T_{i,m,s}/aT_{i,m,s}),a, \pi_a\circ a^{-r+1}, \Delta).
\] 
Thus, $E^{(r)}(T_{i,m,s}) \cong H(T_{i,m,s}/aT_{i,m,s}) \cong \C\|i-1\|\{2km+s\} \oplus \C\|i\|\{s\}$ if $1 \leq r\leq m$.

When $r=m$, the differential on $E^{(m)}(T_{i,m,s})\cong \C\|i-1\|\{2km+s\} \oplus \C\|i\|\{s\}$ is 
\[
d^{(m)} =  g^{(r)}_{T_{i,m,s}} \circ h^{(r)}_{T_{i,m,s}}= \pi_a\circ a^{-m+1}\circ \Delta,
\] 
which is an isomorphism $\C\|i-1\|\{2km+s\} \xrightarrow{\cong} \C\|i\|\{s\}$. So $E^{(m+1)}(T_{i,m,s}) = H(E^{(m)}(T_{i,m,s}),d^{(m)}) \cong 0$. This completes the proof of isomorphism \eqref{eq-T-i-m-s-ec}.
\end{proof}

\begin{proof}[Proof of Theorem \ref{thm-exact-couple-lee-gornik}]
Comparing Lemma \ref{lemma-ec-components} to Lemma \ref{lemma-ss-components}, one can see that Theorem \ref{thm-exact-couple-lee-gornik} follows from Theorems \ref{thm-H_P-decomp} and \ref{thm-spectral-sequence-decomp}.
\end{proof}

\section{The $\bigwedge^\ast \C^{N-1}$-action on $H_N(L)$}\label{sec-action}

As we have seen, every polynomial $P=P(x,a)$ of form \eqref{def-P} induces an exact couple $(H_P(L),H_N(L),a, \pi_a,\Delta)$, which equips $H_N(L)$ with a differential $d_P^{(1)}:=\pi_a \circ \Delta$. In this section, we study this differential $d_P^{(1)}$. Our goal is to prove Theorem \ref{thm-delta-action} and establish the $\bigwedge^\ast \C^{N-1}$-action on $H_N(L)$. 

\subsection{Naturality} In this subsection, we fix a $P=P(x,a)$ of form \eqref{def-P}.

From \cite{KR1}, we know that every link cobordism induces, up to an overall scaling by a non-zero scalar, a homomorphism of the $\slmf(N)$ Khovanov-Rozansky homology $H_N$. We briefly recall the definition of this homomorphism here.

To define this homomorphism, Khovanov and Rozansky first decompose the link cobordism into a finite sequence of Reidemeister and Morse moves, that is, a movie. For each Reidemeister move and Morse move, they define in \cite {KR1} a corresponding chain map of the chain complex $(H(C_N,d_{mf}),d_\chi)$. They then define the chain map associated to this cobordism to be the composition of the chain maps associated to the Reidemeister and Morse moves in this movie. They proved in \cite[Proposition 37]{KR1} that, up to an overall scaling by a non-zero scalar, the homomorphism on $H_N$ induced by this chain map does not depend on the choice of the movie.

In \cite{Krasner,Wu-color-equi}, Khovanov and Rozansky's chain maps associated to Reidemeister and Morse moves are generalized to $\C[a]$-linear homogeneous chain maps of the chain complex $(H(C_P,d_{mf}),d_\chi)$. So each movie presentation of a link cobordism induces a $\C[a]$-linear homogeneous chain map of the complex $(H(C_P,d_{mf}),d_\chi)$.\footnote{It does not seem too hard to generalize Khovanov and Rozansky's proof in \cite{KR1} to show that, up to an overall scaling by a non-zero scalar, the homomorphism on $H_P$ induced by this chain map does not depend on the choice of the movie presentation of the cobordism. But we do not need this to prove our results.} Comparing the constructions in \cite{KR1,Krasner,Wu-color-equi}, we have the following lemma.

\begin{lemma}\label{lemma-cobordism-morphisms}
Let $S$ be a link cobordism from link $L_0$ to $L_1$. Fix diagrams of $L_0$, $L_1$ and a movie presentation of $S$. Denote by $(H(C_N(L_0),d_{mf}),d_\chi) \xrightarrow{f_N} (H(C_N(L_1),d_{mf}),d_\chi)$ and $(H(C_P(L_0),d_{mf}),d_\chi) \xrightarrow{f_P} (H(C_P(L_1),d_{mf}),d_\chi)$ the chain maps induced by this movie presentation of $S$. Then the following diagram commutes.
\[
\xymatrix{
H(C_P(L_0),d_{mf}) \ar[rr]^{f_P} \ar[d]^{\pi_a}&& H(C_P(L_1),d_{mf}) \ar[d]^{\pi_a} \\
H(C_N(L_0),d_{mf}) \ar[rr]^{f_N}&& H(C_N(L_1),d_{mf})
}
\]
\end{lemma}
\begin{proof}
See the constructions in \cite{KR1,Krasner,Wu-color-equi}.
\end{proof}

Next, we interpret $d_P^{(1)}$ as the connecting homomorphism of a long exact sequence, which slightly simplifies the proof of the naturality and significantly simplifies the proof of the anti-commutativity later on. 

For a link $L$, choose one of its diagrams. By Corollary \ref{cor-graph-homology-free} and Lemma \ref{lemma-basis-pi-0}, $H(C_P(L),d_{mf})$ is a free $\C[a]$-module and $H(C_N(L),d_{mf})\cong H(C_P(L),d_{mf})/aH(C_P(L),d_{mf})$. Also, from the proof of Lemma \ref{lemma-basis-pi-0}, one can see that $H(C_P(L)/a^2C_P(L),d_{mf}) \cong H(C_P(L),d_{mf})/a^2H(C_P(L),d_{mf})$. Therefore, the short exact sequence $0\rightarrow \C[a]/(a) \xrightarrow{a} \C[a]/(a^2) \xrightarrow{\pi_a} \C[a]/(a) \rightarrow 0$ induces a short exact sequence
\begin{equation}\label{eq-short-exact-d-1}
0 \rightarrow H(C_N(L),d_{mf}) \xrightarrow{a} H(C_P(L)/a^2C_P(L),d_{mf}) \xrightarrow{\pi_a} H(C_N(L),d_{mf}) \rightarrow 0.
\end{equation}

\begin{lemma}\label{lemma-d-1-short}
Let $\mathscr{H}_P(L) = H(H(C_P(L)/a^2C_P(L),d_{mf}),d_\chi)$. Then short exact sequence \eqref{eq-short-exact-d-1} induces a long exact sequence
\[
\cdots \xrightarrow{\pi_a} H_N^{i-1}(L)\xrightarrow{d_P^{(1)}} H_N^i(L)\xrightarrow{a} \mathscr{H}_P^i(L) \xrightarrow{\pi_a} H_N^i(L) \xrightarrow{d_P^{(1)}}\cdots.
\]
\end{lemma}

\begin{proof}
Denote by $\delta$ the connecting homomorphism in the above long exact sequence. We only need to prove that $\delta=d_P^{(1)}$. Consider the following commutative diagram with short exact rows.
\[
\xymatrix{
0 \ar[r]& H(C_P(L),d_{mf}) \ar[r]^{a} \ar[d]^{\pi_a}& H(C_P(L),d_{mf}) \ar[r]^{\pi_a} \ar[d]^{\pi_{a^2}} & H(C_N(L),d_{mf}) \ar[r] \ar[d]^{\id}& 0 \\
0 \ar[r]& H(C_N(L),d_{mf}) \ar[r]^<<<<<{a}& H(C_P(L)/a^2C_P(L),d_{mf}) \ar[r]^<<<<<{\pi_a}& H(C_N(L),d_{mf}) \ar[r]& 0}
\]
It induces the following commutative diagram with long exact rows.
\[
\xymatrix{
\cdots \ar[r]^<<<<<{\pi_a} & H_N^{i-1}(L) \ar[r]^{\Delta} \ar[d]^{\id} & H_P^{i}(L) \ar[r]^<<<<<{a} \ar[d]^{\pi_a} & \cdots \\
\cdots \ar[r]^<<<<<{\pi_a} & H_N^{i-1}(L) \ar[r]^{\delta} & H_N^{i}(L) \ar[r]^<<<<<{a} & \cdots
}
\]
Thus, $\delta = \pi_a \circ \Delta = d_P^{(1)}$.
\end{proof}

\begin{lemma}\label{lemma-delta-natural}
Let $S$ be a link cobordism from link $L_0$ to $L_1$. Denote by $H_N(L_0) \xrightarrow{f_N} H_N(L_1)$ the homomorphism induced by $S$. Then the following diagram commutes.
\[
\xymatrix{
H_N^{i-1}(L_0) \ar[rr]^{d_P^{(1)}} \ar[d]^{f_N}&& H_N^{i}(L_0) \ar[d]^{f_N} \\
H_N^{i-1}(L_1) \ar[rr]^{d_P^{(1)}}&& H_N^{i}(L_1)
}
\]
\end{lemma}

\begin{proof}
Pick diagrams for $L_0$, $L_1$ and choose a movie presentation of $S$. Denote by $(H(C_N(L_0),d_{mf}),d_\chi) \xrightarrow{f_N} (H(C_N(L_1),d_{mf}),d_\chi)$ and $(H(C_P(L_0),d_{mf}),d_\chi) \xrightarrow{f_P} (H(C_P(L_1),d_{mf}),d_\chi)$ the chain maps induced by this movie presentation of $S$. Of course, $H_N(L_0) \xrightarrow{f_N} H_N(L_1)$ is, up to scaling, the homomorphism induced by the chain map $(H(C_N(L_0),d_{mf}),d_\chi) \xrightarrow{f_N} (H(C_N(L_1),d_{mf}),d_\chi)$. Recall that $f_P$ is $\C[a]$-linear and \linebreak $H(C_P(L)/a^2C_P(L),d_{mf}) \cong H(C_P(L),d_{mf})/a^2H(C_P(L),d_{mf})$. So $f_P$ induces a chain map \linebreak $H(C_P(L)/a^2C_P(L),d_{mf}) \xrightarrow{f_P} H(C_P(L)/a^2C_P(L),d_{mf})$. Thus, we have the following commutative diagram with short exact rows.
\[
\xymatrix{
0 \ar[r]& H(C_N(L_0),d_{mf}) \ar[r]^<<<<<{a} \ar[d]^{f_N} & H(C_P(L_0)/a^2C_P(L),d_{mf}) \ar[r]^<<<<<{\pi_a} \ar[d]^{f_P} & H(C_N(L_0),d_{mf}) \ar[r] \ar[d]^{f_N} & 0 \\
0 \ar[r]& H(C_N(L_1),d_{mf}) \ar[r]^<<<<<{a}& H(C_P(L_1)/a^2C_P(L),d_{mf}) \ar[r]^<<<<<{\pi_a}& H(C_N(L_1),d_{mf}) \ar[r]& 0}
\]
By Lemma \ref{lemma-d-1-short}, this diagram induces the following commutative diagram with long exact rows.
\[
\xymatrix{
\cdots \ar[r]^>>>>>{\pi_a} & H_N^{i-1}(L_0) \ar[r]^{d_P^{(1)}} \ar[d]^{f_N} & H_N^{i}(L_0) \ar[d]^{f_N} \ar[r]^>>>>>{a}  &  \cdots \\
\cdots \ar[r]^>>>>>{\pi_a} & H_N^{i-1}(L_1) \ar[r]^{d_P^{(1)}} & H_N^{i}(L_1) \ar[r]^>>>>>{a}  &  \cdots  
}
\]
This proves the lemma.
\end{proof}

Note that Part (3) of Theorem \ref{thm-delta-action} follows from Lemma \ref{lemma-delta-natural}.

\subsection{Anti-commutativity} In this subsection, we fix a homogeneous polynomial
\begin{equation}\label{def-P-a-1-a-2}
P=P(x,a_1,a_2) = x^{N+1} + xF(x,a_1,a_2)
\end{equation}
of degree $2N+2$, where 
\begin{itemize}
	\item $x$, $a_1$ and $a_2$ are homogeneous variables of degrees $2$, $2k_1$ and $2k_2$, respectively,
	\item $\deg F(x,a_1,a_2) =2N$ and $F(x,0,0)=0$.
\end{itemize}
We define 
\begin{equation}\label{def-P-a-1-a-2-partial}
P_1=P(x,a_1,0) \hspace{2pc} \text{and} \hspace{2pc} P_2=P(x,0,a_2).
\end{equation} 

The goal of this subsection is to show that $d_{P_1}^{(1)}$ and $d_{P_2}^{(1)}$ anti-commute, which implies Part (2) of Theorem \ref{thm-delta-action}.

\begin{lemma}\label{lemma-diagram-chasing}
Let $R$ be a commutative ring and $A$, $B$, $C$ and $D$ chain complexes of $R$-modules, whose differentials raise the homological grading by $1$. Assume there is an exact sequence of chain complexes
\[
0 \rightarrow A \xrightarrow{f} B \xrightarrow{g} C \xrightarrow{h} D \rightarrow 0.
\]
Then this exact sequence induces an $R$-homomorphism $\mathbf{\Delta}: H^i(D) \rightarrow H^{i+2}(A)$ for every homological degree $i$.
\end{lemma}

\begin{proof}
Of course, one can split the exact sequence $0 \rightarrow A \xrightarrow{f} B \xrightarrow{g} C \xrightarrow{h} D \rightarrow 0$ into short exact sequences $0 \rightarrow A \xrightarrow{f} B \rightarrow B/f(A) \rightarrow 0$ and $0 \rightarrow B/f(A) \xrightarrow{g} C \xrightarrow{h} D \rightarrow 0$. Then $\mathbf{\Delta}$ can be defined as the composition of the connecting homomorphisms from these two short exact sequences. But what we actually need later on is that $\mathbf{\Delta}$ is defined by diagram chasing and does not depend on the choices made in that chasing. So this is how we will prove the lemma here.

\[
\xymatrix{
& \vdots & \vdots & \vdots & \vdots & \\
0 \ar[r] & A^{i+2} \ar[r]^{f} \ar[u]^{d_A} & B^{i+2} \ar[r]^{g} \ar[u]^{d_B}  & C^{i+2} \ar[r]^{h} \ar[u]^{d_C}  & D^{i+2} \ar[r] \ar[u]^{d_D}  & 0 \\
0 \ar[r] & A^{i+1} \ar[r]^{f} \ar[u]^{d_A} & B^{i+1} \ar[r]^{g} \ar[u]^{d_B}  & C^{i+1} \ar[r]^{h} \ar[u]^{d_C}  & D^{i+1} \ar[r] \ar[u]^{d_D}  & 0 \\
0 \ar[r] & A^{i} \ar[r]^{f} \ar[u]^{d_A} & B^{i} \ar[r]^{g} \ar[u]^{d_B}  & C^{i} \ar[r]^{h} \ar[u]^{d_C}  & D^{i} \ar[r] \ar[u]^{d_D}  & 0 \\
0 \ar[r] & A^{i-1} \ar[r]^{f} \ar[u]^{d_A} & B^{i-1} \ar[r]^{g} \ar[u]^{d_B}  & C^{i-1} \ar[r]^{h} \ar[u]^{d_C}  & D^{i-1} \ar[r] \ar[u]^{d_D}  & 0 \\
& \vdots \ar[u]^{d_A} & \vdots \ar[u]^{d_B}  & \vdots\ar[u]^{d_C}  & \vdots \ar[u]^{d_D}  &  \\
}
\]

In the above diagram, let $x \in D^i$ be a cycle. That is, $d_D(x)=0$. Since $h$ is surjective, there is a chain $y\in C^i$ such that $h(y)=x$. Then $h(d_C(y))= d_D(x)=0$. So $d_C(y) \in \ker h = g(B^{i+1})$. Thus, there exists a $z \in B^{i+1}$ such that $g(z)=d_C(y)$. Then $g(d_B(z))= d_C(d_C(y))=0$. So $d_B(z) \in \ker g =f(A)$. That is, there exists $w \in A^{i+2}$ such that $f(w)= d_B(z)$. But $f(d_A(w))= d_B(d_B(z))=0$ and $f$ is injective. So $d_A(w)=0$, that is, $w$ is a cycle in $A^{i+2}$.

Next, we show that the mapping $[x] \mapsto [w]$ is a well defined homomorphism on homology, that is, it does not depend on the choices made in the above construction.

Assume $x' \in D^i$ is a cycle such that $[x']=[x]$. Then there is a $\tilde{x} \in D^{i-1}$ with $d_D(\tilde{x}) = x'-x$. Since $h$ is surjective, there is a $\tilde{y} \in C^{i-1}$ satisfying $h(\tilde{y})=\tilde{x}$. Now let $y' \in C^i$ be any chain such that $h(y')=x'$. Then $h(y') = x' = x + d_D (\tilde{x}) = h(y+d_C(\tilde{y}))$. Thus, $y'-y-d_C(\tilde{y}) \in \ker h = g(B^i)$. This means that there exists a $\tilde{z} \in B^i$ such that $g(\tilde{z})= y'-y-d_C(\tilde{y})$. Now let $z' \in B^{i+1}$ be any chain satisfying $g(z')=d_C (y')$. Then $g(z') = d_C (y') = d_C (y + d_C(\tilde{Y}) + g(\tilde{z})) = g(z+d_B(\tilde{z}))$. This implies that $z'-z-d_B(\tilde{z}) \in \ker g = f(B^{i+1})$. So there exists a $\tilde{w} \in A^{i+1}$ such that $f(\tilde{w}) = z'-z-d_B(\tilde{z})$. Finally, let $w' \in A^{i+2}$ be any chain with $f(w')= d_B(z')$. Then $f(w') = d_B(z') = d_B(z+d_B(\tilde{z}) + f(\tilde{w})) = f(w+d_A(\tilde{w}))$. But $f$ is injective. So $w'=w+d_A(\tilde{w})$. This shows that $w'$ is a cycle and $[w']=[w]$.

From the above, we know that $\mathbf{\Delta}: H^i(D) \rightarrow H^{i+2}(A)$ given by $\mathbf{\Delta}([x]) =[w]$ is well defined. It is straightforward to show that $\mathbf{\Delta}$ is $R$-linear.
\end{proof}

\begin{lemma}\label{lemma-delta-anti-commute-general}
Let $P_1$ and $P_2$ be the polynomials defined in \eqref{def-P-a-1-a-2-partial}. Then $d_{P_1}^{(1)} \circ d_{P_2}^{(1)} = -d_{P_2}^{(1)} \circ d_{P_1}^{(1)}$.
\end{lemma}

Note that, applying Lemma \ref{lemma-delta-anti-commute-general} to $P(x,b_i,b_j) = x^{N+1} + b_ix^i + b_j x^j$, we get Part (2) of Theorem \ref{thm-delta-action}.

\begin{proof}[Proof of Lemma \ref{lemma-delta-anti-commute-general}]
Let $L$ be any link and $D$ one of its diagrams. Recall that $P$ is the polynomial in \eqref{def-P-a-1-a-2}. Set $R=\C[a_1,a_2]$. 

Consider the following diagram.
\begin{equation}\label{eq-coefficient-diagram}
\xymatrix{
& 0 \ar[d]& 0 \ar[d]& 0 \ar[d]& \\
0 \ar[r]& R/(a_1,a_2) \ar[r]^{a_1} \ar[d]^{-a_2} \ar@{}[dr]|{-}&  R/(a_1^2,a_2) \ar[r]^{\pi_{a_1}} \ar[d]^{a_2} \ar@{}[dr]|{+} & R/(a_1,a_2) \ar[d]^{a_2} \ar[r] & 0 \\
0 \ar[r]& R/(a_1,a_2^2) \ar[r]^{a_1} \ar[d]^{\pi_{a_2}} \ar@{}[dr]|{+} &  R/(a_1^2,a_2^2) \ar[r]^{\pi_{a_1}} \ar[d]^{\pi_{a_2}} \ar@{}[dr]|{+} & R/(a_1,a_2^2) \ar[d]^{\pi_{a_2}} \ar[r] & 0 \\
0 \ar[r]& R/(a_1,a_2) \ar[r]^{a_1} \ar[d] & R/(a_1^2,a_2) \ar[r]^{\pi_{a_1}} \ar[d] & R/(a_1,a_2) \ar[r] \ar[d] & 0 \\
& 0 & 0 & 0 & 
}
\end{equation}
Note that, in diagram \eqref{eq-coefficient-diagram}, 
\begin{itemize}
	\item all rows are exact,
	\item all columns are exact,
	\item the upper left square anti-commutes, which is indicated by a ``$-$",
	\item the other three squares commute, which is indicated by ``$+$"s.
\end{itemize}
Moreover, we get from diagram \eqref{eq-coefficient-diagram} an exact sequence
\begin{equation}\label{eq-coefficient-sequence}
0 \rightarrow R/(a_1,a_2) \xrightarrow{\left(%
\begin{array}{r}
  a_1 \\
  -a_2 
\end{array}%
\right)
}
\left.%
\begin{array}{c}
  R/(a_1^2,a_2) \\
  \oplus \\
  R/(a_1,a_2^2)
\end{array}%
\right.
\xrightarrow{(a_2,a_1)} 
R/(a_1^2,a_2^2)
\xrightarrow{\pi_{a_1} \circ \pi_{a_2} = \pi_{a_2} \circ \pi_{a_1}}
R/(a_1,a_2) \rightarrow 0.
\end{equation}

By Corollary \ref{cor-graph-homology-free}, $H(C_P(D),d_{mf})$ is a chain complex of graded-free $R$-modules. For $i,j=1,2$ we denote by $\mathscr{C}_{i,j}$ the chain complex $(H(C_P(D),d_{mf})/(a_1^i,a_2^j)H(C_P(D),d_{mf}), d_\chi)$. Note that $\mathscr{C}_{1,1} \cong (H(C_N(D),d_{mf}), d_\chi)$, whose homology is $H_N(L)$. Exact sequence \eqref{eq-coefficient-sequence} induces an exact sequence
\begin{equation}\label{eq-coefficient-chain-sequence}
0 \rightarrow \mathscr{C}_{1,1} \xrightarrow{\left(%
\begin{array}{r}
  a_1 \\
  -a_2 
\end{array}%
\right)
}
\left.%
\begin{array}{c}
  \mathscr{C}_{2,1} \\
  \oplus \\
  \mathscr{C}_{1,2}
\end{array}%
\right.
\xrightarrow{(a_2,a_1)} 
\mathscr{C}_{2,2}
\xrightarrow{\pi_{a_1} \circ \pi_{a_2} = \pi_{a_2} \circ \pi_{a_1}}
\mathscr{C}_{1,1} \rightarrow 0.
\end{equation}

By Lemma \ref{lemma-diagram-chasing}, exact sequence \eqref{eq-coefficient-chain-sequence} induces a homomorphism $\mathbf{\Delta}:H_N^i(L) \rightarrow H_N^{i+2}(L)$. We prove the lemma by showing that
\begin{equation}\label{eq-anti-commute-Delta}
\mathbf{\Delta} = d_{P_1}^{(1)} \circ d_{P_2}^{(1)} = -d_{P_2}^{(1)} \circ d_{P_1}^{(1)}.
\end{equation}
To prove \eqref{eq-anti-commute-Delta}, we demonstrate that each of $d_{P_1}^{(1)} \circ d_{P_2}^{(1)}$ and $-d_{P_2}^{(1)} \circ d_{P_1}^{(1)}$ can be realized by a diagram chasing used to define $\mathbf{\Delta}$.

Again, recall that $H(C_P(D),d_{mf})$ is a chain complex of graded-free $R$-modules. Tensoring $H(C_P(D),d_{mf})$ with every item in diagram \eqref{eq-coefficient-diagram}, we get a diagram of chain complexes
\begin{equation}\label{eq-coefficient-chain-diagram}
\xymatrix{
& 0 \ar[d]& 0 \ar[d]& 0 \ar[d]& \\
0 \ar[r]& \mathscr{C}_{1,1} \ar[r]^{a_1} \ar[d]^{-a_2} \ar@{}[dr]|{-}& \mathscr{C}_{2,1} \ar[r]^{\pi_{a_1}} \ar[d]^{a_2} \ar@{}[dr]|{+} & \mathscr{C}_{1,1} \ar[d]^{a_2} \ar[r] & 0 \\
0 \ar[r]& \mathscr{C}_{1,2} \ar[r]^{a_1} \ar[d]^{\pi_{a_2}} \ar@{}[dr]|{+} &  \mathscr{C}_{2,2} \ar[r]^{\pi_{a_1}} \ar[d]^{\pi_{a_2}} \ar@{}[dr]|{+} & \mathscr{C}_{1,2} \ar[d]^{\pi_{a_2}} \ar[r] & 0 \\
0 \ar[r]& \mathscr{C}_{1,1} \ar[r]^{a_1} \ar[d] & \mathscr{C}_{2,1} \ar[r]^{\pi_{a_1}} \ar[d] & \mathscr{C}_{1,1} \ar[r] \ar[d] & 0 \\
& 0 & 0 & 0 & 
},
\end{equation}
in which 
\begin{itemize}
	\item all rows are exact,
	\item all columns are exact,
	\item the upper left square anti-commutes, which is indicated by a ``$-$",
	\item the other three squares commute, which is indicated by ``$+$"s.
\end{itemize}
This is the diagram that we will chase. Note that, for each homological degree $i$, there is a diagram of the form in \ref{eq-coefficient-chain-diagram}. So our diagram chasing involves three levels of a $3$-dimensional diagram. In stead of drawing the rather complex $2$-dimensional projection of this $3$-dimensional diagram, we look at each homological level individually with the understanding that the each differential map points from one spot on one level to the same spot one level higher.

\begin{equation}\label{eq-coefficient-chain-diagram-i}
\xymatrix{
& 0 \ar[d]& 0 \ar[d]& 0 \ar[d]& \\
0 \ar[r]& \mathscr{C}_{1,1}^i \ar[r]^{a_1} \ar[d]^{-a_2} \ar@{}[dr]|{-}& \mathscr{C}_{2,1}^i \ar[r]^{\pi_{a_1}} \ar[d]^{a_2} \ar@{}[dr]|{+} & \mathscr{C}_{1,1}^i \ar[d]^{a_2} \ar[r] & 0 \\
0 \ar[r]& \mathscr{C}_{1,2}^i \ar[r]^{a_1} \ar[d]^{\pi_{a_2}} \ar@{}[dr]|{+} &  (z\in) \mathscr{C}_{2,2}^i \ar[r]^{\pi_{a_1}} \ar[d]^{\pi_{a_2}} \ar@{}[dr]|{+} & (y_2\in) \mathscr{C}_{1,2}^i \ar[d]^{\pi_{a_2}} \ar[r] & 0 \\
0 \ar[r]& \mathscr{C}_{1,1}^i \ar[r]^{a_1} \ar[d] & (y_1\in) \mathscr{C}_{2,1}^i \ar[r]^{\pi_{a_1}} \ar[d] & (x \in) \mathscr{C}_{1,1}^i \ar[r] \ar[d] & 0 \\
& 0 & 0 & 0 & 
}
\end{equation}

Let us start with homological degree $i$. As in diagram \eqref{eq-coefficient-chain-diagram-i}, let $x$ be any cycle in the $\mathscr{C}_{1,1}^i$ at the lower right corner. By the exactness, there is a $y_1$ in the $\mathscr{C}_{2,1}^i$ in the bottom row such that $\pi_{a_1}(y_1)=x$. Use exactness again, there is a $z$ in the $\mathscr{C}_{2,2}^i$ at the center such that $\pi_{a_2}(z)=y_1$. Let $y_2 = \pi_{a_1}(z)$ in the $\mathscr{C}_{1,2}^i$ in the right column. Since the lower right square commutes, we have that $\pi_{a_2}(y_2)=x$. This finishes the chase at homological degree $i$.

\begin{equation}\label{eq-coefficient-chain-diagram-i+1}
\xymatrix{
& 0 \ar[d]& 0 \ar[d]& 0 \ar[d]& \\
0 \ar[r]& \mathscr{C}_{1,1}^{i+1} \ar[r]^{a_1} \ar[d]^{-a_2} \ar@{}[dr]|{-}& (\alpha_2,\beta_1\in)\mathscr{C}_{2,1}^{i+1} \ar[r]^{\pi_{a_1}} \ar[d]^{a_2} \ar@{}[dr]|{+} & (w_2\in) \mathscr{C}_{1,1}^{i+1} \ar[d]^{a_2} \ar[r] & 0 \\
0 \ar[r]& (\alpha_1,\beta_2\in) \mathscr{C}_{1,2}^{i+1} \ar[r]^{a_1} \ar[d]^{\pi_{a_2}} \ar@{}[dr]|{+} &  (dz\in) \mathscr{C}_{2,2}^{i+1}\ar[r]^{\pi_{a_1}} \ar[d]^{\pi_{a_2}} \ar@{}[dr]|{+} & (dy_2\in) \mathscr{C}_{1,2}^{i+1} \ar[d]^{\pi_{a_2}} \ar[r] & 0 \\
0 \ar[r]& (w_1\in) \mathscr{C}_{1,1}^{i+1} \ar[r]^{a_1} \ar[d] & (dy_1\in) \mathscr{C}_{2,1}^{i+1} \ar[r]^{\pi_{a_1}} \ar[d] & (dx=0 \in) \mathscr{C}_{1,1}^{i+1} \ar[r] \ar[d] & 0 \\
& 0 & 0 & 0 & 
}
\end{equation}

Now we move to homological degree $i+1$. First, we map $x,~y_1,~y_2$ and $z$ by the differential maps. Recall that $x$ is a cycle. So $\pi_{a_1}(y_1)=\pi_{a_2}(y_2)=dx=0$. 

By exactness, there is a $w_1$ in the $\mathscr{C}_{1,1}^{i+1}$ at the lower left corner such that $a_1(w_1) = d y_1$. The chase $x \leadsto y_1 \leadsto w_1$ is the chase used in the definition of the connecting homomorphism of the long exact sequence from the bottom row of diagram \eqref{eq-coefficient-chain-diagram}. So $w_1$ is a cycle. Moreover, by Lemma \ref{lemma-d-1-short}, this connecting homomorphism is $d_{P_1}^{(1)}$. So 
\begin{equation}\label{eq-chase-half-1}
d_{P_1}^{(1)}([x]) = [w_1].
\end{equation}
Similarly, there is a cycle $w_2$ in the $\mathscr{C}_{1,1}^{i+1}$ at the lower left corner such that $a_2(w_2) = d y_2$ and  
\begin{equation}\label{eq-chase-half-2}
d_{P_2}^{(1)}([x]) = [w_2].
\end{equation}

By exactness, there is an $\alpha_1$ in the $\mathscr{C}_{1,2}^{i+1}$ in the left column such that $\pi_{a_2} (\alpha_1) = w_1$. Similarly, there is an $\alpha_2$ in the $\mathscr{C}_{2,1}^{i+1}$ in the top row such that $\pi_{a_1} (\alpha_2) = w_2$.

Note that $\pi_{a_2}(dz-a_1(\alpha_1)) = dy_1-a_1(w_1)=0$. So there is a $\beta_1$ in the $\mathscr{C}_{2,1}^{i+1}$ in the top row such that 
\begin{equation}\label{eq-def-beta-1}
a_2 (\beta_1) = dz-a_1(\alpha_1). 
\end{equation}
Similarly, $\pi_{a_1}(dz-a_2(\alpha_2)) = 0$ and there is a $\beta_2$ in the $\mathscr{C}_{1,2}^{i+1}$ in the left column such that 
\begin{equation}\label{eq-def-beta-2}
a_1(\beta_2) = dz-a_2(\alpha_2).
\end{equation}

\begin{equation}\label{eq-coefficient-chain-diagram-i+2}
\xymatrix{
& 0 \ar[d]& 0 \ar[d]& 0 \ar[d]& \\
0 \ar[r]& (\gamma_1,\gamma_2\in)\mathscr{C}_{1,1}^{i+2} \ar[r]^{a_1} \ar[d]^{-a_2} \ar@{}[dr]|{-}& (d\alpha_2,d\beta_1\in)\mathscr{C}_{2,1}^{i+2} \ar[r]^{\pi_{a_1}} \ar[d]^{a_2} \ar@{}[dr]|{+} & (dw_2=0\in) \mathscr{C}_{1,1}^{i+2} \ar[d]^{a_2} \ar[r] & 0 \\
0 \ar[r]& (d\alpha_1,d\beta_2\in) \mathscr{C}_{1,2}^{i+2} \ar[r]^{a_1} \ar[d]^{\pi_{a_2}} \ar@{}[dr]|{+} &  (ddz=0\in) \mathscr{C}_{2,2}^{i+2}\ar[r]^{\pi_{a_1}} \ar[d]^{\pi_{a_2}} \ar@{}[dr]|{+} &  \mathscr{C}_{1,2}^{i+2} \ar[d]^{\pi_{a_2}} \ar[r] & 0 \\
0 \ar[r]& (dw_1=0\in) \mathscr{C}_{1,1}^{i+2} \ar[r]^{a_1} \ar[d] &  \mathscr{C}_{2,1}^{i+2} \ar[r]^{\pi_{a_1}} \ar[d] &  \mathscr{C}_{1,1}^{i+2} \ar[r] \ar[d] & 0 \\
& 0 & 0 & 0 & 
}
\end{equation}

Finally, we look at homological grading $i+2$. By equation \eqref{eq-def-beta-1}, 
\[
(a_2,a_1)\left(%
\begin{array}{r}
  d\beta_1 \\
  d\alpha_1 
\end{array}%
\right) =ddz=0.
\]
By the exactness of sequence \eqref{eq-coefficient-chain-sequence}, there is a $\gamma_1$ in the $\mathscr{C}_{1,1}^{i+2}$ at the top left corner such that
\[
\left(%
\begin{array}{r}
  a_1(\gamma_1) \\
  -a_2(\gamma_1) 
\end{array}%
\right)=
\left(%
\begin{array}{r}
  d\beta_1 \\
  d\alpha_1 
\end{array}%
\right).
\]
Since $\left(%
\begin{array}{r}
  a_1 \\
  -a_2 
\end{array}%
\right)$ is injective and $\left(%
\begin{array}{r}
  a_1(d\gamma_1) \\
  -a_2(d\gamma_1) 
\end{array}%
\right) = \left(%
\begin{array}{r}
  dd\beta_1 \\
  dd\alpha_1
\end{array}%
\right)=0$, we know that $\gamma_1$ is a cycle. Clearly, the chase $x \leadsto z \leadsto \left(%
\begin{array}{r}
  \beta_1 \\
  \alpha_1
\end{array}%
\right) \leadsto \gamma_1$ defines $\mathbf{\Delta}([x])$. So $[\gamma_1] = \mathbf{\Delta}([x])$. Similarly, there is a cycle $\gamma_2$ in the $\mathscr{C}_{1,1}^{i+2}$ at the top left corner such that
\[
\left(%
\begin{array}{r}
  a_1(\gamma_2) \\
  -a_2(\gamma_2) 
\end{array}%
\right)=
\left(%
\begin{array}{r}
  d\alpha_2 \\
  d\beta_2 
\end{array}%
\right).
\]
The chase $x \leadsto z \leadsto \left(%
\begin{array}{r}
  \alpha_2 \\
  \beta_2
\end{array}%
\right) \leadsto \gamma_2$ defines $\mathbf{\Delta}([x])$. So $[\gamma_2] = \mathbf{\Delta}([x])$. Altogether, we have
\begin{equation}\label{eq-full-chase}
\mathbf{\Delta}([x]) = [\gamma_1] =[\gamma_2].
\end{equation}

Note that $a_2(-\gamma_1) = d \alpha_1$. So the chase $w_1 \leadsto \alpha_1 \leadsto -\gamma_1$ defines the connecting homomorphism of the long exact sequence from the left column in \eqref{eq-coefficient-chain-diagram}, which, by Lemma \ref{lemma-d-1-short}, is $d_{P_2}^{(1)}$. So
\begin{equation}\label{eq-chase-2nd-half-1}
d_{P_2}^{(1)}([w_1]) = -[\gamma_1].
\end{equation}
Similarly, note that $a_1(\gamma_2)=d \alpha_2$. By Lemma \ref{lemma-d-1-short}, the chase $w_2 \leadsto \alpha_2 \leadsto \gamma_2$ gives that
\begin{equation}\label{eq-chase-2nd-half-2}
d_{P_1}^{(1)}([w_2]) = [\gamma_2].
\end{equation} 

Putting equations \eqref{eq-chase-half-1}, \eqref{eq-chase-half-2}, \eqref{eq-full-chase}, \eqref{eq-chase-2nd-half-1} and \eqref{eq-chase-2nd-half-2} together, we get that
\[
\mathbf{\Delta}([x]) = d_{P_1}^{(1)}(d_{P_2}^{(1)}([x])) = - d_{P_2}^{(1)}(d_{P_1}^{(1)}([x])).
\]
This proves the lemma.
\end{proof}

\subsection{Action of $d_P^{(1)}$} In this subsection, we describe the action of $d_P^{(1)}$ on $H_N(L)$ in terms of torsion components of $H_P(L)$ and prove Parts (1) and (4) of Theorem \ref{thm-delta-action}.

Let $P=P(x,a)$ be of form \eqref{def-P} and $L$ a link. Recall that, according to Theorem \ref{thm-H_P-decomp}, $H_P(L)$ decomposes into components of the forms $\C[a]\|i\|\{s\}$ and $\C[a]/(a^m)\|i\|\{s\}$. By Corollary \ref{cor-H_N-decomp}, $\C[a]\|i\|\{s\}$ contributes a component $\C\|i\|\{s\}$ to $H_N(L)$ and $\C[a]/(a^m)\|i\|\{s\}$ contributes a component $\C\|i\|\{s\} \oplus \C\|i-1\|\{s+2km\}$ to $H_N(L)$. The next lemma describes the action of $d_{P}^{(1)}$ on such components and follows from the proof of Lemma \ref{lemma-ec-components}.

\begin{lemma}\label{lemma-d-(1)-tor}
\begin{enumerate}
	\item $d_{P}^{(1)}$ restricts to $0$ on the component $\C\|i\|\{s\}$ of $H_N(L)$ induced by the component $\C[a]\|i\|\{s\}$ of $H_P(L)$.
	\item If $m>1$, then $d_{P}^{(1)}$ restricts to $0$ on the component $\C\|i\|\{s\} \oplus \C\|i-1\|\{s+2km\}$ of $H_N(L)$ induced by the component $\C[a]/(a^m)\|i\|\{s\}$ of $H_P(L)$.
	\item On the component $\C\|i\|\{s\} \oplus \C\|i-1\|\{s+2k\}$ of $H_N(L)$ induced by the component $\C[a]/(a)\|i\|\{s\}$ of $H_P(L)$, the restriction of $d_{P}^{(1)}$ is given by:
	\begin{itemize}
	  \item $d_{P}^{(1)}|_{\C\|i\|\{s\}}=0$,
	  \item $d_{P}^{(1)}|_{\C\|i-1\|\{s+2k\}}$ is an isomorphism $\C\|i-1\|\{s+2k\} \xrightarrow{\cong}\C\|i\|\{s\}$.
  \end{itemize}
\end{enumerate}
\end{lemma}

\begin{proof}
The restrictions of $d_{P}^{(1)}$ on these components are computed in the proof of Lemma \ref{lemma-ec-components}. The current lemma follows from that computation.
\end{proof}

In \cite[4.2.3 Example]{Lee2}, Lee observed that the homomorphism $\Phi$ matches a pair of generators of bi-degree difference $(1,4)$. This is a special case of Part (3) of Lemma \ref{lemma-d-(1)-tor}\footnote{The normalization in \cite{Lee2} is different from ours.}.

\begin{corollary}\label{cor-d-(1)-tor-1}
Let $P=P(x,a)$ be of form \eqref{def-P} and $L$ a link. Then $d_{P}^{(1)}\neq 0$ on $H_N(L)$ if and only if at least one of the components of $H_P(L)$ in decomposition \eqref{eq-H_P-decomp} is of the form $\C[a]/(a)\|i\|\{s\}$.
\end{corollary}

\begin{proof}
This follows from Lemma \ref{lemma-d-(1)-tor}.
\end{proof}

\begin{lemma}\label{lemma-quotient-isomorphism-P-1-P-2}
Suppose 
\begin{eqnarray*}
P_1 & = & P_1(x,a) = x^{N+1} + \sum_{i=1}^{\left\lfloor \frac{N}{k} \right\rfloor} \lambda_{1,i} a^i x^{N+1-ik}, \\
P_2 & = & P_2(x,a) = x^{N+1} + \sum_{i=1}^{\left\lfloor \frac{N}{k} \right\rfloor} \lambda_{2,i} a^i x^{N+1-ik},
\end{eqnarray*}
where $a$ is a homogeneous variable of degree $2k$. Assume that there exists an integer $m$ such that $1\leq m \leq \left\lfloor \frac{N}{k} \right\rfloor$ and $\lambda_{1,i}=\lambda_{2,i}$ for $1\leq i\leq m-1$. Let $D$ be a link diagram with a marking. Define 
\begin{eqnarray*}
\mathscr{C}_{P_1,m}(D) & := & C_{P_1}(D)/a^m C_{P_1}(D),\\
\mathscr{C}_{P_2,m}(D) & := & C_{P_2}(D)/a^m C_{P_2}(D).
\end{eqnarray*}
Then $(\mathscr{C}_{P_1,m}(D), d_\chi)$ and $(\mathscr{C}_{P_2,m}(D), d_\chi)$ are identical as chain complexes of graded matrix factorizations of $0$ over $\C[a]/(a^m)$. Therefore, $H(H(\mathscr{C}_{P_1,m}(D),d_{mf}),d_\chi) \cong H(H(\mathscr{C}_{P_2,m}(D),d_{mf}),d_\chi)$ as $\zed^{\oplus 2}$-graded $\C[a]$-modules.
\end{lemma}

\begin{proof}
Let $x_1,\dots,x_m$ be the variables associated to marked points on $D$. For any MOY resolution $\Gamma$ of $D$, it is obvious from Definition \ref{def-mf-MOY} that $\mathscr{C}_{P_1,m}(\Gamma) := C_{P_1}(\Gamma)/a^m C_{P_1}(\Gamma)$ and $\mathscr{C}_{P_2,m}(\Gamma) := C_{P_2}(\Gamma)/a^m C_{P_2}(\Gamma)$ are the same matrix factorization of $0$ over the ring $\C[x_1,\dots,x_m,a]/(a^m)$. Let $\Gamma_0$ and $\Gamma_1$ be two MOY resolutions of $D$ that are different at exactly one crossing. That is, $\Gamma_0$ and $\Gamma_1$ resolve all but one crossings of $D$ the same way, and that one remaining crossing is resolved to a pair of parallel arcs in $\Gamma_0$ and a wide edge in $\Gamma_1$. From the construction in Lemma \ref{lemma-def-chi}, one can see that the following diagrams commute.
\[
\xymatrix{
\mathscr{C}_{P_1,m}(\Gamma_0) \ar[rr]^{\chi_0} \ar[d]^{\id} && \mathscr{C}_{P_1,m}(\Gamma_1)  \ar[d]^{\id} \\
\mathscr{C}_{P_2,m}(\Gamma_0) \ar[rr]^{\chi_0} && \mathscr{C}_{P_2,m}(\Gamma_1) 
}
\hspace{2pc}
\xymatrix{
\mathscr{C}_{P_1,m}(\Gamma_0)  \ar[d]^{\id} && \mathscr{C}_{P_1,m}(\Gamma_1) \ar[ll]_{\chi_1} \ar[d]^{\id} \\
\mathscr{C}_{P_2,m}(\Gamma_0)  && \mathscr{C}_{P_2,m}(\Gamma_1) \ar[ll]_{\chi_1}
}
\]
Thus, $(\mathscr{C}_{P_1,m}(D), d_\chi)$ and $(\mathscr{C}_{P_2,m}(D), d_\chi)$ are identical as chain complexes of matrix factorizations of $0$ over $\C[x_1,\dots,x_m,a]/(a^m)$. And the lemma follows.
\end{proof}

Lemma \ref{lemma-torsion-lower-bound} follows easily from Lemma \ref{lemma-quotient-isomorphism-P-1-P-2}.

\begin{proof}[Proof of Lemma \ref{lemma-torsion-lower-bound}]
Fix a diagram $D$ of $L$ and apply Lemma \ref{lemma-quotient-isomorphism-P-1-P-2} to $x^{N+1}$ and $P(x,a) = x^{N+1} + \sum_{i=m}^{\left\lfloor\frac{N}{k}\right\rfloor} \lambda_i a^i x^{N+1-ki}$. Then we have $H(H(\mathscr{C}_{x^{N+1},m}(D),d_{mf}),d_\chi) \cong H(H(\mathscr{C}_{P,m}(D),d_{mf}),d_\chi)$ as $\zed^{\oplus 2}$-graded $\C[a]$-modules. But $H(H(\mathscr{C}_{x^{N+1},m}(D),d_{mf}),d_\chi) \cong H_N(L) \otimes_\C \C[a]/(a^m)$. This means that, all direct sum components of $H(H(\mathscr{C}_{P,m}(D),d_{mf}),d_\chi)$ must be of the form $\C[a]/(a^m)\|i\|\{s\}$. If $H_P(L)$ has a torsion component of the form $\C[a]/(a^l)\|i\|\{s\}$ with $1\leq l<m$. Then the chain complex $(H(C_P(D), d_{mf}),d_\chi)$ has a direct sum component 
\[
T_{i,l,s} = 0\rightarrow \C[a]\|i-1\|\{s+2kl\} \xrightarrow{a^l} \C[a]\|i\|\{s\} \rightarrow 0.
\]
Since $H(C_P(D), d_{mf})$ is a free $\C[a]$-module, we know that 
\[
H(\mathscr{C}_{P,m}(D),d_{mf}) \cong H(C_P(D), d_{mf})/a^m H(C_P(D), d_{mf}).
\]
So $(H(\mathscr{C}_{P,m}(D),d_{mf}),d_\chi)$ contains a direct sum component 
\[
T_{i,l,s}/a^m T_{i,l,s} = 0\rightarrow \C[a]/(a^m)\|i-1\|\{s+2kl\} \xrightarrow{a^l} \C[a]/(a^m)\|i\|\{s\} \rightarrow 0.
\]
Therefore, $H(H(\mathscr{C}_{P,m}(D),d_{mf}),d_\chi)$ has a direct sum component $\C[a]/(a^l)\|i-1\|\{s+2km\} \oplus \C[a]/(a^{l})\|i\|\{s\}$. By Lemma \ref{lemma-graded-module-decomp-unique}, this is a contradiction. Thus, $H_P(L)$ does not contain torsion components isomorphic to $\C[a]/(a^l)\|i\|\{s\}$.
\end{proof}

Next, we apply Lemma \ref{lemma-torsion-lower-bound} and Corollary \ref{cor-d-(1)-tor-1} to prove Part (1) of Theorem \ref{thm-delta-action}.

\begin{corollary}\label{cor-delta-N=0}
\begin{enumerate}
	\item If $P(x,a) = x^{N+1} + \sum_{i=2}^{\left\lfloor\frac{N}{k}\right\rfloor} \lambda_i a^i x^{N+1-ki}$, then $d_P^{(1)}=0$ for any link. 
	\item $\delta_N=0$ for any link, where $\delta_N$ is defined in Subsection \ref{subsec-intro-action}.
\end{enumerate}
\end{corollary}

\begin{proof}
Let $L$ be any link and $D$ a diagram of $L$.

We prove Part (1) of the corollary first. By Lemma \ref{lemma-torsion-lower-bound}, $H_P(L)$ contains no torsion components of the form $\C[a]/(a)\|i\|\{s\}$. Then, by Corollary \ref{cor-d-(1)-tor-1}, $d_P^{(1)}=0$ on $H_N(L)$. This proves Part (1).

Now we prove Part (2). Recall that $\delta_N = d_{P_N}^{(1)}$, where $P_N=P_N(x,a)=x^{N+1}+b_N x^N$ and $b_N$ is a homogeneous variable of degree $2$. Define $P=P(y,b_N) = (y-\frac{b_N}{N+1})^{N+1} + b_N (y-\frac{b_N}{N+1})^{N}$. Note that $P$ is of form \eqref{def-P} and the coefficient of $b_Ny^N$ in $P$ is $0$. Thus,  by Lemma \ref{lemma-torsion-lower-bound}, $H_P(L)$ contains no torsion components of the form $\C[a]/(a)\|i\|\{s\}$.

Put a marking on $D$ and let $x_1,\dots,x_m$ be the variables associated to marked points on $D$. We introduce another collection of homogeneous variables $y_1,\dots,y_m$ of degree $2$. Denote by $\xi:\C[x_1,\dots,x_m,b_N] \rightarrow \C[y_1,\dots,y_m,b_N]$ the ring isomorphism given by $\xi(b_N) = b_N$ and $\xi(x_i) = y_i-\frac{b_N}{N+1}$ for $i=1,\dots,m$. In the remainder of this proof, we write $R_x=\C[x_1,\dots,x_m,b_N]$ and $R_y = \C[y_1,\dots,y_m,b_N]$.

Let $\Gamma$ be any MOY resolution of $D$. Assume $\Gamma_{i;p}$ and $\Gamma_{i,j;p,q}$ depicted in Figure \ref{fig-pieces-of-Gamma} are pieces of $\Gamma$. By definition \ref{def-mf-MOY}, it is clear that $\xi$ induces an isomorphism $\xi: C_{P_N}(\Gamma_{i;p}) \rightarrow C_{P}(\Gamma_{i;p})$. For $\Gamma_{i,j;p,q}$, $\xi$ induces an isomorphism
\begin{eqnarray*}
\xi: C_{P_N}(\Gamma_{i,j;p,q}) = \left(%
\begin{array}{cc}
  \ast & x_i+x_j-x_p-x_q \\
  \ast & x_ix_j-x_px_q \\
\end{array}%
\right)_{R_x}\{-1\} & \xrightarrow{\cong} & \left(%
\begin{array}{cc}
  \ast & \xi(x_i+x_j-x_p-x_q) \\
  \ast & \xi(x_ix_j-x_px_q) \\
\end{array}%
\right)_{R_y}\{-1\} \\
& = & \left(%
\begin{array}{cc}
  \ast & y_i+y_j-y_p-y_q \\
  \ast & y_iy_j-y_py_q -\frac{b_N}{N+1}(y_i+y_j-y_p-y_q) \\
\end{array}%
\right)_{R_y}\{-1\} \\
& \cong & \left(%
\begin{array}{cc}
  \ast & y_i+y_j-y_p-y_q \\
  \ast & y_iy_j-y_py_q  \\
\end{array}%
\right)_{R_y}\{-1\} \\
& \cong & C_{P}(\Gamma_{i,j;p,q}).
\end{eqnarray*}
In the above computation, we used \cite[Corollary 2.16 and Lemma 2.18]{Wu-color}. This shows that, for any MOY resolution $\Gamma$ of $D$, $\xi$ induces an isomorphism $\xi: C_{P_N}(\Gamma) \xrightarrow{\cong} C_{P}(\Gamma)$.

Now let $\Gamma_0$ and $\Gamma_1$ be two MOY resolutions of $D$ that are different at exactly one crossing. That is, $\Gamma_0$ and $\Gamma_1$ resolve all but one crossings of $D$ the same way, and that one remaining crossing is resolved to a pair of parallel arcs in $\Gamma_0$ and a wide edge in $\Gamma_1$. Assume $\Gamma_{i;p}\sqcup \Gamma_{j;q}$ and $\Gamma_{i,j;p,q}$ in Figure \ref{fig-chi} are the pieces of $\Gamma_0$ and $\Gamma_1$ from resolving this crossing. Then the homomorphisms $\xymatrix{C_{P_N}(\Gamma_{i;p}\sqcup \Gamma_{j;q}) \ar@<.5ex>[r]^{\chi_0} & C_{P_N}(\Gamma_{i,j;p,q}) \ar@<.5ex>[l]^{\chi_1}}$ induce homomorphisms $\xymatrix{C_{P}(\Gamma_{i;p}\sqcup \Gamma_{j;q}) \ar@<.5ex>[rr]^{\xi\circ\chi_0\circ\xi^{-1}} && C_{P}(\Gamma_{i,j;p,q}) \ar@<.5ex>[ll]^{\xi\circ\chi_1\circ\xi^{-1}}}$. Note that
\begin{itemize}
	\item $\xi\circ\chi_0\circ\xi^{-1}$ and $\xi\circ\chi_1\circ\xi^{-1}$ are both homogeneous homomorphisms of degree $1$.
	\item Since $\chi_0$, $\chi_1$ are homotopically non-trivial and $\xi$ is an isomorphism, $\xi\circ\chi_0\circ\xi^{-1}$ and $\xi\circ\chi_1\circ\xi^{-1}$ are also homotopically non-trivial.
\end{itemize}
From the uniqueness part of Lemma \ref{lemma-def-chi}, one can see that, up to homotopy and scaling by non-zero scalars, $\xi\circ\chi_0\circ\xi^{-1}$ and $\xi\circ\chi_1\circ\xi^{-1}$ are the homomorphisms $\xymatrix{C_{P}(\Gamma_{i;p}\sqcup \Gamma_{j;q}) \ar@<.5ex>[r]^{\chi_0} & C_{P}(\Gamma_{i,j;p,q}) \ar@<.5ex>[l]^{\chi_1}}$ defined in Lemma \ref{lemma-def-chi}. Thus, the following diagrams commute up to homotopy and scaling by non-zero scalars.
\[
\xymatrix{
C_{P_N}(\Gamma_0) \ar[rr]^{\chi_0} \ar[d]^{\xi} && C_{P_N}(\Gamma_1)  \ar[d]^{\xi} \\
C_{P}(\Gamma_0) \ar[rr]^{\chi_0} && C_{P}(\Gamma_1) 
}
\hspace{2pc}
\xymatrix{
C_{P_N}(\Gamma_0)  \ar[d]^{\xi} && C_{P_N}(\Gamma_1) \ar[ll]_{\chi_1} \ar[d]^{\xi} \\
C_{P}(\Gamma_0)  && C_{P}(\Gamma_1) \ar[ll]_{\chi_1}
}
\]

The above shows that $C_{P_N}(D)$ and $C_{P}(D)$ are isomorphic as objects in the category of chain complexes over the homotopy category of graded matrix factorizations of $0$ over $\C[b_N]$. Thus, $H_{P_N}(L) \cong H_P(L)$ as $\zed^{\oplus 2}$-graded $\C[b_N]$-modules. Therefore, $H_{P_N}(L)$ contains no torsion components of the form $\C[a]/(a)\|i\|\{s\}$. By Corollary \ref{cor-d-(1)-tor-1}, $\delta_N=d_{P_N}^{(1)}=0$ on $H_N(L)$. This proves Part (2) of the corollary.
\end{proof}

We have proved Parts (1-3) of Theorem \ref{thm-delta-action} so far. It remains to prove Part (4). We start with the following corollary of Lemma \ref{lemma-quotient-isomorphism-P-1-P-2}.

\begin{corollary}\label{cor-P-truncation}
Suppose $b_i$ is a homogeneous variable of degree $2N+2-2i$ and 
\begin{eqnarray*}
P_1 & = & P_1(x,b_i) = x^{N+1} + \sum_{j=1}^{\left\lfloor \frac{N}{N+1-i} \right\rfloor} \lambda_{j} b_i^j x^{N+1-j(N+1-i)}, \\
P_2 & = & P_2(x,b_i) = x^{N+1} + \lambda_1 b_i x^i,
\end{eqnarray*}
where $\lambda_1,\dots,\lambda_{\left\lfloor \frac{N}{N+1-i} \right\rfloor} \in \C$. Then, for any link $L$,  $d_{P_1}^{(1)} = d_{P_2}^{(1)}$ on $H_N(L)$.
\end{corollary}

\begin{proof}
Fix a diagram $D$ of $L$ and a marking on $D$. Define $\mathscr{C}_{P_1,m}(D)$ and $\mathscr{C}_{P_2,m}(D)$ as in Lemma \ref{lemma-quotient-isomorphism-P-1-P-2}. According to that lemma, for $m=1,2$, $(\mathscr{C}_{P_1,m}(D), d_\chi)$ and $(\mathscr{C}_{P_2,m}(D), d_\chi)$ are identical as chain complexes of matrix factorizations of $0$ over $\C[b_i]/(b_i^m)$. In particular, note that identity, as an isomorphism between the above two chain complexes, is $\C[b_i]$-linear. Recall that $H(C_{P_i}(D),d_{mf})$ is a free $\C[b_i]$-module, $\mathscr{C}_{P_i,1}(D) \cong C_N(D)$ and $H(\mathscr{C}_{P_i,2}(D),d_{mf}) \cong H(C_{P_i}(D),d_{mf}) / a^2H(C_{P_i}(D),d_{mf})$. So we have the following commutative diagram with exact rows.
\[
\xymatrix{
0 \ar[r]& H(C_N(D),d_{mf}) \ar[r]^{b_i} \ar[d]^{\id} & H(\mathscr{C}_{P_1,2}(D),d_{mf}) \ar[r]^{\pi_{b_i}} \ar[d]^{\id} & H(C_N(D),d_{mf}) \ar[r] \ar[d]^{\id} & 0 \\
0 \ar[r]& H(C_N(D),d_{mf}) \ar[r]^{b_i}& H(\mathscr{C}_{P_2,2}(D),d_{mf}) \ar[r]^{\pi_{b_i}} & H(C_N(D),d_{mf}) \ar[r] & 0 
}
\]
By Lemma \ref{lemma-d-1-short}, this induces the following commutative diagram with exact rows.
\[
\xymatrix{
\cdots \ar[r]^{\pi_{b_i}}& H_N^j(L) \ar[r]^{d_{P_1}^{(1)}} \ar[d]^{\id} & H_N^{j+1}(L) \ar[r]^{b_i} \ar[d]^{\id} & \cdots \\
\cdots \ar[r]^{\pi_{b_i}}& H_N^j(L) \ar[r]^{d_{P_2}^{(1)}} & H_N^{j+1}(L) \ar[r]^{b_i} & \cdots 
}
\]
Thus, $d_{P_1}^{(1)} = d_{P_2}^{(1)}$ on $H_N(L)$.
\end{proof}

The following lemma concludes the proof of Part (4) of Theorem \ref{thm-delta-action}.

\begin{lemma}\label{lemma-delta-scaling}
Suppose $b_i$ is a homogeneous variable of degree $2N+2-2i$ and $\lambda$ a non-zero scalar. Define $P_i=x^{N+1} + b_i x^i$ and $\check{P}_i=x^{N+1} + \lambda b_i x^i$. Then, for any link $L$, $d_{\check{P}_i}^{(1)} = \lambda d_{P_i}^{(1)} = \lambda \delta_i$ on $H_N(L)$.
\end{lemma}

\begin{proof}
Fix a diagram $D$ of $L$ and a marking of $D$. Consider the ring automorphism $\zeta:\C[b_i] \rightarrow \C[b_i]$ given by $\zeta(b_i) = \lambda b_i$. Note that $\zeta$ induces on $\C[b_i]/(b_i)$ the identity automorphism $\C[b_i]/(b_i) \xrightarrow {\id} \C[b_i]/(b_i)$. For any MOY resolution $\Gamma$ of $D$, $\zeta$ induces an isomorphism $\zeta: C_{P_i}(\Gamma) \rightarrow C_{\check{P}_i}(\Gamma)$, which, in turn, induces a chain complex isomorphism $\zeta: C_{P_i}(D) \rightarrow C_{\check{P}_i}(D)$. Note that this chain complex isomorphism induces the identity chain map $C_{P_i}(D)/b_iC_{P_i}(D) \cong C_N(D) \xrightarrow {\id} C_N(D) \cong C_{\check{P}_i}(D)/b_iC_{\check{P}_i}(D)$. Using $\zeta$, we get the following commutative diagram with exact rows. 
\begin{equation}\label{diagram-delta-scaling-1}
\xymatrix{
0 \ar[r] & H(C_N(D),d_{mf}) \ar[r]^>>>>>{b_i} \ar[d]^{\id} & H(C_{P_i}(D)/b_i^2C_{P_i}(D),d_{mf}) \ar[r]^>>>>>{\pi_{b_i}} \ar[d]^{\zeta}_{\cong} & H(C_N(D),d_{mf}) \ar[r] \ar[d]^{\id} & 0 \\
0 \ar[r]& H(C_N(D),d_{mf}) \ar[r]^>>>>>{\lambda b_i} \ar[d]^{\lambda \id} & H(C_{\check{P}_i}(D)/b_i^2C_{\check{P}_i}(D),d_{mf}) \ar[r]^>>>>>{\pi_{b_i}} \ar[d]^{\id} & H(C_N(D),d_{mf}) \ar[r] \ar[d]^{\id} & 0 \\
0 \ar[r]& H(C_N(D),d_{mf}) \ar[r]^>>>>>{b_i} & H(C_{\check{P}_i}(D)/b_i^2C_{\check{P}_i}(D),d_{mf}) \ar[r]^>>>>>{\pi_{b_i}} & H(C_N(D),d_{mf}) \ar[r] & 0
}
\end{equation}
By Lemma \ref{lemma-d-1-short}, diagram \eqref{diagram-delta-scaling-1} induces a commutative diagram with exact rows
\[
\xymatrix{
\cdots \ar[r]^{\pi_{b_i}}& H_N^j(L) \ar[r]^{\delta_i} \ar[d]^{\id} & H_N^{j+1}(L) \ar[r]^{b_i} \ar[d]^{\id} & \cdots \\
\cdots \ar[r]^{\pi_{b_i}}& H_N^j(L) \ar[r]^{\Delta} \ar[d]^{\id} & H_N^{j+1}(L) \ar[r]^{\lambda b_i} \ar[d]^{\lambda\id} & \cdots \\
\cdots \ar[r]^{\pi_{b_i}}& H_N^j(L) \ar[r]^{d_{\check{P}_i}^{(1)}} & H_N^{j+1}(L) \ar[r]^{b_i} & \cdots 
},
\]
where $\Delta$ is the connecting homomorphism induced by the second row of diagram \eqref{diagram-delta-scaling-1}. Thus, we have $d_{\check{P}_i}^{(1)} = \lambda \Delta = \lambda \delta_i$.
\end{proof}

\subsection{A recapitulation of the proof of Theorem \ref{thm-delta-action}} The proof of Theorem \ref{thm-delta-action} is spread out in the first three subsections of this section. Here we give a quick recap of this proof.
\begin{itemize}
	\item Part (1) is proved in Corollary \ref{cor-delta-N=0}.
	\item Applying Lemma \ref{lemma-delta-anti-commute-general} to $P=x^{N+1}+b_i x^i + b_j x^j$, one gets Part (2).
	\item Part (3) is a special case of Lemma \ref{lemma-delta-natural}.
	\item Corollaries \ref{cor-delta-N=0}, \ref{cor-P-truncation} and Lemma \ref{lemma-delta-scaling} imply that, for a polynomial $P(x,a)=x^{N+1} + \sum_{i=1}^{\left\lfloor \frac{N}{k} \right\rfloor} \lambda_i a^i x^{N+1-ik}$ with $\deg a = 2k$ and $\lambda_i \in \C$, 
\[
d_P^{(1)} = \begin{cases}
0 & \text{if } \lambda_1 =0 \text{ or } k=1, \\
\lambda_1 \delta_{N+1-k} & \text{otherwise.}
\end{cases}
\] 
This proves Part (4).
\end{itemize}

\subsection{An example}\label{subsec-example-link} Next we compute $H_{P_i}(L)$ for the closed $2$-braid $L$ in Figure \ref{fig-link-example}, which allows us to conclude that, on $H_N(L)$, the differentials $\delta_1,\dots,\delta_{N-1}$ are non-zero, but $\delta_i\delta_j =0$ for any $1\leq i,j\leq N-1$.

\begin{figure}[ht]
\[
\xymatrix{
 \setlength{\unitlength}{1pt}
\begin{picture}(120,80)(-20,-10)

\qbezier(0,0)(5,0)(10,10)
\qbezier(10,10)(15,20)(20,20)

\qbezier(9,12)(0,30)(0,35)
\qbezier(11,8)(15,0)(20,0)

\qbezier(20,0)(25,0)(30,10)
\qbezier(30,10)(35,20)(40,20)

\qbezier(20,20)(25,20)(29,12)
\qbezier(31,8)(35,0)(40,0)

\qbezier(40,0)(45,0)(50,10)
\qbezier(50,10)(55,20)(60,20)

\qbezier(40,20)(45,20)(49,12)
\qbezier(51,8)(55,0)(60,0)

\qbezier(60,0)(65,0)(70,10)
\qbezier(70,10)(80,30)(80,35)

\qbezier(60,20)(65,20)(69,12)
\qbezier(71,8)(75,0)(80,0)

\qbezier(0,0)(-20,0)(-20,30)
\qbezier(-20,30)(-20,60)(40,60)
\qbezier(40,60)(100,60)(100,30)
\qbezier(100,30)(100,0)(80,0)

\qbezier(0,35)(0,50)(40,50)
\qbezier(80,35)(80,50)(40,50)

\put(40,60){\vector(-1,0){0}}

\put(40,50){\vector(-1,0){0}}

\put(40,59){\line(0,1){2}}

\put(40,49){\line(0,1){2}}

\put(40,63){$x_2$}

\put(40,40){$x_1$}

\put(35,-10){$L$}

\end{picture} & \setlength{\unitlength}{1pt}
\begin{picture}(120,80)(-20,-10)

\qbezier(0,0)(-20,0)(-20,30)
\qbezier(-20,30)(-20,60)(40,60)
\qbezier(0,35)(0,20)(20,20)

\qbezier(40,60)(100,60)(100,30)
\qbezier(100,30)(100,0)(80,0)
\qbezier(80,35)(80,20)(60,20)

\put(0,0){\line(1,0){80}}
\put(20,20){\line(1,0){40}}

\qbezier(0,35)(0,50)(40,50)
\qbezier(80,35)(80,50)(40,50)

\put(40,60){\vector(-1,0){0}}

\put(40,50){\vector(-1,0){0}}

\put(40,59){\line(0,1){2}}

\put(40,49){\line(0,1){2}}

\put(40,63){$x_2$}

\put(40,40){$x_1$}

\put(35,-10){$\Gamma_0$}

\end{picture} & \setlength{\unitlength}{1pt}
\begin{picture}(120,80)(-20,-10)

\qbezier(0,0)(-20,0)(-20,30)
\qbezier(-20,30)(-20,60)(40,60)
\qbezier(0,0)(5,0)(15,10)
\qbezier(0,35)(0,25)(15,10)

\qbezier(40,60)(100,60)(100,30)
\qbezier(100,30)(100,0)(80,0)
\qbezier(80,0)(75,0)(65,10)
\qbezier(80,35)(80,25)(65,10)

\qbezier(0,35)(0,50)(40,50)
\qbezier(80,35)(80,50)(40,50)

\put(40,60){\vector(-1,0){0}}

\put(40,50){\vector(-1,0){0}}

\put(40,59){\line(0,1){2}}

\put(40,49){\line(0,1){2}}

\put(40,63){$x_2$}

\put(40,40){$x_1$}

\put(35,-10){$\Gamma_1$}

\linethickness{3pt}

\put(15,10){\line(1,0){50}}

\end{picture}
}
\]

\caption{$L$ and two of its MOY resolutions}\label{fig-link-example-marked} 

\end{figure}
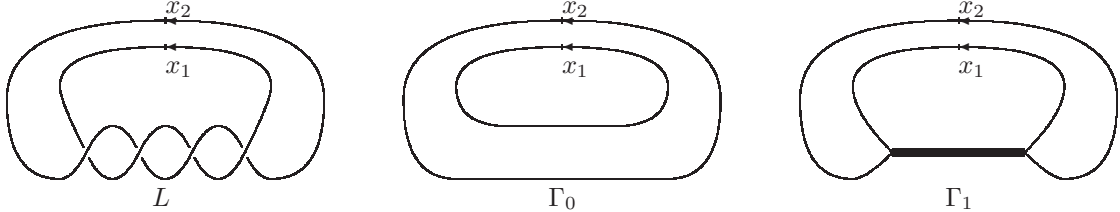

In our computation, we use the diagram of $L$ with two marked points in Figure \ref{fig-link-example-marked}. We also denote by $\Gamma_0$, $\Gamma_1$ the two MOY resolutions of $L$ in Figure \ref{fig-link-example-marked}. Before going any further, let us recall the Gaussian Elimination Lemma.

\begin{lemma}\cite[Lemma 4.2]{Bar-fast}\label{gaussian-elimination}
Let $\mathcal{C}$ be an additive category and
\[
\mathtt{I}=\cdots\rightarrow C\xrightarrow{\left(%
\begin{array}{c}
  \alpha\\
  \beta \\
\end{array}%
\right)}
\left.%
\begin{array}{c}
  A\\
  \oplus \\
  D
\end{array}%
\right.
\xrightarrow{
\left(%
\begin{array}{cc}
  \phi & \delta\\
  \gamma & \varepsilon \\
\end{array}%
\right)}
\left.%
\begin{array}{c}
  B\\
  \oplus \\
  E
\end{array}%
\right.
\xrightarrow{
\left(%
\begin{array}{cc}
  \mu & \nu\\
\end{array}%
\right)} F \rightarrow \cdots
\]
a chain complex over $\mathcal{C}$. Assume that $A\xrightarrow{\phi} B$ is an isomorphism in $\mathcal{C}$ with inverse $\phi^{-1}$. Then $\mathtt{I}$ is homotopic to 
\[
\mathtt{II}=
\cdots\rightarrow C \xrightarrow{\beta} D
\xrightarrow{\varepsilon-\gamma\phi^{-1}\delta} E\xrightarrow{\nu} F \rightarrow \cdots.
\]
\end{lemma}

By \cite[Theorem 1.1]{Wu-2braids}\footnote{Strictly speaking, \cite[Theorem 1.1]{Wu-2braids} is stated only for $P(x,a)=x^{N+1}-ax$. But it is straightforward to check that this theorem and its proof remain true for a general $P(x,a)$ of form \eqref{def-P}.}, for any polynomial $P=P(x,a)$ of form \eqref{def-P}, $C_P(L)$ is homotopic to the chain complex
{\small
\begin{equation}\label{complex-C-P-L}
0 \rightarrow C_P(\Gamma_0)\{-4N+4\} \xrightarrow{\chi_0} C_P(\Gamma_1)\{-4N+3\} \xrightarrow{0} C_P(\Gamma_1)\{-4N+1\} \xrightarrow{x_1-x_2} C_P(\Gamma_1)\{-4N-1\} \xrightarrow{0} C_P(\Gamma_1)\{-4N-3\} \rightarrow 0,
\end{equation}}

\noindent where $C_P(\Gamma_0)\{-4N+4\}$ is at homological degree $0$ and $\chi_0$ is the homomorphism associated to wide edge in $\Gamma_1$. Thus, 
\begin{equation}\label{complex-C-P-L-decomp}
C_P(L)  \simeq  \mathbf{C}_1 \oplus \mathbf{C}_2 \oplus \mathbf{C}_3,
\end{equation}
where
\begin{eqnarray*}
\mathbf{C}_1 & = & 0 \rightarrow C_P(\Gamma_0)\|0\|\{-4N+4\} \xrightarrow{\chi_0} C_P(\Gamma_1)\|1\|\{-4N+3\} \rightarrow 0,\\
\mathbf{C}_2 & = & 0 \rightarrow C_P(\Gamma_1)\|2\|\{-4N+1\} \xrightarrow{x_1-x_2} C_P(\Gamma_1)\|3\|\{-4N-1\} \rightarrow 0, \\
\mathbf{C}_3 & = & 0 \rightarrow C_P(\Gamma_1)\|4\|\{-4N-3\} \rightarrow 0.
\end{eqnarray*}

\begin{figure}[ht]
\[
\xymatrix{
 \setlength{\unitlength}{1pt}
\begin{picture}(120,80)(-20,-10)

\qbezier(0,0)(5,0)(10,10)
\qbezier(10,10)(15,20)(20,20)

\qbezier(9,12)(0,30)(0,35)
\qbezier(11,8)(15,0)(20,0)

\qbezier(0,0)(-20,0)(-20,30)
\qbezier(-20,30)(-20,60)(40,60)
\qbezier(40,60)(100,60)(100,30)
\qbezier(100,30)(100,0)(20,0)
\qbezier(80,35)(80,20)(20,20)

\qbezier(0,35)(0,50)(40,50)
\qbezier(80,35)(80,50)(40,50)

\put(40,60){\vector(-1,0){0}}

\put(40,50){\vector(-1,0){0}}

\put(40,59){\line(0,1){2}}

\put(40,49){\line(0,1){2}}

\put(40,63){$x_2$}

\put(40,40){$x_1$}

\put(35,-10){$U_{-1}$}

\end{picture} && \setlength{\unitlength}{1pt}
\begin{picture}(120,80)(-20,-10)

\qbezier(0,0)(-20,0)(-20,30)
\qbezier(-20,30)(-20,60)(40,60)

\qbezier(40,60)(100,60)(100,30)
\qbezier(100,30)(100,0)(80,0)

\put(0,0){\line(1,0){80}}

\put(40,60){\vector(-1,0){0}}

\put(40,59){\line(0,1){2}}

\put(40,63){$x_2$}

\put(35,-10){$U_0$}

\end{picture} 
}
\]

\caption{Two diagrams of the unknot}\label{fig-unknots} 

\end{figure}
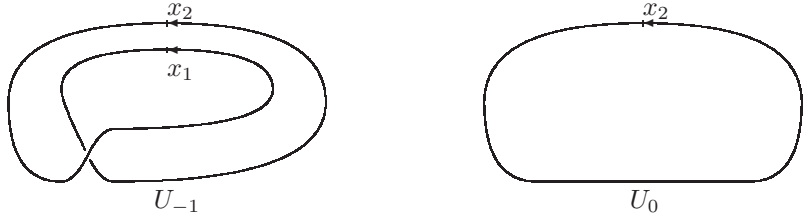

Consider the two diagrams of the unknot in Figure \ref{fig-unknots}. Note that $\mathbf{C}_1 = C_P(U_{-1}) \{-3N+3\}$. So
\begin{equation}\label{homology-C-1}
H(H(\mathbf{C}_1,d_{mf}),d_\chi) = H_P(U_{-1})\{-3N+3\} \cong H_P(U_{0})\{-3N+3\} \cong \C[x_2,a]/(\frac{\partial P(x_2,a)}{\partial x_2})\|0\|\{-4N+4\}.
\end{equation}
In particular, note that $H(H(\mathbf{C}_1,d_{mf}),d_\chi)$ is a free $\C[a]$-module. 

By \cite[Proposition 10]{Krasner}, $\mathbf{C}_3 \simeq \bigoplus_{i=0}^{N-2}C_P(U_0)\|4\|\{-3N-5-2i\}$.\footnote{We are not tracking the $\zed_2$-grading here.} So 
\begin{equation}\label{homology-C-3}
H(H(\mathbf{C}_3,d_{mf}),d_\chi) \cong  \bigoplus_{i=0}^{N-2}H_P(U_0)\|4\|\{-3N-5-2i\} \cong \bigoplus_{i=0}^{N-2} \C[x_2,a]/(\frac{\partial P(x_2,a)}{\partial x_2}) \|4\|\{-4N-4-2i\}.
\end{equation}
Note that $H(H(\mathbf{C}_3,d_{mf}),d_\chi)$ is again a free $\C[a]$-module. 

It remains to compute $H(H(\mathbf{C}_2,d_{mf}),d_\chi)$. Write $P(x,a) = \sum_{i=1}^{N+1} f_i x^i$, where $f_{N+1}=1$ and each $f_i$ is a monomial of $a$ of degree $2N+2-2i$. (For degree reasons, many of these $f_i$'s vanish.) By \cite[the proof of Lemma 2.18]{Wu7}, as an endomorphism of $C_P(\Gamma_0)$,
\begin{equation}\label{eq-pd-vanish}
\mathsf{m}(\sum_{i=0}^{N} (i+1)f_{i+1} x_1^i) = \mathsf{m}(\frac{\partial P(x_1,a)}{\partial x_1}) \simeq 0,
\end{equation}
where $\mathsf{m}(\ast)$ is the endomorphism given by the multiplication by $\ast$.

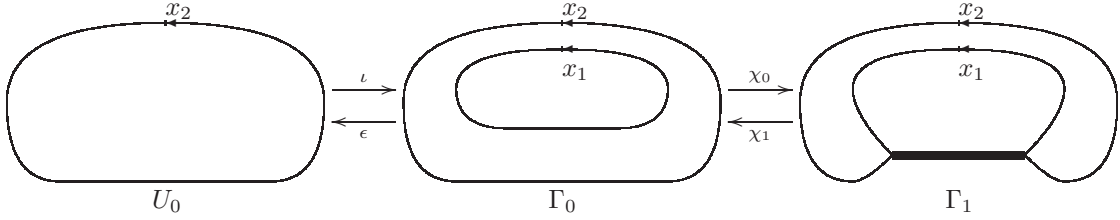
\begin{figure}[ht]
\[
\xymatrix{
\setlength{\unitlength}{1pt}
\begin{picture}(120,80)(-20,-10)

\qbezier(0,0)(-20,0)(-20,30)
\qbezier(-20,30)(-20,60)(40,60)

\qbezier(40,60)(100,60)(100,30)
\qbezier(100,30)(100,0)(80,0)

\put(0,0){\line(1,0){80}}

\put(40,60){\vector(-1,0){0}}

\put(40,59){\line(0,1){2}}

\put(40,63){$x_2$}

\put(35,-10){$U_0$}

\end{picture} \ar@<3.5pc>[r]^{\iota} & \setlength{\unitlength}{1pt}
\begin{picture}(120,80)(-20,-10)

\qbezier(0,0)(-20,0)(-20,30)
\qbezier(-20,30)(-20,60)(40,60)
\qbezier(0,35)(0,20)(20,20)

\qbezier(40,60)(100,60)(100,30)
\qbezier(100,30)(100,0)(80,0)
\qbezier(80,35)(80,20)(60,20)

\put(0,0){\line(1,0){80}}
\put(20,20){\line(1,0){40}}

\qbezier(0,35)(0,50)(40,50)
\qbezier(80,35)(80,50)(40,50)

\put(40,60){\vector(-1,0){0}}

\put(40,50){\vector(-1,0){0}}

\put(40,59){\line(0,1){2}}

\put(40,49){\line(0,1){2}}

\put(40,63){$x_2$}

\put(40,40){$x_1$}

\put(35,-10){$\Gamma_0$}

\end{picture} \ar@<3.5pc>[r]^{\chi_0} \ar@<-2.5pc>[l]^{\epsilon} & \setlength{\unitlength}{1pt}
\begin{picture}(120,80)(-20,-10)

\qbezier(0,0)(-20,0)(-20,30)
\qbezier(-20,30)(-20,60)(40,60)
\qbezier(0,0)(5,0)(15,10)
\qbezier(0,35)(0,25)(15,10)

\qbezier(40,60)(100,60)(100,30)
\qbezier(100,30)(100,0)(80,0)
\qbezier(80,0)(75,0)(65,10)
\qbezier(80,35)(80,25)(65,10)

\qbezier(0,35)(0,50)(40,50)
\qbezier(80,35)(80,50)(40,50)

\put(40,60){\vector(-1,0){0}}

\put(40,50){\vector(-1,0){0}}

\put(40,59){\line(0,1){2}}

\put(40,49){\line(0,1){2}}

\put(40,63){$x_2$}

\put(40,40){$x_1$}

\put(35,-10){$\Gamma_1$}

\linethickness{3pt}

\put(15,10){\line(1,0){50}}

\end{picture} \ar@<-2.5pc>[l]^{\chi_1}
}
\]

\caption{Definition of $\alpha$ and $\beta$}\label{fig-def-alpha-beta} 

\end{figure}

Next, we explicitly write down the inclusions and projections in the decomposition 
\begin{equation}\label{decomp-MOY-I}
C_P(\Gamma_1) \simeq \bigoplus_{i=0}^{N-2}C_P(U_0)\{N-2-2i\}.
\end{equation} 
Consider the homomorphisms in Figure \ref{fig-def-alpha-beta}, where
\begin{itemize}
	\item $\iota$ and $\epsilon$ are the homomorphisms associated to the circle creation and annihilation (see \cite{Krasner} for their definitions,)
	\item $\chi_0$ and $\chi_1$ are the $\chi$-maps associated to the wide edge in $\Gamma_1$.
\end{itemize}
Recall that 
\begin{enumerate}
	\item $\iota$ and $\epsilon$ are homogeneous of degree $-N+1$ and $\C[x_2,a]$-linear. For $1\leq i \leq N-1$, 
	\begin{equation}\label{eq-comp-e-i}
	\epsilon \circ \mathsf{m}(x_1^i) \circ \iota \simeq \begin{cases}
	\id_{C_P(U_0)} & \text{if } i=N-1, \\
	0 & \text{if }, i=0,1,\dots, N-2.
	\end{cases}
	\end{equation}
	\item $\chi_0$ and $\chi_1$ are homogeneous of degree $1$ and $\C[x_1,x_2,a]$-linear. $\chi_1 \circ \chi_0 \simeq (x_1-x_2) \id_{C_P(\Gamma_0)}$.
\end{enumerate}

Define $\alpha=\chi_0 \circ \iota$ and $\beta = \epsilon \circ \chi_1$. Note that these are homogeneous homomorphisms of degree $-N+2$. For $i=0,1,\dots,N-2$, define $\alpha_i: C_P(U_0)\{2+2i-N\} \rightarrow C_P(\Gamma_1)$ and $\beta_i: C_P(\Gamma_1) \rightarrow C_P(U_0)\{2+2i-N\}$ by $\alpha_i = \mathsf{m}(\sum_{p=0}^i \sum_{l=0}^{i-p} (N-p+1)f_{N-p+1} x_1^l x_2^{i-p-l}) \circ \alpha$ and $\beta_i= \frac{1}{N+1}\beta \circ \mathsf{m}(x_1^{N-i-2})$. Note that $\alpha_i$ and $\beta_i$ are homogeneous homomorphisms of degree $0$. For any $0 \leq i,j \leq N-2$,
\begin{eqnarray*}
\beta_j \circ \alpha_i & = & \frac{1}{N+1}\beta \circ \mathsf{m}(x_1^{N-j-2}\sum_{p=0}^i \sum_{l=0}^{i-p} (N-p+1)f_{N-p+1} x_1^l x_2^{i-p-l}) \circ \alpha \\
& = & \frac{1}{N+1}\epsilon \circ \chi_1 \circ \mathsf{m}(x_1^{N-j-2}\sum_{p=0}^i \sum_{l=0}^{i-p} (N-p+1)f_{N-p+1} x_1^l x_2^{i-p-l}) \circ \chi_0 \circ \iota \\
& = & \frac{1}{N+1}\epsilon  \circ \mathsf{m}((x_1-x_2)x_1^{N-j-2}\sum_{p=0}^i \sum_{l=0}^{i-p} (N-p+1)f_{N-p+1} x_1^l x_2^{i-p-l}) \circ \iota \\
& = & \frac{1}{N+1}\sum_{p=0}^i (N-p+1)f_{N-p+1} \cdot\epsilon  \circ \mathsf{m}(x_1^{N-j-2} (x_1-x_2)\sum_{l=0}^{i-p} x_1^l x_2^{i-p-l}) \circ \iota \\
& = & \frac{1}{N+1}\sum_{p=0}^i (N-p+1)f_{N-p+1} \cdot\epsilon  \circ \mathsf{m}(x_1^{N-j-2} (x_1^{i-p+1}-x_2^{i-p+1})) \circ \iota \\
& = & \frac{1}{N+1}\sum_{p=0}^i (N-p+1)f_{N-p+1} \cdot(\epsilon  \circ \mathsf{m}(x_1^{N+i-j-p-1}) \circ \iota - \epsilon  \circ \mathsf{m}(x_1^{N-j-2}x_2^{i-p+1}) \circ \iota) \\
\text{[by \eqref{eq-comp-e-i}]} & \simeq & \sum_{p=0}^i (N-p+1)f_{N-p+1} \cdot \epsilon  \circ \mathsf{m}(x_1^{N+i-j-p-1}) \circ \iota
\end{eqnarray*}

If $j>i$, then $N+i-j-p-1 \leq N-2$ for $p=0,\dots,i$. So, by \eqref{eq-comp-e-i}, $\beta_j \circ \alpha_i \simeq 0$ in this case. 

If $j<i$ then, by \eqref{eq-pd-vanish}, 
\begin{eqnarray*}
\beta_j \circ \alpha_i & \simeq & \frac{1}{N+1}\epsilon  \circ \mathsf{m}(x_1^{i-j-1}\sum_{p=0}^i (N-p+1)f_{N-p+1} x_1^{N-p}) \circ \iota \\
&  \simeq & -\frac{1}{N+1}\epsilon  \circ \mathsf{m}(x_1^{i-j-1}\sum_{p=i+1}^N (N-p+1)f_{N-p+1} x_1^{N-p}) \circ \iota \\
&  \simeq & -\frac{1}{N+1}\sum_{p=i+1}^N (N-p+1)f_{N-p+1} \epsilon  \circ \mathsf{m}(x_1^{N-p+i-j-1}) \circ \iota \\
& \simeq & 0,
\end{eqnarray*}
where, in the last step, we used \eqref{eq-comp-e-i} and that $N+i-j-p-1 \leq N-2$ for $p=i+1,\dots,N$.

If $i=j$, then, by \eqref{eq-comp-e-i} and that $f_{N+1}=1$, we have
\[
\beta_i \circ \alpha_i \simeq \frac{1}{N+1}\sum_{p=0}^i (N-p+1)f_{N-p+1} \cdot \epsilon  \circ \mathsf{m}(x_1^{N-p-1}) \circ \iota \simeq \cdot \epsilon  \circ \mathsf{m}(x_1^{N-1}) \circ \iota \simeq \id_{C_P(U_0)}.
\]

Altogether, we get that, for $0 \leq i,j \leq N-2$,
\begin{equation}\label{eq-a-b-comp}
\beta_j \circ \alpha_i  \simeq \begin{cases}
\id_{C_P(U_0)} & \text{if } i=j,\\
0 & \text{otherwise.}
\end{cases}
\end{equation}
Thus, we can use $\alpha_i$ and $\beta_j$ as the inclusions and projections in decomposition \eqref{decomp-MOY-I}. 

Recall that the differential map of $\mathbf{C}_2$ is the endomorphism $\mathsf{m}(x_1-x_2)$ of $C_P(\Gamma_1)$. Its action on the components of $C_P(\Gamma_1)$ in decomposition \eqref{decomp-MOY-I} is given by 
\begin{equation}
\beta_j \circ \mathsf{m}(x_1-x_2) \circ \alpha_i  = \beta_j \circ \mathsf{m}(x_1) \circ \alpha_i - \beta_j \circ \mathsf{m}(x_2) \circ \alpha_i = \beta_{j-1} \circ \alpha_i - x_2\cdot \beta_j \circ \alpha_i \simeq \begin{cases}
-x_2\cdot \id_{C_P(U_0)} & \text{if } i=j,\\
\id_{C_P(U_0)} & \text{if } i=j-1,\\
0 & \text{otherwise.}
\end{cases}
\end{equation}
Thus, using decomposition \eqref{decomp-MOY-I}, we have 
\[
\mathbf{C}_2 \cong 0 \rightarrow \left[%
\begin{array}{c}
  C_P(U_0)\{2-N\}\\
  \oplus \\
  C_P(U_0)\{4-N\} \\
  \oplus \\
  \vdots \\
  \oplus \\
  C_P(U_0)\{N-4\} \\
  \oplus \\
  C_P(U_0)\{N-2\} \\
\end{array}%
\right]\|2\|\{-4N+1\}
\xrightarrow{D_{N-1}}
\left[%
\begin{array}{c}
  C_P(U_0)\{4-N\}\\
  \oplus \\
  C_P(U_0)\{6-N\} \\
  \oplus \\
  \vdots \\
  \oplus \\
  C_P(U_0)\{N-2\} \\
  \oplus \\
  C_P(U_0)\{2-N\} \\
\end{array}%
\right]\|3\|\{-4N-1\}
\rightarrow 0,
\]
where the differential $D_{N-1}$ is the $(N-1) \times (N-1)$ matrix
\[
D_{N-1} = \left(%
\begin{array}{cccccc}
1 & -x_2 & 0 & \cdots & 0 & 0 \\
0 & 1 & -x_2 & \cdots & 0 & 0 \\
0 & 0 & 1 & \cdots & 0 & 0 \\
\cdots & \cdots &\cdots &\cdots &\cdots &\cdots  \\
0 & 0 & 0 & \cdots & -x_2 & 0 \\
0 & 0 & 0 & \cdots & 1 & -x_2 \\
-x_2 & 0 & 0 & \cdots & 0 & 0 \\
\end{array}%
\right).
\]
Here note the difference in the ordering of components in the two columns in the chain complex. 

Now apply Gaussian elimination (Lemma \ref{gaussian-elimination}) to the ``$1$" at the upper left corner of $D_{N-1}$. We get that 
\[
\mathbf{C}_2 \simeq 0 \rightarrow \left[%
\begin{array}{c}
  C_P(U_0)\{4-N\} \\
  \oplus \\
  \vdots \\
  \oplus \\
  C_P(U_0)\{N-4\} \\
  \oplus \\
  C_P(U_0)\{N-2\} \\
\end{array}%
\right]\|2\|\{-4N+1\}
\xrightarrow{D_{N-2}}
\left[%
\begin{array}{c}
  C_P(U_0)\{6-N\} \\
  \oplus \\
  \vdots \\
  \oplus \\
  C_P(U_0)\{N-2\} \\
  \oplus \\
  C_P(U_0)\{2-N\} \\
\end{array}%
\right]\|3\|\{-4N-1\}
\rightarrow 0,
\]
where the differential $D_{N-2}$ is the $(N-2) \times (N-2)$ matrix
\[
D_{N-2} = \left(%
\begin{array}{cccccc}
1 & -x_2 & 0 & \cdots & 0 & 0 \\
0 & 1 & -x_2 & \cdots & 0 & 0 \\
0 & 0 & 1 & \cdots & 0 & 0 \\
\cdots & \cdots &\cdots &\cdots &\cdots &\cdots  \\
0 & 0 & 0 & \cdots & -x_2 & 0 \\
0 & 0 & 0 & \cdots & 1 & -x_2 \\
-x_2^2 & 0 & 0 & \cdots & 0 & 0 \\
\end{array}%
\right).
\]
Clearly, we can apply Gaussian elimination to the ``$1$" at the upper left corner of $D_{N-1}$ again and again. After $N-2$ Gaussian eliminations, we get that
\begin{equation}\label{complex-C-2-simplified}
\mathbf{C}_2 \simeq 0 \rightarrow   C_P(U_0) \|2\|\{-3N-1\} \xrightarrow{-x_2^{N-1}}  C_P(U_0)\|3\|\{-5N+1\} \rightarrow 0.
\end{equation}

From now on, we specialize to the case $P=P_i(x,b_i) = x^{N+1} + b_i x^i$, where $\deg b_i = 2N+2-2i$. In this case, $C_{P_i}(U_0) \cong M \|0\|\{-N+1\}$ and 
\[
H(H(\mathbf{C}_2,d_{mf}), d_\chi) \cong H( 0 \rightarrow   M \|2\|\{-4N\} \xrightarrow{-x_2^{N-1}} M \|3\|\{-6N+2\} \rightarrow 0),
\]
where $M$ is the graded free $\C[b_i]$-module 
\begin{equation}\label{def-M}
M=\C[x_2,b_i]/((N+1)x_2^N +ib_ix_2^{i-1}) = \bigoplus_{j=0}^{N-1} \C[b_i] \cdot x_2^{j} \cong \bigoplus_{j=0}^{N-1} \C[b_i]\{2j\}.
\end{equation}
It is straightforward to check that 
\begin{eqnarray}
\label{homology-C-2-2} H^2(H(\mathbf{C}_2,d_{mf}),d_\chi) & \cong  & \ker (\mathsf{m}(x_2^{N-1}))\{-4N\} \\
\nonumber & = & \bigoplus_{j=0}^{i-2} \C[b_i] \cdot ((N+1)x_2^{N-i+j+1} +ib_i x_2^j)\{-4N\} \\ 
\nonumber & \cong & \bigoplus_{j=0}^{i-2} \C[b_i] \{2(-N-i+j+1)\}.
\end{eqnarray}
and
\begin{eqnarray}
\label{homology-C-2-3} H^3(H(\mathbf{C}_2,d_{mf}),d_\chi) & \cong & \mathrm{coker} (\mathsf{m}(x_2^{N-1}))\{-6N+2\} \\
\nonumber & = & (\bigoplus_{j=0}^{i-2} \C[b_i] \cdot x_2^j)\{-6N+2\} \oplus (\bigoplus_{j=i-1}^{N-2} \C[b_i]/(b_i) \cdot x_2^j)\{-6N+2\} \\
\nonumber & \cong & (\bigoplus_{j=0}^{i-2} \C[b_i] \{-6N+2j+2\}) \oplus (\bigoplus_{j=i-1}^{N-2} \C[b_i]/(b_i) \{-6N+2j+2\}).
\end{eqnarray}

Combining \eqref{complex-C-P-L-decomp}, \eqref{homology-C-1}, \eqref{homology-C-3}, \eqref{def-M}, \eqref{homology-C-2-2} and \eqref{homology-C-2-3}, we get the following lemma.

\begin{lemma}\label{lemma-H-P-i-L}
Let $L$ be the closed $2$-braid in Figure \ref{fig-link-example} and $P_i(x,b_i) = x^{N+1} + b_i x^i$, where $1\leq i \leq N-1$ and $\deg b_i = 2N+2-2i$. Then
\begin{eqnarray*}
H_{P_i}^0(L) & \cong & \bigoplus_{j=0}^{N-1} \C[b_i]\{-4N+4+2j\}, \\
H_{P_i}^1(L) & \cong & 0, \\
H_{P_i}^2(L) & \cong & \bigoplus_{j=0}^{i-2} \C[b_i] \{2(-N-i+j+1)\}, \\
H_{P_i}^3(L) & \cong & (\bigoplus_{j=0}^{i-2} \C[b_i] \{-6N+2j+2\}) \oplus (\bigoplus_{j=i-1}^{N-2} \C[b_i]/(b_i) \{-6N+2j+2\}), \\
H_{P_i}^4(L) & \cong & \bigoplus_{l=0}^{N-2} \bigoplus_{j=0}^{N-1} \C[b_i] \{-4N-4-2l+2j\}, \\
H_{P_i}^l(L) & \cong & 0 \text{ if } l<0 \text{ or } l>4.
\end{eqnarray*}
\end{lemma}

\begin{corollary}\label{cor-delta-i-j-L}
Let $L$ be the closed $2$-braid in Figure \ref{fig-link-example}. Then, for any $1\leq i \leq N-1$, we have
\[
\delta_i|_{H_{N}^l(L)} \begin{cases}
\neq 0 & \text{if } l=2, \\
=0 & \text{if } l \neq 2.
\end{cases}
\]
In particular, as endomorphisms of $H_N(L)$, $\delta_i \neq 0$, but $\delta_i\delta_j =0$ for any $1\leq i,j \leq N-1$.
\end{corollary}

\begin{proof}
By Lemma \ref{lemma-H-P-i-L}, all torsion components of $H_{P_i}(L)$ are isomorphic to $\C[b_i]/(b_i)$ and are at homological degree $3$. The corollary follows from this and Lemma \ref{lemma-d-(1)-tor}.
\end{proof}

\end{document}